\documentclass{amsart}
\usepackage{amsmath,amssymb,amsthm}
\usepackage{datetime}

\usepackage[margin=1in]{geometry}
\usepackage[colorlinks,linkcolor = blue,citecolor = blue]{hyperref}

\newtheorem{theorem}{Theorem}
\newtheorem{proposition}{Proposition}[section]
\newtheorem{corollary}{Corollary}
\newtheorem{lemma}{Lemma}[section]
\theoremstyle{definition}
\newtheorem{definition}{Definition}[section]
\newtheorem{remark}{Remark}
\theoremstyle{definition}

\newtheorem{assumption}{Assumption}

\numberwithin{equation}{section}

\DeclareMathOperator{\R}{\mathbb{R}}

\DeclareMathOperator{\V}{\mathcal{V}}
\DeclareMathOperator{\F}{\mathcal{F}}
\let\S\relax
\DeclareMathOperator{\S}{\mathcal{S}}
\DeclareMathOperator{\bd}{\partial}
\DeclareMathOperator{\bigO}{\mathcal{O}}

\DeclareMathOperator{\tens}{\otimes}

\usepackage{xcolor}
\usepackage{comment}
\usepackage[normalem]{ulem}
\usepackage{bbm}
\usepackage{commath}
\usepackage{mathtools}

\newcommand{\red}[1]{\textcolor{red}{#1}}

\newcommand{\RR}{\mathbb{R}}
\newcommand{\eps}{\varepsilon}
\newcommand{\Sweak}{\mathcal{S}_{weak}}
\newcommand{\Nuo}{\mathcal{V}_0}
\newcommand{\Nu}{\mathcal{V}}
\newcommand{\ds}{\displaystyle}
\newcommand{\dt}{\partial_t}
\newcommand{\dx}{\partial_x}

\usepackage{graphicx,marvosym}
\usepackage{cleveref}

\title{Sharp $a$-contraction estimates for small extremal shocks}

\author[Golding]{William M. Golding}
\address[William M. Golding]{\newline Department of Mathematics, \newline The University of Texas at Austin, Austin, TX 78712, USA}
\email{wgolding@utexas.edu}

\author[Krupa]{Sam G.  Krupa}
\address[Sam G. Krupa]{\newline Max Planck Institute for Mathematics in the Sciences, Inselstr. 22, 04103 Leipzig, Germany}
\email{Sam.Krupa@mis.mpg.de}

\author[Vasseur]{Alexis F. Vasseur}
\address[Alexis F. Vasseur]{\newline Department of Mathematics, \newline The University of Texas at Austin, Austin, TX 78712, USA}
\email{vasseur@math.utexas.edu}

\date{\today}

\thanks{\textbf{Acknowledgment.} 
A. Vasseur is partially supported by the NSF grant: DMS 1614918. W. Golding and S. Krupa were partially supported by  NSF-DMS Grant 1840314.
}

\begin{document}
\maketitle

\begin{abstract}
In this paper, we study the $a$-contraction property of small extremal shocks for 1-d systems of hyperbolic conservation laws endowed with a single convex entropy, when subjected to large perturbations. We show that the weight coefficient $a$ can be chosen with amplitude proportional to  the size of the shock. The main result of this paper is a key building block in the companion paper, [{arXiv:2010.04761}, 2020], in which uniqueness and BV-weak stability results for $2\times 2$ systems of hyperbolic conservation laws are proved. 
\end{abstract}

\tableofcontents

\section{Introduction}

We consider  1-d hyperbolic systems of conservation laws with $n$ unknowns
\begin{align}\label{cl}
u_t+(f(u))_x=0  \qquad t>0, \ x\in \RR,
\end{align}
where $(t,x)\in {\mathbb R}^+\times \mathbb R$ are space and time, and $u=(u_1,\ldots,u_n)\in \Nuo\subseteq \RR^n$ are the unknowns. We assume the set of states $\Nuo$ is convex and bounded, and denote $\Nu$ its interior. Then, the flux function, $f=(f_1,\cdot\cdot\cdot, f_n)\in [C(\Nuo)]^n\cap [C^4(\Nu)]^n$, is assumed to be continuous on $\Nuo$ and $C^4$ on $\Nu$.
We assume that there exists a strictly convex entropy functional $\eta\in C(\Nuo)\cap C^3(\Nu)$ and an associated entropy flux functional $q\in C(\Nuo)\cap C^3(\Nu)$ such that
\begin{equation}\label{ineq:entropy}
q'=\eta'f', \qquad \mathrm{on} \ \ \Nu.
\end{equation}
We consider only entropic solutions to \eqref{cl}. That is, we consider solutions verifying
\begin{align}\label{entropy}
(\eta(u))_t+(q(u))_x\leq 0  \qquad t>0, \ x\in \RR,
\end{align}
in the sense of distributions. Throughout the paper, we use a single entropy $\eta$. Therefore, we do \emph{not} assume that (\ref{entropy}) holds for all entropies.
We describe the precise set of hypotheses on $f$ in Assumption \ref{assum}. The assumptions on $f$ are fairly general.
\vskip.3cm
Before describing the result in detail, we note that all Assumptions \ref{assum} are verified by both the system of isentropic Euler equations  for  $\gamma>1$, and the full Euler system (see for instance \cite{Leger2011}, or \cite{CKV1}). Therefore, the main result of this paper applies in both cases. 
We recall that the isentropic Euler equation is given by
\begin{equation}\label{eq:Euler}
\begin{array}{l}
\rho_t+(\rho v)_x=0, \qquad t>0, x\in \RR,\\
(\rho v)_t+(\rho v^2+\rho^\gamma)_x=0, \qquad t>0, x\in\RR,
\end{array}
\end{equation}
and is endowed with the physical entropy $\eta(u)=\rho v^2/2+\rho^{{\gamma}}/(\gamma-1)$, where $u=(\rho,\rho v)$. For any fixed constant $C>0$, we can define the space of states as the  invariant region:
\begin{equation}\label{InvReg}
\Nuo=\{ u=(\rho, \rho v)\in \RR^+\times\RR: \ \ -C< w_1(u)=v-c_1\rho^\frac{\gamma-1}{2} \leq  w_2(u)=v+c_1\rho^\frac{\gamma-1}{2}<C \}.
\end{equation}
Note that $\Nuo=\Nu \cup \{(0,0)\}$ where $(0,0)$ is the vacuum state, which justifies the precise distinction of $\Nu$ and $\Nuo$  (see  \cite{VASSEUR2008323, Leger2011}).
\vskip0.3cm
In the case of the full Euler equation, we have 
 \begin{equation}\label{eq:fullEuler}
\left\{\begin{array}{l}
\ds{\dt \rho+\dx (\rho u)=0}\\[0.3cm]
\ds{\dt (\rho u )+\dx (\rho u^2+ P)=0}\\[0.3cm]
\ds{\dt (\rho E )+\dx (\rho u E + uP)=0,}
\end{array}\right.
\end{equation}
where $E = \frac{1}{2}u^2 + e$. The equation of state for a polytropic gas is given by 
\begin{equation}\label{polytropic}
P = (\gamma-1) \rho e
\end{equation}
where $\gamma >1$. We consider the entropy/entropy-flux pair
\begin{equation}\label{Eulerconvex}
\eta(\rho, \rho u, \rho E) = (\gamma -1) \rho \ln \rho - \rho \ln e, \qquad G(\rho, \rho u, \rho E) = (\gamma -1) \rho u \ln \rho - \rho u \ln e,
\end{equation}
where, in conservative variables, we have $e = \ds{\frac{\rho E}{\rho} - \frac{(\rho u)^2}{2 \rho^2}}$. We can also define a state domain which includes the vacuum as, for instance,
$$
\Nuo=\{ u=(\rho, \rho v, \rho E)\in \RR^+\times\RR\times \RR: \ \ \rho+|v|+|E|\leq C \}.
$$
In this work, we further restrict our attention to solutions verifying the so-called Strong Trace Property.
\begin{definition}[Strong Trace Property]\label{defi_trace}
Let $u\in L^\infty(\RR^+\times\RR)$. We say that $u$ verifies the strong trace property if for any Lipschitzian curve $t\to h(t)$, there exists two bounded functions $u_-,u_+\in L^\infty(\RR^+)$ such that for any $T>0$
$$
\lim_{n\to\infty}\int_0^T\sup_{y\in(0,1/n)}|u(t,h(t)+y)-u_+(t)|\,dt=\lim_{n\to\infty}\int_0^T\sup_{y\in(-1/n,0)}|u(t,h(t)+y)-u_-(t)|\,dt=0.
$$
\end{definition}
For convenience, we will use the notation $u_+(t)=u(t, h(t)+)$, and $u_-(t)=u(t, h(t)-)$.
We can now precisely define the space of weak solutions  considered in this paper:
\begin{equation}\label{eq:weaksol}
\Sweak=\{u\in L^\infty(\RR^+\times \RR:\Nuo) \ \ \mathrm{weak \ solution \ to } \ \eqref{cl}\eqref{ineq:entropy}, \ \ \text{verifying Definition \ref{defi_trace}}  \}.
\end{equation}
Note that this space has no smallness condition. Since the Strong Trace Property is weaker than $BV$, any $BV$ entropic solution to  \eqref{cl}\eqref{ineq:entropy} belongs to $\Sweak$.
\vskip.3cm
Glimm showed in  \cite{gl} the existence of global in time, small BV, entropic solutions to   \eqref{cl} \eqref{ineq:entropy}. Bressan and Goatin \cite{MR1701818} established uniqueness of these solutions under the Tame Oscillation Condition, improving an earlier theorem of Bressan and LeFloch \cite{MR1489317}.
Separately, uniqueness was also shown when the Tame Oscillation Condition is replaced by the assumption that the trace of solutions along space-like curves has bounded variation, see \cite{MR1757395}. Although our main theorem, Theorem \ref{thm_main}, is stated for $n\times n$ systems, the companion paper \cite{CKV1} provides an important application of it in the case of $2\times 2$ systems. Together, this paper and \cite{CKV1} improve the known uniqueness result for small $BV$ solutions by removing the a priori assumptions mentioned above. More precisely, $BV$ solutions are stable in a larger class of weak solutions, the class $\Sweak$ defined above, providing a $BV$/weak stability result.
\vskip0.3cm
The main result of this paper is a precise estimate on the weight function $a$ needed to obtain a weighted $L^2$-contraction property for small extremal shocks (known as the $a$-contraction property). The estimate on the weight is used crucially in \cite{CKV1}. Dafermos and DiPerna \cite{MR546634,MR523630} showed the weak/strong stability and uniqueness in the case of  Lipschitz solutions by studying the evolution of the quantity:%. In the spirit of the   $L^1$ theory of Kruzkov \cite{zbMATH03341462_kruzkov_translated}, from a single convex entropy $\eta$, it is possible to generate a whole family of entropy indexed by $w\in \Nu_0$ as
$$
\eta(v|w)=\eta(v)-\eta(w)-\nabla\eta(w)\cdot(v-w),
$$
where $w$ is a strong solution, and $v$ is only a weak solution. The quantity $\eta(v|w)$, defined for $(v,w)\in \Nuo\times \Nu$, is called the relative entropy of $v$ with respect to $w$, and is equivalent to $|v-w|^2$. The role of the $a$-contraction  property is to extend the study of the relative entropy to the case where the strong solution $w$ is replaced by a discontinuous traveling wave, a shock, and the weak solution $v$ belongs to $\Sweak$. As we will illustrate below, the study of the relative entropy in this setting is the building block upon which the general theory of weak/BV stability and uniqueness of \cite{CKV1} is built. In the presence of shocks, it is necessary to introduce shifts and weights (the $a$ coefficients, see \cite{serre_vasseur}).
 % In the scalar case, it is well-known that entropy solutions satisfy an $L^1$ contraction estimate due to Kruzkov \cite{zbMATH03341462_kruzkov_translated}. However, the $L^1$ contraction property may fail spectacularly for systems \red{add citations for examples}. The role of the $a$-contraction property is to replace the $L^1$ distance with a suitable pseudo-distance under which the semigroup of entropy solutions is contractive (not quasi-contractive).
 In particular, following  \cite{MR3519973}, for any $u\in\Sweak$, we define a pseudo-distance $E_t(u)$,
\begin{equation}\label{E_defn}
\begin{aligned}
E_t(u)&=a_1\int_{-\infty}^{h(t)} \eta(u(t,x)|u_L)\,dx+a_2\int_{h(t)}^\infty  \eta(u(t,x)|u_R)\,dx, %\quad\eta(v|w)=\eta(v)-\eta(w)-\nabla\eta(w)(v-w)% \text{ for}(v,w)\in \Nuo\times \Nu,
\end{aligned}
\end{equation}
where $t\mapsto h(t)$ is a Lipshitz shift function, $a_1,a_2 > 0$ are weights, $(u_L,u_R,\sigma)$ is a fixed entropic shock with $u_L,u_R\in \Nu$, and $\eta$ is the strictly convex entropy for which $u(t,x)$ satisfies \eqref{entropy}. The weight coefficients $a_1,a_2$ depend only on $f$, $\eta$, and the choice of shock $(u_L,u_R,\sigma)$. However, the shift usually depends not only on $f$, $\eta$, and $(u_L,u_R,\sigma)$, but, additionally, on the particular weak solution $u$ under consideration. Since the relative entropy is equivalent to $|v-w|^2$,
%The shift $h(t)$ may only depend on $f$, $\eta$, and the choice of shock $(u_L,u_R,\sigma)$, not on the particular weak solution $u$ under consideration. The quantity $\eta(v|w)$, defined for $(v,w)\in \Nuo\times \Nu$, is called the relative entropy of $v$ with respect to $w$, and is equivalent to $|v-w|^2$. 
if $S(t,x)$ is the shock wave corresponding to $(u_L,u_R,\sigma)$, then $E_t$ satisfies 
\begin{equation}\label{E_approx}
E_t(u) \approx \int_{\R} a(t,x) |u(t,x ) - S(t,x+h(t))|^2 \ dx,\qquad \text{where }a(t,x) = \begin{cases} a_1 &\text{if }x < h(t)\\ a_2 &\text{if }x > h(t)\end{cases},
\end{equation}
in an appropriate sense. The variation in the weight $a(t,x)$, determined by $a_1$ and $a_2$, measures how far $E_t(\cdot)$ is from the standard $L^2$ distance. The main result of this paper is that for any extremal shock $(u_L,u_R,\sigma)$, there are amplitudes $a_1,a_2$, for which the pseudo-distance $E_t$, defined from $a_1,\ a_2$, and $(u_L,u_R,\sigma)$ via (\ref{E_defn}) satisfies the contraction estimate
\begin{equation}\label{a_cont_estimate}
E_t(u)  \le E_0(u),
\end{equation}
for any $u\in \Sweak$  and an associated  shift $t \mapsto h(t)$. Moreover, we give a quantitative control on the variation in the weight $a(t,x)$. The inequality (\ref{a_cont_estimate}) is known as the $a$-contraction property, with weight function $a$ and shift $h$. The control on the weight $a$ enables the estimate (\ref{a_cont_estimate}) to be related back to and used in the standard small $BV$ framework. More precisely, the main theorem of this paper is:

%%%%%%%%%%%%%%%%%%%%%%%%%%%%
\begin{theorem}\label{thm_main}

Consider a system \eqref{cl}  verifying all the Assumptions \ref{assum}.
Let $d\in \Nu$. Then there exist constants $\alpha_1,\alpha_n, \hat{\lambda}$ and  $C_1, \eps>0$, with   $\alpha_1<\alpha_n$, % and  $\hat{\lambda}\geq 2L$, 
 such that the following is true:\vskip0.1cm 
Consider any 1-shock or n-shock  $(u_L,u_R)$   with $|u_L-d|+|u_R-d|\leq \eps$
 and let  $s_0 $ be the strength of the shock $s_0 =|u_L-u_R|$.
%, and let $i=1$ if it is a 1_shock, and $i=2$ if it is a  
For any $a_1>0,a_2>0$ verifying 
\begin{eqnarray}
1+\frac{C_1s_0 }{2} \leq \frac{a_1}{a_2}\leq  1+2C_1s_0 &&\text{ if } (u_L,u_R) \text{ is a 1-shock}\label{control_a_one}\\
1-2C_1 s_0\leq \frac{a_1}{a_2}\leq 1-\frac{C_1s_0 }{2} && \text{ if } (u_L,u_R) \text{ is an n-shock}\label{control_a_two},
\end{eqnarray}
 and for any $u\in \Sweak$, any $\bar{t}\in[0,\infty)$, and any $x_0\in \RR$,
there exists a Lipschitz shift function  $h\colon[\bar{t},\infty)\to\mathbb{R}$, with $h(\bar{t})=x_0$,  such that the following dissipation functional verifies
\begin{align} \nonumber
&a_2\left[q(u(h(t)+,t);u_R)-\dot{h}(t) \ \eta(u(h(t)+,t)|u_R)\right]\\ \label{diss:shock}
&\qquad\qquad  -a_1\left[q(u(h(t)-,t);u_L)-\dot{h}(t) \ \eta(u(h(t)-,t)|u_L)\right] \leq 0,
\end{align}
for almost all  $t\in[\bar{t},\infty)$. % $1/2<y<2$.
%\vskip0.2cm
Moreover, if $(u_L,u_R)$ is a 1-shock, then for almost all  $t\in[\bar{t},\infty)$,  %any $v\in\mathbb{R}^2$ such that $\abs{v-a}\leq \epsilon$, 
\begin{align}\label{control_one_shock1020}
-\frac{\hat{\lambda}}{2}\leq \dot{h}(t) \leq \alpha_1<\inf_{v\in B_{2\eps}(d)}\lambda_2(v).
\end{align}
%\begin{align}
%h_1(t)\leq x_0 +\lambda_2(v)(t-\bar{t}).
%\end{align}
Similarly, if $(u_L,u_R)$ is a n-shock,  then for almost all  $t\in[\bar{t},\infty)$, %then for any $v\in\mathbb{R}^2$ such that $\abs{v-a}\leq \epsilon$, 
\begin{align}\label{control_two_shock1020}
\sup_{v\in B_{2\eps}(d)}\lambda_{n-1}(v)< \alpha_n\leq \dot{h}(t) \leq\frac{\hat{\lambda}}{2}.
\end{align}

%\textbf{Note that in \eqref{control_one_shock1020} and \eqref{control_two_shock1020} we had to change $B_{\eps}(d)$ (in front tracking+shifts paper) to $B_{2\eps}(d)$ to account for the the fact that $u_-$ can be $\frac{1}{C}$ away from $u_L\in B_{\eps}(d)$}
%\begin{align}
%h_2(t)\geq x_0 +\lambda_1(v)(t-\bar{t}).
%\end{align}

%\vskip0.2cm

%Furthermore, for a fixed function $u\in \Sweak$, if $h_1$ is the shift function associated to a 1-shock with $x_0^1$, and $h_2$ is the shift function associated to a 2-shock with $x_0^2$, verifying $x_0^1\leq x^2_0$, then 
%\begin{align}
%h_1(t)\leq h_2(t) \qquad \text{for all } t\geq \bar{t},
%\end{align}
%and if $h_1(t)= h_2(t)$ for some $t\in[\bar{t},T)$, and $h_1$ and $h_2$ are both differentiable at $t$, then
%\begin{align}
%\dot{h}_1(t) < \dot{h}_2(t).
%\end{align}

\end{theorem}
%%%%%%%%%%%%%%%%%%%%%%%%%%%%%%%%%%%
This theorem implies the following $a$-contraction property of small extremal shocks:
\begin{corollary}\label{cor}
Under the same assumptions and notations as Theorem \ref{thm_main}, there exists a constant $C>0$ such that for any $1$-shock or $n$-shock $(u_L,u_R)$ small enough, there are amplitudes $a_1,a_2 > 0$ with $\left|\frac{a_1}{a_2} - 1\right| \le C|u_L-u_R|$, verifying the following: For any weak solution $u\in \Sweak$, there exists a Lipschitz shift $t\mapsto h(t)$ 
such that the pseudo-distance,
$$
E_t(u) = a_1\int_{-\infty}^{h(t)} \eta(u(t,x)|u_L)\,dx+a_2\int_{h(t)}^\infty  \eta(u(t,x)|u_R)\,dx,
$$
is contractive. More precisely, $E$ is a non-increasing function of time: For every $t>0$,
$$
E_t(u) \leq E_0(u).
$$
\end{corollary}

\begin{comment}
By the existence of a convex entropy $\eta$, we can define an associated pseudo-distance for any $(a,b)\in \Nuo\times \Nu$:
\begin{align}
\eta(a|b)=\eta(a)-\eta(b)-\nabla\eta(b)(a-b).
\end{align}
The quantity  $\eta(a|b)$ is called the relative entropy of $a$ with respect to $b$, and is equivalent to $|a-b|^2$.  
We also define the relative entropy-flux: For $(a,b)\in \Nuo\times \Nu$,
\begin{align}\label{def_q}
q(a;b)=q(a)-q(b)-\nabla\eta(b)\cdot(f(a)-f(b)).
\end{align}
Our main result is the following.
\end{comment}

As stated above, this provides an $L^2$ stability result up to the shift $h$, since we have the following straightforward lemma due to the convexity of $\eta$ and Taylor's theorem.
%\textbf{(copied verbatim from \cite{move_entire_solution_system}....)}
\begin{lemma}[From \cite{Leger2011,VASSEUR2008323}]\label{l2_rel_entropy_lemma}
 For any fixed compact set $V\subset \mathcal{V}$, there exists $c^*,c^{**}>0$ such that for all $(u,v)\in \Nuo\times V$,
\begin{align}
c^*|u-v|^2\leq \eta(u|v)\leq c^{**}|u-v|^2.
\end{align}
The constants $c^*,c^{**}$ depend on bounds on the second derivative of $\eta$ in $V$, and on the continuity of $\eta$ on $\Nuo$. 
\end{lemma}
%This elementary  lemma follows directly from Taylor's theorem  (see \cite{Leger2011,VASSEUR2008323}).
\vskip0.3cm
The relative entropy $\eta(u | v)$ and the corresponding relative entropy method were first introduced by Dafermos \cite{MR546634} to show the weak/strong stability of Lipschitz solutions to \eqref{cl}\eqref{ineq:entropy}.   %It was then used by DiPerna in  \cite{MR523630} to show the uniqueness of shocks (see also Chen Frid \cite{MR1911734} for the uniqueness of the Riemann problem of the Euler equation). 
The strength of this method stems from the fact that if $u$ is a weak solution of \eqref{cl}, \eqref{ineq:entropy}, then $u$ verifies a full family of entropy inequalities. Corresponding to the relative entropy $\eta(a | b)$, we define $q(a ; b)$, the relative entropy flux via:
\begin{align}\label{def_q}
q(a;b)=q(a)-q(b)-\nabla\eta(b)\cdot(f(a)-f(b)).
\end{align}
Then, for any $b\in \Nu$ constant, each $\eta(\cdot | b)$ is an entropy for \eqref{cl} and each $u\in\Sweak$ satisfies
\begin{equation}\label{ineq:relative}
(\eta(u|b))_t+(q(u;b))_x\leq 0,
\end{equation}
in the sense of distributions.
Similar to the Kruzkov theory for scalar conservation laws  \cite{zbMATH03341462_kruzkov_translated}, \eqref{ineq:relative} provides a full family of entropies measuring the distance of the solution to any fixed values $b$ in $\Nu$. The main difference is that the distance is equivalent to the square of the $L^2$ norm rather than the $L^1$ norm. As in the Kruzkov theory, \eqref{ineq:relative} directly implies the stability of constant solutions (by integrating \eqref{ineq:relative} in $x$). Modulating \eqref{ineq:relative} with a smooth function $(t,x)\mapsto b(t,x) $ provides the well-known weak-strong uniqueness result. 
\vskip0.3cm 
However, when considering discontinuous functions $b$ with shock fronts, the situation diverges significantly from that of Kruzkov's, because the $L^2$ norm is not as well suited as the $L^1$ norm for the study of stability of shocks. Nevertheless, the method was used by DiPerna \cite{MR523630} to show the uniqueness of single shocks (see also Chen, Frid, and Li, \cite{MR1911734} for the Riemann problem of the Euler equation). 
\vskip.3cm
In \cite{VASSEUR2008323}, it was proposed that the relative entropy method could be used to obtain stability of discontinuous solutions. The guiding principle of this program is that the $L^2$ distance captures the stability of shock profiles quite well, although only up to a shift, which is more sensitive to $L^2$ perturbations \cite{Leger2011}. Leger, in \cite{Leger2011_original}, showed that in the scalar case, shock profiles satisfy the $a$-contraction property with $a_1 = a_2$; the contraction property holds in $L^2$ up to the shift $t\mapsto h(t)$.
% reminiscent to the $L^1$ contraction of the Kruzkov theory.
 However, it was shown in \cite{serre_vasseur} that the $L^2$ contraction property is usually false for systems, necessitating the introduction of the weights $a_1$ and $a_2$.
%However, it can be recovered by weighting the relative entropy \cite{MR3519973} as in Corollary \ref{cor}.
In \cite{MR3519973}, it is shown that for a large class of systems, introducing the weights $a_1$ and $a_2$ does indeed recover the $L^2$ contraction property, in the form of the $a$-contraction estimate (\ref{a_cont_estimate}). However, \cite{MR3519973} does not show any precise control over the weights $a_1$ and $a_2$.
This is the main improvement of Theorem \ref{thm_main} over the existing theory, which is crucial for relating the results in the $a$-contraction theory, i.e. Corollary \ref{cor}, back to the general small $BV$ stability theory.
%Note that the definition of the functional $E$ with  $a_1=a_2$, and $h(t)=st$ would imply the contraction property of the shock for the relative entropy. But the result, to be valid, needs the weights $a_i$, and the shifts $h$, giving the name to the method: a-contraction with shifts.
\vskip0.3cm
In particular, the next step in the program of $BV$ stability is showing small $BV$ solutions are stable (and unique) in the larger $\Sweak$. A BV/weak stability result of this type is shown for $2\times 2$ systems in the companion paper \cite{CKV1} using Theorem \ref{thm_main}, and hence most of the $a$-contraction theory, as a black box. %For general $n\times n$ systems,  Glimm showed in  \cite{gl} the existence of global in time small BV entropic solutions to   \eqref{cl} \eqref{ineq:entropy}. Uniqueness of these solutions   was established by Bressan and Goatin \cite{MR1701818} under the Tame Oscillation Condition. It improved an earlier theorem by Bressan and LeFloch \cite{MR1489317}. Uniqueness was also known to  prevail when the Tame Oscillation Condition is replaced by the assumption that the trace of solutions along space-like curves has bounded variation, see Bressan and Lewicka \cite{MR1757395}. Since any $BV$ solutions belongs to $\Sweak$, the result of \cite{CKV1} (together with Theorem \ref{thm_main}) shows that in the case of 2 unknowns, these a priori assumptions are not needed to obtain the uniqueness result.
As written above, since any $BV$ solution belongs to $\Sweak$, the result of \cite{CKV1} (together with Theorem \ref{thm_main}) shows that in the case of 2 unknowns, both the Tame Oscillation Condition and the condition of bounded variation along space-like curves are unnecessary assumptions in the uniqueness theory for small $BV$ solutions.
\vskip0.3cm
Let us sketch the method used in \cite{CKV1}.
Consider a sequence $u_n\in \Sweak$ such that the initial values $u^0_n\in L^\infty$ converge in $L^2$ to a small BV function $u^0_{BV}$. 
To make the leap from a single shock to the general small $BV$ solution $u_{BV}$ with initial value $u^0_{BV}$, in \cite{CKV1}, one replaces $u_{BV}$ with a piecewise constant function $v_n$ %approximation $u_{BV}^\eps$ of $u_{BV}$, 
generated from an approximation of the initial value $u^0_{BV}$. The approximation $v_n$ is obtained using the front tracking algorithm \cite{MR1492096, Bbook}, but with artificial shifts on the fronts. We apply Theorem \ref{thm_main} to the finitely many discontinuities present in $v_n$ to obtain the contraction inequality \eqref{a_cont_estimate} when one modifies the definition of the pseudo-distance from (\ref{E_defn}) to read
\begin{equation}\label{E2}
E_t^n(u) = \int_{\R} a_n(t,x) \eta(u_n(t,x) | v_n(t,x)) \ dx,
\end{equation}
where $a_n$ is now piecewise constant and $v_n$ is defined from the shift functions $t \mapsto h_{n,j}(t)$, each associated to a singularity $j$. 
Due to the presence of the shifts $h_{n,j}$, the function $v_n$ cannot be directly compared to $u_{BV}$: In fact, $v_n$ is no longer even a solution to \eqref{cl}. Instead, it is shown that $v_n$ remains small in $BV$ and uniformly verifies the bounded variation along space-like curves condition, introduced in  \cite{MR1757395}. By controlling the weight function $a_n$ uniformly by above and by below, the $a$-contraction estimate associated (\ref{E2}) implies that since both  $u_n^0$ (initial value of the solution $u_n\in \Sweak$) and $v^0_n$ converge in $L^2$ to $u^0_{BV}$, for every $t>0$, $u_n(t)$ also converges in $L^2$ to $v(t)$, the limit of the associated $v_n(t)$. %Note that since the $v_n$ depend, via the shifts, on $u_n$, the function $v$ depends both on the weak solution $u$. 
The function $v$ is a solution to (\ref{cl}) (since each $u_n$ solves (\ref{cl})), is small $BV$, and satisfies the bounded variation along space-like curves condition (since each $v_n$ does). By the uniqueness theorem of \cite{MR1757395} for small $BV$ solutions, we conclude $u=v=u_{BV}$. %This provides also the uniqueness by fixing $u=u_n$ for all $n$, in case $u^0=u_BV^0$.
%$$
%\mathcal{L}=\{u\in L^\infty(\R^+;BV(\R)), \text{with small} BV {variation}, \text{verifying the bounded oscillations condition  along space-like curves}\}.
%$$
\vskip0.1cm
To rigorously justify this argument, two important properties have to be verified at the level of the $a$-contraction. The first concerns control by above and below of the ``composite weight function,'' $a_n(t,x)$. Since weights $a_1,a_2$ are generated for each singularity in $v_n$, to control the weight function when the number of singularities grows and $n$ tends to infinity, it is crucial to show that the variation of the weight can be chosen proportionally to the size of the shock $u_L,u_R$. This is the main difficulty tackled in this article in
\eqref{control_a_one} and \eqref{control_a_two}.
\vskip0.1cm
The second important property concerns the functions $v_n$: We must verify $v_n$ stay small in $BV$ and uniformly verify the  bounded variation along space-like curves condition.
This is ensured by the front tracking method, provided that all the artificial shifts generated by the $a$-contraction method maintain the separation of wave families.
%This introduces two important difficulties. The first concerns the ``composite shift,'' $t\mapsto h^\eps(t)$.
 The method generates one artificial shift $h(t)$ for each singularity in $v_n$, depending heavily on the structure of $v_n$ itself. Therefore, one must prevent the situation in which an artificial shift %generated at a single singularity
  allows a 1-shock to collide from the left with a 2-shock. Such a situation would cause the whole process to collapse. Following \cite{a_contraction_riemann_problem}, this problem is solved by \eqref{control_one_shock1020} and \eqref{control_two_shock1020}.
\vskip0.3cm
In a parallel program, a similar method is developed with the goal of proving the Bianchini-Bressan conjecture, namely  the convergence from Navier-Stokes to Euler for small $BV$ initial datum \cite{Bbook}. The aim is to obtain a $BV$-weak stability result, where the space $\Sweak$ is replaced by the set of weak inviscid limits of solutions to Navier-Stokes equation.  The case of a single shock (analogous to Theorem \ref{thm_main} when working with $\Sweak$) was proved in \cite{2019arXiv190201792K}. Although this program is similar in spirit to that of the present paper, the proofs in \cite{2019arXiv190201792K} differ greatly from those in this paper.

\section{Proof Outline}

The main idea behind the proof of the $a$-contraction property, as stated in Corollary \ref{cor}, is integrating \eqref{ineq:relative} for $b=u_L$ on $x<h(t)$, and  for $b=u_R$ on $x>h(t)$. The strong trace property of $u$ guarantees that for almost every $t>0$, 
\begin{equation}\label{ineqE}
\frac{dE(t)}{dt} \leq h'(t) (a_1 \eta(u_-(t)|u_L)-a_2\eta(u_+(t)|u_R))-a_1q(u_-(t);u_L)+a_2q(u_+(t);u_R).
\end{equation}
Without loss of generality, we consider the case where $u_L,u_R$ is a $1$-shock.
We define a subset of  the state space $\Nu$  as 
 \begin{equation}\label{defPi}
 \Pi:=\left\{v\in \V \ | \   \frac{a_1}{a_2}  \eta(v|u_L) - \eta(v|u_R)<0 \right\} .
\end{equation} 
If the weak solution $u$ does not have any discontinuity at $h(t)$ and $u(t,h(t)) =v\in \partial \Pi$, we have at this point $a_1 \eta(v|u_L)-a_2  \eta(v|u_R)=0 $, and so
the coefficient in front of $h'(t)$ in \eqref{ineqE} vanishes. In this case,  the value of the right hand side of \eqref{ineqE} is not changed if we replace $h'(t)$ by
$\lambda_1(v)$, the first eigenvalue of $f'(v)$. In order for $a$-contraction \eqref{ineqE} to remain valid in all cases, we need that for every $v\in \partial \Pi$:
\begin{equation}\label{ineqD}
D_{cont}(v):= [q(v;u_R)-\lambda_1(v)\eta(v|u_R)]-\frac{a_1}{a_2}  [q(v;u_L)-\lambda_1(v)\eta(v|u_L)]\leq 0.
\end{equation}
Note that this inequality, together with the set $\Pi$ depend only on the system \eqref{cl}, not on any actual solutions. %This is a condition that has to be valid for the $a$-contraction to be valid for any weak solutions $u$.
If $a_1=a_2$, $\partial \Pi$ is the hyperplane of all states $v$ equidistant from $u_L$ and $u_R$. For $n\geq 2$, this is a large set, and for many systems, \eqref{ineqD} is not verified for all $v\in \partial\Pi$ (see \cite{serre_vasseur}). However, under reasonable assumptions, it can be shown that \eqref{ineqD} is valid 
for $a_2/a_1$ small enough, that is, for states $v$ close enough to $u_L$. Note that  $\Pi$ can be rewritten as 
$
 \Pi=\left\{v\in \V \ | \   \eta(v|u_L) < (a_2/a_1) \eta(v|u_R) \right\} .
 $
Therefore, for small values of $a_2/a_1$, $\Pi$ is actually a small neighborhood of $u_L$, and \eqref{ineqD} holds on $\partial\Pi$ 
%This shows that the  property becomes  true for $a_1/a_2$ small enough since $\Pi$ becomes a small neighborhood of $u_L$ in this asymptotic
 (see \cite{MR3537479,MR3519973}). %This is the main idea behind the $a$-contraction theory: We can obtain a contraction by weighting the relative entropy. 
\vskip0.3cm
%Note that this argument is too crude to show that $a_1/a_2-1$ can be chosen to the order of the side of the shock $s_0=|u_l-u_R|$. 
Our first proposition is a significant improvement of this argument for small shocks. It shows that \eqref{ineqD}  holds true  in all $\Pi$ with  $a_1/a_2 \approx 1+C_1s_0$, for a big enough fixed constant $C_1$, and any small enough strength $s_0$ of the  shock $u_L,u_R$. To describe this proposition, let us first fix some notation. For any $C>0, s_0>0$, consider shocks $u_L, u_R$ with strength $s_0$.  We then define  
\begin{equation}\label{defn_eta}
\tilde\eta(u) := (1 + Cs_0)\eta(u | u_L) - \eta(u | u_R) \qquad \text{and}\qquad \tilde q(u) := (1+Cs_0)q(u ; u_L) - q(u ; u_R).
\end{equation}
With this notation $D_{cont}$ and $\Pi_{C,s_0}$ have particularly simple forms, namely, 
\begin{equation}\label{defn_Dcont}
D_{cont}(u) = -\tilde q(u) + \lambda_1(u)\tilde\eta(u) \qquad \text{and} \qquad  \Pi=\Pi_{C,s_0} = \{u \ | \ \tilde\eta(u) < 0\}.
\end{equation}
With this language in hand, we state our first important proposition.
\begin{proposition}\label{prop_cont}
Consider a fixed system and a state set $\Nu$ verifying Assumptions \ref{assum}, and $d\in \Nu$. 
There is a  universal constant $K$ depending only on the system and $d$, such that for any $C_1>0$ large enough, there exists  $\tilde{s}_0(C_1)>0$,   such that for any $C>0$ with $C_1/2\leq C\leq 2C_1$, for any  1-shock $u_L,u_R$  with $|u_L-d|+|u_R-d|\leq \tilde s_0$, and $|u_L-u_R|=s_0$ with 
$0 < s_0 < \tilde s_0$ and any $u\in \Pi_{C,s_0}$, we have 
\begin{equation}
D_{cont}(u) \le -Ks_0^3.
\end{equation}
\end{proposition} 
This proposition proves (\ref{ineqD}) in all of $\Pi$ and provides a sharp bound on the dissipation. These two properties will be needed to show the second proposition. In our study of \eqref{ineqE}, we have considered only the case where the solution $u$ has no discontinuity at $x=h(t)$. However, we frequently have situations where $x=h(t)$ corresponds locally to a Rankine-Hugoniot discontinuity curve of $u$. At such a time $t$,  $(u(t,h(t)-), u(t,h(t)+), \dot{h}(t))$ corresponds to a shock $(u_-,u_+,\sigma_\pm)$ of the system. The second proposition ensures that \eqref{ineqD} remains valid in this situation.
For any such shock $(u_-,u_+,\sigma_\pm)$ of the system with $u_-,u_+\in \Nu$, we define
\begin{equation}\label{defn_DRH}
D_{RH}(u_-,u_+,\sigma_\pm)= [q(u_+;u_R)-\sigma_\pm\eta(u_+|u_R)]-(1+Cs_0)  [q(u_-;u_L)-\sigma_\pm \eta(u_-|u_L)].
\end{equation}
We actually only  need to control this quantity for  1-shocks  (and n-shocks) such that $u_-\in \Pi_{C,s_0}$ (see section \ref{secSam}). In the other situations, the sign on the dissipation is enforced by the choice of $\dot{h}(t)$. 
We state our second main proposition.
\begin{proposition}\label{prop_shock}
Consider a fixed system and a state set $\Nu$ verifying Assumptions \ref{assum}, and $d\in \Nu$. For any  $C_1>0$ large enough, there exists $\tilde{s}_0(C_1)>0 $, such that 
for any $C>0$ with $C_1/2<C<C_1$,  for any  1-shock $u_L,u_R$  with $|u_L-d|+|u_R-d|\leq \tilde s_0$, and $|u_L-u_R|=s_0$ with  $0 < s_0 < \tilde s_0$, and any 1-shock $(u_-,u_+,\sigma_\pm)$, with  $u_- \in \Pi_{C,s_0}$ and $u_+\in \Nu$, we have
\begin{equation}
D_{RH}(u_-,u_+,\sigma_\pm) \le 0.
\end{equation}
\end{proposition}
 If we consider a 1-shock with  $u_+$ converging to $u_-$, the shock $(u_-,u_+,\sigma_\pm)$ converges to $(u,u,\lambda_1(u))$ with $u=u_-$, and so 
 $D_{RH}(u_-,u_+,\sigma_\pm)$ converges to $D_{RH}(u,u,\lambda(u))=D_{cont}(u)$. This justifies the definition of $D_{cont}$. It shows also that the continuous case is included in the Proposition \ref{prop_shock}. However, the stronger estimate contained in Proposition \ref{prop_cont} is necessary to prove Proposition \ref{prop_shock} in the shock case. 
\vskip0.3cm

The rest of the paper is organized as follows: The abstract assumptions on the system together with the preliminaries are presented in Section \ref{secAssum}. Then, Theorem \ref{thm_main} is proved in Section \ref{secSam} assuming Proposition \ref{prop_cont} and Proposition \ref{prop_shock}. Proposition \ref{prop_cont} is proved in Section \ref{secWill1}. Finally, Proposition \ref{prop_shock} is proved in Section  \ref{secWill2}.
\vskip0.3cm

We attempt to adhere to the following notation conventions in the sequel, which will be especially important in Sections \ref{secWill1} and \ref{secWill2}:
\begin{itemize}
\item Throughout, $C$ denotes the specific constant used in the definition of $\Pi_{C,s_0}$ in (\ref{defn_eta}). Constants denoted $K$ (often with descriptive subscripts or superscripts) will denote universal constants in the sense that they depend only on the system (\ref{cl}) and in particular, not on $C$ or $s_0$. 
\item We will make heavy use of the notation $a \lesssim b$, which we take to mean there exists $K$ depending only the system, that is, on $f$, $\eta$, and $q$, such that $a \le Kb$. We also write $a\sim b$ to denote $a\lesssim b$ and $b\lesssim a$.
\item We often write expressions of the form $a(C,s_0) \lesssim b(C,s_0)$, for sufficiently large $C$ and sufficiently small $s_0$. This should be read as there exists a $K_1 = K_1(f,\eta,q) > 0$ and $\tilde C = \tilde C(f,\eta,q) > 0$ sufficiently large and $\tilde s = \tilde s(f,\eta,q,\tilde C)$ sufficiently small such that for $\tilde C < C < 2\tilde C$ and $\tilde s > s_0 > 0$, $a(C,s_0) \le K_1 b(C,s_0)$. 
\item We also use the notation $a = b + \bigO(g)$ to mean $|a - b| \lesssim g$, where $a,b$ might be tensor-valued. This notation is frequently used in conjunction with $a\lesssim b$.
\item For $A$ a $k_1$-tensor and $B$ a $k_2$-tensor, we define $A\tens B$ as the $k_1 + k_2$ tensor, $a_{i_1\cdots i_{k_1}}b_{j_1\cdots,j_{k_2}}$. When working with tensor expressions, we also adopt the Einstein summation convention that we sum over any repeated indices.
\item We use both $g^\prime$ and $\nabla g$ to denote derivatives of functions. We attempt to use $g^\prime$ for $g$ a vector or matrix-valued function and $\nabla g$ for $g$ a scalar-valued function.
\item Finally, $|\cdot|$ is used indiscriminately to denote finite dimensional norms on vectors, matrices, and higher tensors.
\end{itemize}

\section{Assumptions and preliminaries}\label{secAssum}

We first list the abstract assumptions needed on the system (\ref{cl}). Note that the system is defined entirely by the flux function $f$.

\begin{assumption}\label{assum}
We assume that the flux function has the following regularity on the bounded, convex set of states $\Nu_0$:  $f=(f_1,\cdot\cdot\cdot, f_n)\in [C(\Nuo)]^n\cap [C^4(\Nu)]^n$, where $\Nu$ is the interior of $\Nu_0$. In addition, we assume the following:
\begin{itemize}
\item[(a)] For any $u\in \Nu$, $f'(u)$ is a diagonalizable matrix with eigenvalues verifying   $\lambda_1(u)<\lambda_2(u)$ and $\lambda_{n-1}(u)<\lambda_n(u)$.
For $i=1,n$ we denote $r_i(u)$ a unit eigenvector associated to the eigenvalue $\lambda_i(u)$.
\vskip0.1cm
\item[(b)]  For any $u\in\Nu$, and $i=1,n$, we assume $\lambda'_i(u)\cdot r_i(u)\neq0$.
\vskip0.1cm
%Both characteristics families of the system \eqref{cl} are genuinely nonlinear  in $\Nu$ in the sense of Lax \cite{MR0093653}.
\item[(c)] There exists a strictly convex  function $\eta\in C(\Nuo)\cap C^3(\Nu)$ and a function $q\in C(\Nuo)\cap C^3(\Nu)$ such that
\begin{equation*}\label{eq:entropy}
q'=\eta'f' \qquad \mathrm{on} \ \ \Nu.
\end{equation*}
\vskip0.1cm
\item[(d)] For $u_L\in \Nu$, we denote $s\to  S^1_{u_L}(s)$ the $1$-shock  curve through $u_L$ defined for $s>0$. We choose the parametrization such that $s=|u_L-S^1_{u_L}(s)|$. Therefore, $(u_L, S^1_{u_L}(s), \sigma^1_{u_L}(s))$ is the $1$-shock with left hand  state $u_L$ and strength $s$. Similarly, we define $s\to S^n_{u_R}$ to be the $n$-shock curve such that $(S^n_{u_R}, u_R, \sigma^n_{u_R}(s))$ is the $n$-shock with right hand state $u_R$ and strength $s$. We assume that these curves are defined globally in $\Nu$ for every $u_L\in \Nu$ and $u_R\in \Nu$.

%\item[(d)] For any $ b\in \Nu$, and any left eigenvector  $\ell$ of $ f'(b)$:  the function $u\to \ell\cdot f(u)$ is either convex or concave on $\Nu$.
\item[(e)] There exists $L>0$ such that $|\lambda_i(u)|\leq L$ for any $u\in \Nu$ and $i=1,n$.
\item[(f)] (for 1-shocks) If $(u_L,u_R)$ is an entropic Rankine-Hugoniot discontinuity with shock speed $\sigma$, then $\sigma>\lambda_1(u_R).$
\item[(g)] (for 1-shocks) If $(u_L,u_R)$ (with $u_L\in B_{\epsilon}(d)$) is an entropic Rankine-Hugoniot discontinuity with shock speed $\sigma$ verifying,
\begin{align*}
\sigma\leq \lambda_1(u_L),
\end{align*}
then $u_R$ is in the image of $S^1_{u_L}$. That is, there exists $s_{u_R}\in[0,s_{u_L})$ such that $S^1_{u_L}(s_{u_R})=u_R$ (and hence $\sigma=\sigma^1_{u_L}(s_{u_R})$).
\item[(h)] (for n-shocks) If $(u_L,u_R)$ is an entropic Rankine-Hugoniot discontinuity with shock speed $\sigma$, then $\sigma<\lambda_n(u_L).$
\item[(i)] (for n-shocks) If $(u_L,u_R)$ (with $u_R\in B_{\epsilon}(d)$) is an entropic Rankine-Hugoniot discontinuity with shock speed $\sigma$ verifying,
\begin{align*}
\sigma\geq \lambda_n(u_R),
\end{align*}
then $u_L$ is in the image of $S^n_{u_R}$. That is, there exists $s_{u_L}\in[0,s_{u_R})$ such that $S^n_{u_R}(s_{u_L})=u_L$ (and hence $\sigma=\sigma^n_{u_R}(s_{u_L})$).
\item[(j)]  For $u_L\in \Nu$, and  for all $s>0$,  $\ds{\frac{d}{ds}\eta(u_L | S^1_{u_L}(s))}>0$ (the shock ``strengthens" with $s$).
Similarly, for $u_R\in \Nu$, and for all $s>0$, $\ds{\frac{d}{ds}\eta(u_R | S^n_{u_R}(s))}>0$. Moreover, for each $u_L,u_R\in \Nu$ and $s > 0$, $\frac{d}{ds}\sigma^1_{u_L}(s) < 0$ and $\frac{d}{ds}\sigma^n_{u_R}(s) > 0$.
\end{itemize}
\end{assumption} 
As written in the introduction, these assumptions are classical and fairly general. Assumption (a) is the hyperbolicity assumption. Note that we assume strict hyperbolicity only in $\Nu$, and only for the smallest and largest eigenvalues. Assumption (b) is the genuinely nonlinear condition on the $i$-characteristic families, but only for $i=1$ and $i=n$. Assumption (c) is the existence of a strictly convex entropy. Assumption (d) is always valid locally. We assume it globally in $\Nu$. Assumption (e) is the boundedness of the extremal eigenvalues close to the boundary of $\Nu$, where they may fail to be defined. Assumptions (e) to (i) are related to the Liu condition, but are actually weaker. Finally, Assumption (j) says that the strength of the shock increases with $s$ when measured by the relative entropy. It is a natural condition verified by the Euler systems (see \cite{Leger2011}). Note that it has been shown in Barker, Blake, Freist\"{u}hler, and Zumbrun \cite{MR3338447}, that stability may fail if Assumption (j) is replaced by the property $\frac{d}{ds}\eta( S^1_{u_L}(s))>0$.
Assumptions (a) to (j) are now classical for the $a$-contraction theory (see \cite{Leger2011}).

Note that Assumptions (a) to (j) should be considered global assumptions on the system. However, locally more is true. First, the additional regularity of $f$ in $\Nu$, combined with the hyperbolicity of Assumption (a) implies we may always find eigenvalues and eigenvectors for $f^\prime$, locally satisfying stronger estimates than Assumption \ref{assum} (e).

\begin{lemma}[From \protect{\cite[p.~236]{dafermos_big_book}}]\label{lem_normalization}
For each $1 \le i \le n$, there exist $C^3$ functions $\lambda^i$, $l^i$, and $r_i$ satisfying the normalization conditions $|l^i| = |r_i| = 1$, $l^i \cdot r_j = 0$ for $i \neq j$, and $l^i \cdot r_i > 0$ and the eigenvector, eigenvalue relations,
\begin{equation}
l^i(u)f^\prime(u) = \lambda^i(u)l^i(u) \quad \text{and} \quad f^\prime(u)r_i(u) = \lambda^i(u)r_i(u).
\end{equation}
If $\nabla \lambda^1(u) \cdot r_1(u) \neq 0$ for each $u$, we may also choose $r_1,\cdots,r_n$ so that $\nabla\lambda^1(u) \cdot r_1(u) < 0$. Furthermore, if $\eta$ is an entropy for $f$, $\nabla^2\eta(u)r_i(u) \parallel l^i(u)$.
\end{lemma}
\begin{remark}
From the integrability condition on the entropy $\eta$,
\begin{align*}
    \nabla^2\eta(u)f'(u)=(f'(u))^{T}\nabla^2\eta(u),
\end{align*}
we get $\nabla^2\eta(u)r_i(u) \parallel l^i(u)$ (see \cite[p.~243]{dafermos_big_book}).
\end{remark}

Second, the regularity of $f$ combined with Assumption \ref{assum} (a) and (b) implies precise asymptotics and local estimates on the shock curves $\S^1_u$ and $\S^n_u$. We will need the full strength of these local improvements in Sections \ref{secWill1} and \ref{secWill2} when we work close to a single small shock.
\begin{lemma}[From \protect{\cite[p.~263-265, Theorem 8.2.1, Theorem 8.3.1, Theorem 8.4.2]{dafermos_big_book}}]\label{lem_hugoniot}
For any fixed state $v\in \Nu$, there is an open neighborhood $U$ of $v$
such that for each $i = 1, n$, there exist functions $s_u:U \rightarrow \R$, $\sigma_u^i(s):(U\times [0,s_u)) \rightarrow \R$ and $\S^i_u(s):(U\times [0,s_u))\rightarrow \Nu_0$ satisfying the Rankine-Hugoniot condition, for any $u\in U$ and any $0 \le s < s_u$,
\begin{equation}
f(\S^i_u(s)) - f(u) = \sigma_u^i(s)\left(\S^i_u(s) - u\right)
\end{equation}
and the Lax admissibility criterion, for $u \in U$ and $0 < s < s_u$,
\begin{equation}
\lambda^i(\S^i_u(s)) < \sigma_u^i(s) < \lambda^i(u).
\end{equation}
Furthermore, $u \mapsto s_u$ is Lipschitz, $(u,s) \mapsto \sigma^i_u(s)$ is $C^3$, and $(u,s) \mapsto \S^i_u(s)$ is $C^3$. Finally, $\sigma^i_u(s)$ satisfies the asymptotic expansion,
\begin{equation}
\sigma_u^i(s) = \frac{1}{2}\left(\lambda^i(u) + \lambda^i(\S^i_u(s))\right) + \bigO(s^2)
\end{equation}
and similarly $\S^i_u(s)$ satisfies the asymptotic expansion
\begin{equation}
\S^i_u(s) = u + r_i(u)s + \frac{s^2}{2}\nabla r_i(u) r_i(u) + \bigO(s^3).
\end{equation}
\end{lemma}

\begin{comment}
\begin{lemma}[From \red{Leger}]\label{lem_leger}
Suppose $\eta$ is strictly convex and $\eta \in C^2(\Nu)$. Then, for any $\overline{\Nu} \subset\subset \Nu$, there are constants $K_1$ and $K_2$ depending only on $\overline{\Nu}$ and $\eta$ such that for all $u,v\in\overline{\Nu}$,
\begin{equation}
K_1|u-v|^2 \le \eta(u | v) \le K_2|u-v|^2.
\end{equation}
Moreover, we bound the dependence of $K_1$ and $K_2$ on $\overline{\Nu}$ through, $K_1 \sim \inf_{u\in\Nu}\mu_1(u)$ where $\mu_1(u)$ is the minimal eigenvalue of $\nabla\eta^2$ on $\overline{\Nu}$. Similarly, and $K_2 \sim sup_{u\in \overline{\Nu}} \mu_n(u)$ where $\mu_n$ is the maximal eigenvalue of $\nabla\eta^2$.
\end{lemma}
\end{comment}
\vskip.2cm
%We will also use the following computation due to Lax \cite{MR0093653}. 
The following lemma is now classical and gives the entropy lost due to dissipation at a shock $(u,\S^i_u(s),\sigma^i_u(s))$. The version stated here is taken from \cite{Leger2011}:
\begin{lemma}[\protect{\cite[Lemma 3]{Leger2011}}]\label{lem_lax}
Suppose $\eta$ and $q$ are an entropy-entropy flux pair and $u,v\in \Nu$, then for $1 \le i \le n$,
\begin{equation}
q(\S^i_u(s) ; v) - \sigma^i_u(s)\eta(\S^i_u(s) | v) = q(u ; v) - \sigma^i_u(s)\eta(u | v) + \int_0^s \dot{\sigma}(t) \eta(u | \S^i_u(t)) \ dt.
\end{equation}
\end{lemma}
%For a proof of \Cref{lem_lax}, see \cite{MR3537479}.
Note that \Cref{lem_lax} holds for any $i$-family, $i=1,2,\ldots,n$ and not only extremal families (1-shocks and n-shocks) -- it follows from the Rankine-Hugoniot condition.

\vskip.2cm
%\subsection{Structural Lemmas}
\begin{flushleft}
\uline{\bf{Structural Lemmas}}
\end{flushleft}
\vskip.2cm
For the remainder of this section, we adopt the convention (used again in Sections \ref{secWill1} and \ref{secWill2}) that $(u_L,u_R,\sigma_{LR})$ is a fixed entropic $1$-shock with $u_R = \S^1_{u_L}(s_0)$ and $u_L,u_R\in \Nu$. We now study the geometry and basic properties of the set $\Pi = \Pi_{C,s_0}$ defined in (\ref{defn_Dcont}) using $u_L$ and $u_R$. We also prove several basic estimates on the functions $\tilde\eta$ and $\tilde q$ (defined above in (\ref{defn_eta})), which will be used frequently in the study of $D_{cont}$ and $D_{RH}$ in Sections \ref{secWill1} and \ref{secWill2}, respectively.
We begin with the main lemma for the functions $\tilde\eta$ and $\tilde q$ and the related basic structure of $\Pi$.

\begin{lemma}\label{lem_eta}
For any $C > 0$ and $s_0 > 0$, $\tilde\eta$ and $\tilde q$ are an entropy, entropy-flux pair for the system (\ref{cl}) and the set $\Pi = \Pi_{C,s_0}\subset \Nu$ is strictly convex with nonempty interior. For $C$ sufficiently large and $s_0$ sufficiently small, $\Pi$ is compactly contained within $\Nu$ and $diam(\Pi) \sim C^{-1}$, where the implicit constant depends only on the system. Moreover, we compute the exact formulae:
\begin{equation}\label{eqn_lemeta_desired1}
\nabla\tilde\eta(u) = Cs_0\left(\nabla \eta(u) - \nabla\eta(u_L)\right) - \left(\nabla\eta(u_R) - \nabla\eta(u_L)\right) \quad \text{and} \quad \nabla^2\tilde\eta(u) = Cs_0 \nabla^2\eta(u).
\end{equation}
Finally, for any $C$ sufficiently big and $s_0$ sufficiently small, and any $u\in \Pi_{C,s_0}$,
\begin{equation}\label{eqn_lemeta_desired2}
|\nabla\tilde\eta(u)| \lesssim s_0 \qquad \text{and} \qquad \nabla^2\tilde\eta(u)v \cdot v  \sim Cs_0,
\end{equation}
for any $|v| = 1$, where the implicit constants depend only on the system.
\end{lemma}

\begin{proof}
We show $\nabla\tilde q = \nabla \tilde\eta f^\prime$ so that $(\tilde \eta,\tilde q)$ is an entropy, entropy-flux pair for (\ref{cl}). Indeed, differentiating (\ref{defn_eta}) and using $\nabla q = \nabla \eta f^\prime$,
\begin{align}
\nabla \tilde q(u) &= Cs_0\nabla q(u) - (1+Cs_0)\nabla\eta(u_L)f^\prime(u) + \nabla\eta(u_R)f^\prime(u)\\
	&= \left[Cs_0 \nabla\eta(u) - (1+Cs_0)\nabla\eta(u_L) + \nabla\eta(u_R)\right]f^\prime(u)\\
\nabla \tilde \eta(u) &= Cs_0\nabla\eta(u) - (1 + Cs_0)\nabla\eta(u_L) + \nabla\eta(u_R).\label{eqn_lemeta_eqn1}
\end{align}
Now, differentiating (\ref{eqn_lemeta_eqn1}) again yields $\nabla^2\tilde\eta(u) = Cs_0\nabla^2\eta(u)$. This proves $\tilde\eta$ is a strictly convex entropy for (\ref{cl}) and equation (\ref{eqn_lemeta_desired1}) holds.

Next, we note since $\Pi = \{u \in \Nu \ | \ \tilde\eta(u) \le 0\}$, $\Pi$ is strictly convex because $\tilde\eta$ is strictly convex. To obtain the diameter estimate $diam(\Pi) \lesssim C^{-1}$,
we note that if $u \in \Pi$,
\begin{equation}\label{eqn_lemeta_eqn2}
\begin{aligned}
(1 + Cs_0)\eta(u | u_L) &\le \eta(u | u_R) \le \eta(u | u_L) + \int_{u_L}^{u_R} \partial_v\eta(u | v) \ dv\\
	&\lesssim \eta(u | u_L) + |u_R - u_L|(|u - u_L| + |u - u_R|), 
\end{aligned}
\end{equation}
where $\partial_v \eta(u | v) = -\nabla^2\eta(u)(u - v)$. By Lemma \ref{l2_rel_entropy_lemma} and $|u_R - u_L| \sim s_0$, (\ref{eqn_lemeta_eqn2}) implies
\begin{equation}\label{eqn_lemeta_eqn3}
Cs_0|u-u_L|^2 \lesssim s_0(|u - u_L| + s_0).
\end{equation}
So, if $|u - u_L| \ge s_0$, (\ref{eqn_lemeta_eqn3}) implies $|u - u_L| \lesssim C^{-1}$. Therefore, taking $C$ sufficiently large and then $s_0$ sufficiently small, depending only on $C$, say $s_0 \lesssim C^{-1}$, we obtain $diam(\Pi) \lesssim C^{-1}$. Thus, for $C$ sufficiently large $diam(\Pi) \le d(u_L,\Nu_0)$ and $\Pi$ is compactly contained in $\Nu$. 
Furthermore, using $|u - u_L| \lesssim C^{-1}$ for generic $u\in \Pi$ and $|u_L - u_R|\lesssim s_0$, (\ref{eqn_lemeta_desired1}) implies (\ref{eqn_lemeta_desired2}).
Again using Lemma \ref{l2_rel_entropy_lemma}, $\tilde\eta(u_L) = -\eta(u_L | u_R) \lesssim -s_0^2 < 0$ and by continuity $\Pi$ contains an open ball about $u_L$.
Finally, we show the diameter lower bound $diam(\Pi) \gtrsim C^{-1}$ by showing a sufficiently large portion of the line $L(t) = u_L + tr_1(u_L)$ is contained within $\Pi$. Since $\Pi$ is strictly convex and $u_L \in \Pi$, $L(t)$ intersects $\bd\Pi$ exactly twice: once for $t^+>0$ and once for $t^-<0$. Taylor expanding around $u_L$,
\begin{equation}
0 = \tilde\eta(L(t^-)) = \tilde\eta(u_L) + t^-\tilde\nabla\eta(u_L)\cdot r_1(u_L) + \bigO(Cs_0|t^-|^2).
\end{equation}
Therefore, evaluating (\ref{eqn_lemeta_desired1}) at $u = u_L$ and using $u_R = u_L + s_0r_1(u_L) + \bigO(s_0^2)$,
\begin{equation}\label{eqn_lemeta_eqn4}
0 = -s_0 + t^-s_0\nabla^2\eta r_1(u_L)\cdot r_1(u_L) + \bigO(Cs_0|t^-|^2 + s_0^2).
\end{equation}
Because $t^-$ is negative and $\eta$ is strictly convex, rearranging terms in (\ref{eqn_lemeta_eqn4}) yields the bound
\begin{equation}\label{eqn_lemeta_eqn5}
1 + |t^-| \lesssim C|t^-|^2 + s_0.
\end{equation}
Using first (\ref{eqn_lemeta_eqn5}) in the form, $1 \lesssim s_0 + C|t^-|^2$, yields $|t^-| \gtrsim C^{-1/2} \gtrsim Cs_0$ for $s_0$ sufficiently small. Second, using (\ref{eqn_lemeta_eqn5}) again, yields $1 \lesssim C|t^-| + C^{-1}$, which implies $|t^-| \gtrsim C^{-1}$ for $C$ sufficiently large and $s_0$ sufficiently small. So, $$diam(\Pi) \ge |L(0) - L(t^{-})| = |t^{-}| \gtrsim C^{-1}$$ completes the proof.
\end{proof}

Lemma \ref{lem_eta} is fundamental to the forthcoming analysis for several reasons. First, the estimates will be used frequently below. Second, Lemma \ref{lem_eta} is the first indication that $s_0$ and $C$ play fundamentally different roles in the analysis: $C^{-1}$ controls global geometric properties of the set $\Pi$, i.e. the diameter, while $s_0$ controls local behavior more associated to the system (\ref{cl}). This implies that there are essentially two important length scales within $\Pi$, namely the ``large scale'' $C^{-1}$ and the ``small scale'' $s_0$. These themes are repeated in the following three lemmas linking the geometry of $\Pi$ to quantitative estimates on $\tilde\eta$. First, we lower bound the magnitude of $\nabla \tilde\eta$ on $\bd\Pi$ which yields that $\nabla\tilde\eta$ is a non-vanishing outward normal vector on $\Pi$. Second, we show that $-\tilde\eta$ is comparable to the distance to $\bd\Pi$ globally inside $\Pi$. Third, we show that $C$ is essentially the curvature of $\Pi$.

\begin{lemma}\label{lem_nu}
For $C$ sufficiently large, $s_0$ sufficiently small, and any $\overline{u} \in \bd \Pi$,
\begin{equation}\label{eqn_lemnu_desired}
|\nabla \tilde\eta(\overline{u})| \gtrsim s_0.
\end{equation}
In particular, the vector field $\nu(\overline{u}):= \frac{\nabla \tilde\eta(\overline{u})}{|\nabla\tilde\eta(\overline{u})|}$ is the outward unit normal vector field for $\Pi$.
\end{lemma}

\begin{proof}
Let $\overline{u}\in \bd \Pi$ and take $v \in \bd \Pi$ such that $|\overline{u} - v| \sim C^{-1}$ (this is possible with a constant independent of $\overline{u}$ by Lemma \ref{lem_eta}). Then, since $\tilde\eta(v) = \tilde \eta(\overline{u}) = 0$, applying Lemma \ref{l2_rel_entropy_lemma} to the entropy $\tilde\eta$,
\begin{equation}\label{eqn_lemnu_eqn1}
C^\prime|v - \overline{u}|^2 \le \tilde\eta(v | \overline{u}) = -\nabla\tilde\eta(\overline{u}) \cdot (v - \overline{u}) \le |\nabla\tilde\eta(\overline{u})| |v - \overline{u}|,
\end{equation}
where $C^\prime$ is the constant from Lemma \ref{l2_rel_entropy_lemma}. Note that since the entropy $\tilde\eta$ depends on $s_0$ and $C$, so too does $C^\prime$. However, a cursory look at the proof of Lemma \ref{l2_rel_entropy_lemma} in \cite[Appendix A]{Leger2011} yields that the asymptotics for the constant $C^\prime$ as $C^\prime(C,s_0) \sim \inf_{u\in \Pi, |w| = 1} \nabla^2\tilde\eta(u)w \cdot w \sim Cs_0$ by Lemma \ref{lem_eta}. Equation (\ref{eqn_lemnu_eqn1}) combined with $C^\prime\sim Cs_0$ and $|\overline{u} - v| \sim C^{-1}$ yields the desired bound, (\ref{eqn_lemnu_desired}).
\end{proof}

\begin{lemma}\label{lem_dist}
For $C$ sufficiently large, $s_0$ sufficiently small, and $u\in \Pi_{C,s_0}$, we have
\begin{equation}
-\tilde\eta(u) \lesssim s_0 d(u,\bd\Pi).
\end{equation}
\end{lemma}

\begin{proof}
Since $\bd\Pi$ is compact, for $u\in \Pi$, there is a minimizer $\overline{u} \in \bd\Pi$ satisfying
\begin{equation}\label{eqn_lemdist_eqn1}
\min_{v\in \bd\Pi}|u-v| = |u-\overline{u}|.
\end{equation}
Note, since $\overline{u}\in\bd \Pi$, $\tilde\eta(\overline{u}) = 0$.
Therefore, by the fundamental theorem of calculus, Lemma \ref{lem_eta}, and (\ref{eqn_lemdist_eqn1}),
\begin{equation}
-\tilde\eta(u) = |\tilde\eta(u)| = \left|\tilde\eta(\overline{u}) + \int_{\overline{u}}^u \nabla\tilde\eta(w) \ dw \right| \lesssim s_0 |u-\overline{u}| = s_0 d(u,\bd\Pi).
\end{equation}
\end{proof}

Finally, the last of our structural lemmas should be read as a statement about the curvature of $\Pi$. Although we need only a bound on the curvature in the sense that the Lipschitz constant of the unit normal is essentially $C$, one can show this in several stronger senses than we do. For instance, it is straightforward to show that each principal curvature (with respect to the outward unit normal) satisfies $\kappa_i\sim -C$. One can show that $\Pi$ has positive curvature with a lower bound in the geometric sense that all sectional curvatures are bounded below by $C^2$. Lower bounds on the curvature, in this sense, are the central hypotheses in several major local-to-global type results in the comparison geometry, which reinforces our intuition that $C$ contains global information about the geometry of $\Pi$.

\begin{lemma}\label{lem_pi}
For $\nu$ the outward unit normal on $\bd\Pi$, for any $u,w\in \bd\Pi$,
\begin{equation}
|\nu(u) - \nu(w)| \sim C|u - w|,
\end{equation}
for $C$ sufficiently large and $s_0$ sufficiently small.
\end{lemma}

\begin{proof}
By Lemma \ref{lem_nu}, $\nu(u) = \frac{\nabla\tilde\eta(u)}{|\nabla\tilde\eta(u)|}$ and $\nu\in C^2(\bd\Pi;S^n)$. Therefore, we compute
\begin{equation}\label{eqn_lempi_eqn1}
\nabla \nu(u) = \frac{\nabla^2\tilde\eta(u)}{|\nabla\tilde\eta(u)|} - \frac{\nabla\tilde\eta(u)\tens \left(\nabla^2\tilde\eta(u)\nabla\tilde\eta(u)\right)}{|\nabla\tilde\eta(u)|^{3}}.
\end{equation}
By Lemma \ref{lem_eta} and Lemma \ref{lem_nu}, $|\nabla\tilde\eta(u)| \sim s_0$ on $\bd\Pi$ and $|\nabla^2\tilde\eta(u)| \lesssim Cs_0$. Therefore, (\ref{eqn_lempi_eqn1}) implies $|\nabla\nu| \lesssim C$ and hence $|\nu(u) - \nu(w)| \lesssim C|u-w|$ for all $C$ sufficiently large and $s_0$ sufficiently small.

The lower bound is more delicate. Let us fix $u\in \bd\Pi$. For $0 < \alpha < \frac{\pi}{2}$, define the cone centered at $u$ with direction $-\nu(u)$ and angle $\alpha$ via
\begin{equation}
P_\alpha(u) := \set{w \in \Nu \ \big | \ \cos(\alpha) \le -\nu(u) \cdot \frac{(w - u)}{|w - u|}}.
\end{equation}
For $w\in \bd\Pi$ and $w \notin P_\alpha(u)$, define $e = (w-u) - \left[(w-u) \cdot \nu(u)\right]\nu(u)$, so that $e$ is the projection of $w-u$ onto the tangent space of $\bd\Pi$ at $u$. Using $\nu(u) \cdot e = 0$, Lemma \ref{lem_nu}, and the fundamental theorem of calculus, we estimate
\begin{equation}\label{eqn_lempi_eqn2}
\begin{aligned}
|\nu(u) - \nu(w)||e| \ge \left|e\cdot (\nu(u) - \nu(w))\right| &\gtrsim \frac{1}{s_0}|\nabla\tilde\eta(u) -\nabla\tilde\eta(w)|\\
    &=\int_w^u e\cdot \nabla^2\tilde\eta(v) \ dv\\
    &\gtrsim C\int_0^1 e\cdot \nabla^2\eta(w + t(u-w))(u-w)\ dt\\
    &\gtrsim C\nabla^2\eta(w)(u - w) \cdot e.
\end{aligned}
\end{equation}
It remains to lower bound the quantity $A := \nabla^2\eta(w)(u-w) \cdot e$. We decompose $A$ into two parts via
\begin{equation}
A = \left(\nabla^2\eta(w)(u-w) \cdot (u-w)\right) - \left(\left[(u-w)\cdot\nu(u)\right]\nabla^2\eta(w)(u-w)\cdot\nu(u)\right).
\end{equation}
First, since $\eta$ is convex, $\nabla^2\eta(w)(u-w)\cdot (u-w) \gtrsim |u-w|^2$. Second, since $w\notin P_\alpha(u)$,
\begin{equation}\label{eqn_lempi_eqn3}
\left|\left[(u-w)\cdot\nu(u)\right]\nabla^2\eta(w)(u-w)\cdot\nu(u)\right|\lesssim |u-w|^2 \cos(\alpha).
\end{equation}
Therefore, we may pick $\alpha_0$ with $\frac{\pi}{2} -\alpha_0$ sufficiently small, independent of $C$ and $s_0$, such that if $w\notin P_{\alpha_0}(u)$, then $A \gtrsim |u-w|^2 \ge |u-w||e|$. From (\ref{eqn_lempi_eqn2}), it follows that
\begin{equation}\label{eqn_lempi_eqn4}
|\nu(u) - \nu(w)| \gtrsim C|u - w| \quad \text{for }u \in \bd\Pi\text{ and }w \notin P_{\alpha_0}(u).
\end{equation}
It remains to prove the estimate for $w\in P_{\alpha_0}(u)$.
Fix $w \in P_{\alpha_0}(u)$. Then, we note by the strict convexity of $\Pi$, the line $L(t) = u + t(w-u)$ crosses $\bd\Pi$ exactly twice, at $t=0$ and $t=1$. Since $-\nu(u)\cdot (w-u) > \cos(\alpha_0)|w - u|$, $L$ is inward-pointing at time $t = 0$ and so must be outward pointing at time $t = 1$. Using $w\in P_{\alpha_0}(u)$, we obtain
\begin{equation}
    \left[\nu(w) - \nu(u)\right] \cdot \frac{(w - u)}{|w - u|} > \cos(\alpha_0).
\end{equation}
As $\alpha_0$ was chosen independent of $C$ and $s_0$, $|\nu(u) - \nu(w)| \gtrsim 1$ for all $u\in \bd\Pi$ and $w\in P_{\alpha_0}(w)$. Therefore, by Lemma \ref{lem_eta}, we conclude
\begin{equation}\label{eqn_lempi_eqn5}
    |\nu(u) - \nu(w)| \gtrsim 1 \gtrsim C|u - w|  \quad \text{for }u \in \bd\Pi\text{ and }w \in P_{\alpha_0}(u).
\end{equation}
Combining (\ref{eqn_lempi_eqn4}) and (\ref{eqn_lempi_eqn5}) yields the desired lower bound.
\end{proof}

%%%%%%%%%%%%%%%%%%%%%%%%%%%%%%%%%%%%%%%%%%%%%%%%%%%%%%%%%%%%%%%%%%%%%%%%%%%%%%%%%%%%%%%%%%%%%%

\section{Proof of Theorem \ref{thm_main}}\label{secSam}

In this section, we prove that  Proposition \ref{prop_shock} implies Theorem \ref{thm_main}.
%\textbf{For this proof, we need the connection between $\sigma=\abs{u_L-u_R}$ and $s_0$ (are they equal? up to some constants?)}
We construct the shift function $t\to h(t)$ by solving an ODE in the Fillippov sense, in the spirit of  \cite{Leger2011_original, Leger2011}. To enforce the separation of waves \eqref{control_one_shock1020} and   \eqref{control_two_shock1020}, the proof uses the  framework developed first in \cite{move_entire_solution_system}. We specialize to the present situation and simplify the argument.
\vskip0.3cm

Note that we only need to show \Cref{thm_main} for a $1$-shock.
\Cref{thm_main} follows for an $n$-shock by considering the conservation law $u_t-f(u)_x=0$, because an $n$-shock for the system $u_t+f(u)_x=0$ is a $1$-shock for the system $u_t-f(u)_x=0$.
In particular, \eqref{control_one_shock1020} for $1$-shocks provides  \eqref{control_two_shock1020} for $n$-shocks, and  \eqref{control_a_one} for $1$-shocks provides \eqref{control_a_two} for $n$-shocks.
\vskip0.3cm
The proof of Theorem \ref{thm_main} for $1$-shocks is split into 5 steps.
\vskip0.3cm
\noindent
{\bf Step 1: Fixing the constants $C_1$, $\tilde{s}_0$, and $\eps$.} 
Let us fix $d\in \Nu$. From Assumption \ref{assum} (a), $\lambda_1(d)<\lambda_2(d)$.  Therefore, there exist $\alpha_1$ and $\eps_d>0$ such that the ball centered at $d$ with radius $\eps_d$ verifies $B(d,\eps_d)\subset \Nu$, and 
\begin{equation}\label{alpha}
\sup_{v\in B(d,\eps_d)}\lambda_1(v)<\alpha_1< \inf_{v\in B(d,\eps_d)}\lambda_2(v).
\end{equation}
Due to Lemma \ref{lem_eta}, if $C/2$ is big enough, then $\Pi_{C,s_0}\subset B(d,\eps_d/2)\subset \Nu$ for any $u_L, u_R$ such that  $|u_L-d|+|u_R-d|\leq \eps_d/8$ with  $s_0=|u_L-u_R|$.  From now on, we  consider a fixed constant $C_1>0$ verifying this property, Lemma \ref{lem_nu}, and Proposition \ref{prop_shock}. We then fix $\tilde{s}_0>0$ verifying   Proposition \ref{prop_shock} for these fixed $d$ and $C_1$. 
We then take 
\begin{equation}\label{eps}
\eps=\min(\tilde{s}_0, \eps_d/8).
\end{equation}
\vskip0.1cm
For any $u_L, u_R$ such that $|u_L-d|+|d-u_R|\leq \eps$, we denote $s_0=|u_L-u_R|$.
Then for any $a_1,a_2>0$ verifying \eqref{control_a_one}, We denote $C>0$ such that 
\begin{equation}\label{eqc}
\frac{a_1}{a_2}=1+Cs_0.
\end{equation}
From  \eqref{control_a_one}, $C$ verifies  $C_1/2<C<2C_1$, and so verifies Proposition  \ref{prop_shock}.  For the rest of the proof, the constant $C$ is proportional to the fixed constant $C_1$. However, we need to consider all possible strengths, $s_0$, of the shock $(u_L,u_R)$ for $0<s_0\leq \tilde{s}_0$.

\vskip0.4cm
\noindent
{\bf Step 2: Control of $\tilde{q}$ from $\tilde{\eta}$.} We remind the reader that the functionals $\tilde{\eta}$ and $\tilde{q}$, defined in \eqref{defn_eta}, depend on $C$ (which is now proportional to $C_1$ fixed), and on the $1$-shock $u_L,u_R$, especially through its strength $s_0=|u_L-u_R|$. For any $u\in \Nu_0\setminus \Pi_{C,s_0}$, we denote $d(u,\partial\Pi)=\inf\{|u-v| \ : \ v\in  \Pi_{C,s_0}  \}$.  This step is dedicated to the proof of the following lemma.

\begin{lemma}\label{lem:ext}% Let $d\in \Nu$. Then,
There exist  constants  $\delta>0$, and  $C^*>0$, such that  for any $u_L, u_R\in \Nu$ such that 
$|u_L-d|+|u_R-d|\leq \tilde{s}_0$,  %, if $u_L\in B_d(1/K)$, 
%$\Pi_{C,s_0}$ is contained in the ball of radius $\frac{K}{C}$ centered at $d$. Moreover, for any $u\in\{u\in\Nu| \tilde{q}(u)\leq 0\}\cap B_{{\frac{2K}{C}}}(d)$:
\begin{eqnarray}
%&&\tilde{\eta}(u)>0 \text{  for } u\in \Nu\setminus \Pi_{C,s_0},\\
\label{etadirection}
%&&|\nabla \tilde{\eta}(u)-\nabla\tilde{\eta}(\bar{u})|\leq \frac{1}{2}|\nabla\tilde{\eta}(\bar{u})|, \text{ for any } u\in \Nu\setminus \Pi_{C,s_0} \text{  at a distance from } \Pi_{C,s_0} \text{ smaller than } \eps,\\
&&\tilde{\eta}(u)\geq\delta s_0 \ \! d(u,\partial\Pi) \text{ for any  } u\in \Nu_0\setminus \Pi_{C,s_0},\\% \text{  for any } u\in \Nu, \text{ at a distance of  }  \Pi_{C,s_0} \text{ bigger than } \eps,\\
\label{control_q_ratio}
&&\abs{\tilde{q}(u)}\leq C^*\abs{\tilde{\eta}(u)}, \text{ whenever } u\in \Nu_0\setminus \Pi_{C,s_0} \text{ verifies } \tilde{q}(u)\leq 0.
\end{eqnarray}
\end{lemma}
\begin{proof} 
From Lemma \ref{lem_nu}, there exists $\delta>0$ (independent of $u_L,u_R, s_0, C$, since  $\tilde{s}_0$ and $C_1$ are now fixed), such that for any $v\in \partial \Pi_{C,s_0}$:
$$
\nabla\tilde{\eta}(v)\cdot\nu(v)\geq \delta s_0.
$$
For any $u\in \Nu_0$, denote $\bar{u}$ the orthogonal projection of $u$ on $\Pi_{C,s_0}$. Since $\Pi_{C,s_0}$ is smooth and strictly convex, whenever $u\notin \Pi_{C,s_0}$:
$$
u-\bar{u}=|u-\bar{u}|\nu(\bar{u}), \qquad |u-\bar{u}|=d(u,\partial\Pi).
$$
Consider $\phi(t)=\tilde{\eta}(\bar{u}+t(u-\bar{u}))$ for $t\in[0,1]$. We have 
$$
\phi'(0)=(u-\bar{u})\cdot\nabla\tilde{\eta}(\bar{u})\geq \delta s_0 d(u,\partial\Pi).
$$
Moreover $\phi$ is convex so for all $0\leq t\leq 1$:
$$
\phi(t)\geq \phi(0)+t\phi'(0)=\tilde{\eta}(\bar{u})+t \delta s_0 \ \! d(u,\partial\Pi).
$$
Since $\tilde{\eta}(\bar{u})=0$ ($\bar{u}\in \partial \Pi_{C,s_0}$), for $t=1$ we obtain:
$$
\tilde{\eta}(u)=\phi(1)\geq  \delta s_0  \ \! d(u,\partial\Pi).
$$
This proves \eqref{etadirection}.
\vskip0.1cm
Since $q$ is not $C^1$ up to the boundary of $\Nu$, to show \eqref{control_q_ratio} we need to consider the two cases: $u$ far from the boundary of $\Nu$ and $u$ close to the boundary of $\Nu$.  
\vskip0.1cm
First, to handle $u$ close to the boundary of $\Nu$, we recall $q$ and $f$ are continuous functions on the compact set $\overline{\Nu}_0$.
From the definitions \eqref{def_q}, \eqref{defn_eta}, there exists a constant $\tilde{C}>0$ such that for any $u\in \Nu_0$, $|\tilde{q}(u)|\leq \tilde{C}s_0$.
Since $\Pi_{C,s_0}\subset B_d(\eps_d/2)$, together with \eqref{etadirection}, this provides \eqref{control_q_ratio} for any $u\in \Nu_0\setminus  B_d(3\eps_d/4)$ for $C^*=4\tilde{C}/(\delta \eps_d)$.
\vskip0.1cm
Second, to handle $u$ far from the boundary of $\Nu$, we note $B_d(3\eps_d/4)$ is at least $\eps_d/4$ far away from $\R^n\setminus \Nu_0$, so $f'$ and $\nabla\eta$ are uniformly bounded on this ball. 
Consider any $u\in B_d(3\eps_d/4)\setminus \Pi_{C,s_0}$ such that $\tilde{q}(u)\leq 0$. Let $\bar{u}$ be the orthogonal projection of $u$ on 
$\partial  \Pi_{C,s_0}$.
From Proposition \ref{prop_cont}, together with \eqref{defn_Dcont}, we have $\tilde{q}(\bar{u})>0$. Therefore, for a possibly different constant $\tilde{C}$, still independent of $s_0$,  
\begin{equation}
\begin{aligned}\label{rel_entrop_eq}
|\tilde{q}(u)|=-\tilde{q}(u)\leq \tilde{q}(\bar{u})-\tilde{q}(u)=-\int\limits_0^1 \nabla\tilde q((1-t)\bar{u}+tu)\cdot(u-\bar{u})\,dt\\
\qquad =-\int\limits_0^1 \Big(\nabla\tilde\eta((1-t)\bar{u}+tu)f'((1-t)\bar{u}+tu)\Big)\cdot(u-\bar{u})\,dt,\\
\leq \tilde{C}s_0|u-\bar{u}|=\tilde{C}s_0 \ \! d(u,\partial \Pi).
\end{aligned}
\end{equation}
Together with \eqref{etadirection}, this provides \eqref{control_q_ratio}, in this case, with $C^*=\tilde{C}/\delta$. Taking $C^* = \max(\tilde{C}/\delta,4\tilde{C}/(\delta \eps_d))$ gives the result.
\end{proof}

\vskip0.3cm
\noindent
{\bf Step 3: Construction of the shift.} We now fix a shock $u_L,u_R$, and fix the weight $a_1,a_2>0$. From Step 1, this corresponds to fixing 
$s_0$ and $C$. From \eqref{eqc},
$$
A=\{ u \in \Nu_0 \ | \ a_1\eta(u  |  u_L) > a_2\eta(u  | u_R)\}=\Nu_0\setminus \overline{\Pi_{C,s_0}},
$$
which is an open set of $\Nu_0$. Therefore, the indicator function of this set is lower semi-continuous function on $\Nu_0$, and 
$
-\mathbbm{1}_{A}
$
is upper semi-continuous on $\Nu_0$.
 With the hypotheses of the theorem, these constants verify the hypotheses of Proposition \ref{prop_shock}.
Note that the smallest eigenvalue  $ \lambda_1(u)$ of $f'(u)$ may not be defined on $ \Nu_0\setminus \Nu$ (since $f$ is not differentiable on those points). These values of $u$ correspond to the vacuum in the case of the Euler equation. By definition, we extend this function on $\Nu_0$ by
$$
 \lambda_1(u)\coloneqq L \text{ for } u\in \Nu_0\setminus \Nu,
$$
where $L$ is defined by Assumption \ref{assum} (e). Note that  this function $\lambda_1$ is still  upper semi-continuous in $u$ on the whole domain $\Nu_0$. 
We define now the velocity functional:
\begin{equation}\label{define_V}
V(u) \coloneqq \lambda_1(u)-(C^*+2L)\mathbbm{1}_{\{ u \in \Nu_0 \ | \ a_1\eta(u  |  u_L) > a_2\eta(u  | u_R)\}}(u), \text{ for } u\in \Nu_0,
\end{equation}
Where $C_*>0$  is defined in Lemma \ref{lem:ext}. %to be chosen later in the proof (which will be made larger as we go).% and where $L$ is defined by Assumption \ref{assum} (e).
This function is still upper semi-continuous in $u$ on the whole domain $\Nu_0$. 
We want to solve the following ODE with discontinuous right hand side,
\begin{align}\label{eqn:h}
\begin{cases}
\dot{h}(t)=V(u(h(t),t)),\\
h(\bar{t})=x_0.
\end{cases}
\end{align}
The existence of an $h$ satisfying \eqref{eqn:h} in the sense of Filippov is the content of the following lemma, originally proved in \cite[Proposition 1]{Leger2011}.
%taken directly from \cite[Lemma 4.5]{move_entire_solution_system}, but originally proved in \cite[Proposition 1]{Leger2011}.

\begin{lemma}[Existence of Filippov flows]\label{fil_lemma}
Let $V\colon \Nuo\to\mathbb{R}$ be bounded and upper semi-continuous on $\Nuo$ and continuous on $U$ an open, full measure subset of $\Nuo$. Let $u\in \Sweak$, $x_0\in\mathbb{R}$, and $\bar t\in[0,\infty)$.
Then, we can solve \eqref{eqn:h} in the Filippov sense. That is, there exists a Lipschitz function $h\colon [\bar{t},\infty)\to\mathbb{R}$ such that 
\begin{align}
V_{min}(t) &\le \dot{h}(t) \le V_{max}(t),\label{fil_contain}\\
h(\bar t)&=x_0,\label{fil_prop2}\\
\rm{Lip}[h]&\leq \norm{V}_{L^\infty}\label{fil_prop1}
\end{align}
for almost every $t$, where $u_\pm\coloneqq u(h(t)\pm,t)$ and 
\begin{equation}
V_{max}(t) = \max(V(u^+(t)),V(u^-(t))) \quad \mathrm{and} \quad V_{min}(t) = \begin{cases} \min(V(u^+(t)),V(u^-(t))) &\mathrm{if }\ u^+(t),u^-(t) \in U\\ -\norm{V}_{L^\infty} &\mathrm{otherwise}\end{cases}.
\end{equation}
%and $I[a,b]$ denotes the closed interval with endpoints $a$ and $b$.
Furthermore, for almost every $t$,
\begin{align}
f(u_+)-f(u_-)&=\dot{h}(u_+-u_-),\label{fil_entrop1}\\
q(u_+)-q(u_-)&\leq \dot{h}(\eta(u_+)-\eta(u_-)),\label{fil_entrop2}
\end{align}
i.e., for almost every $t$, either $(u_-,u_+,\dot{h})$ is an entropic shock (for the entropy $\eta$) or $u_+=u_-$.
\end{lemma}
%In this article, we do not include a proof of  \eqref{fil_prop1}, \eqref{fil_prop2}, \eqref{fil_contain},\eqref{fil_entrop1} or \eqref{fil_entrop2}.
%The proof of \eqref{fil_prop1}, \eqref{fil_prop2} and \eqref{fil_contain} is very similar to the proof of Proposition 1 in \cite{Leger2011} (and see also \cite[Lemma 4.2]{move_entire_solution_system}). 

Note that for any function $h$ which is Lipschitz continuous, the results \eqref{fil_entrop1} and \eqref{fil_entrop2} are well-known when the solution $u$ is BV. However, %when $u$ is only known to have strong traces (as in the Strong Trace Property), 
 \eqref{fil_entrop1} and \eqref{fil_entrop2} are still valid under the strong trace property on $u$ (see \cite[Lemma 6]{Leger2011} or the appendix in \cite{Leger2011}). 
We can readily apply this lemma to our situation.  We will use its precise results below.%, since  the function $V$ defined by \eqref{define_V} is upper semi-continuous in $u$. Indeed,  the indicator functions of open sets are lower semi-continuous, while the negative of a lower-semicontinuous function is itself upper semi-continuous, $\lambda_1$ is upper semi-continuous, and the sum of upper semi-continuous functions is  upper semi-continuous.

\vskip0.3cm
\noindent
{\bf Step 4: Proof of \eqref{diss:shock}.}
Let us denote $u_\pm\coloneqq u(h(t)\pm,t)$. We will prove \eqref{diss:shock} for each fixed time $t$, by considering the four following cases
(which allow for $u_+=u_-$).
\begin{eqnarray*}
\text{Case 1.  } a_1\eta(u_- |  u_L) > a_2\eta(u_- | u_R) &\text{ and } &a_1\eta(u_+  |  u_L) > a_2\eta(u_+  | u_R),\\
\text{Case 2.  } a_1\eta(u_-  |  u_L) > a_2\eta(u_-  | u_R) &\text{ and } &a_1\eta(u_+  |  u_L) \leq a_2\eta(u_+  | u_R),\\
\text{Case 3.  } a_1\eta(u_-  |  u_L) \leq a_2\eta(u_-  | u_R) & \text{ and } &a_1\eta(u_+  |  u_L) > a_2\eta(u_+  | u_R),\\
\text{Case 4.  } a_1\eta(u_-  |  u_L) \leq a_2\eta(u_-  | u_R)  &\text{ and } &a_1\eta(u_+  |  u_L) \leq a_2\eta(u_+  | u_R).
\end{eqnarray*}

Since we need only prove \eqref{diss:shock} for almost every $t$, by Lemma \ref{fil_lemma}, we may neglect a null set of times and analyze only times for which the following dichotomy holds:
\begin{equation}
u_+(t) = u_-(t) \quad \mathrm{or} \quad (u_-(t),u_+(t),\dot{h}(t)) \text{ is an entropic shock.}
\end{equation}
Further, Lemma \ref{fil_lemma} implies 
\begin{equation}\label{control_dot_h1}
\dot{h}(t) \le \max(\lambda_1(u_+) - (C^* + 2L)\mathbbm{1}_{\Nuo \backslash \overline{\Pi}}(u_+), \lambda_1(u_-) - (C^* + 2L)\mathbbm{1}_{\Nuo \backslash \overline{\Pi}}(u_-))
\end{equation}
and for $u_+,u_- \notin \bd\Pi$ and $u_+,u_- \notin \Nuo\backslash\Nu$,
\begin{equation}\label{control_dot_h2}
\dot{h}(t) \ge \min(\lambda_1(u_+) - (C^* + 2L)\mathbbm{1}_{\Nuo \backslash \overline{\Pi}}(u_+), \lambda_1(u_-) - (C^* + 2L)\mathbbm{1}_{\Nuo \backslash \overline{\Pi}}(u_-))
\end{equation}

\begin{comment}
From \Cref{fil_lemma}, we get
\begin{equation}\label{control_dot_h}
\begin{array}{l}
\displaystyle{\dot{h}(t)\in I\Bigg[\lambda_1(u_+)-(C^*+2L)\mathbbm{1}_{\{ u \in \Nu \ | \ a_1\eta(u  |  u_L) > a_2\eta(u  | u_R)\}}(u_+),}\\[0.3cm]
\displaystyle{\qquad\qquad\qquad\qquad\lambda_1(u_-)-(C^*+2L)\mathbbm{1}_{\{ u \in \Nu \ | \ a_1\eta(u  |  u_L) > a_2\eta(u  | u_R)\}}(u_-)\Bigg].}
\end{array}
\end{equation}
\end{comment}

\noindent
\uline{\emph{Case 1}}: This case corresponds to $u_+,u_-\notin \overline{\Pi}$. In this case, due to \eqref{control_dot_h1} we have
\begin{align*}
\dot{h}(t) \le -(C^* + 2L) +  \max(\lambda_1(u_+),\lambda_1(u_-)).
\end{align*}
Thus, by the choice of constants,
%Thus, by making $C_*$ sufficiently large we can ensure that in this case (\emph{Case 1}) we must have that 
\begin{align}\label{h_fast}
\dot{h}(t) < -L \leq \inf_{u\in \Nu_0}\lambda_1(u),\qquad \dot{h}(t)\leq -C^*.
\end{align}
The first inequality insures that $(u_+(t),u_-(t),\dot{h}(t))$ cannot be an entropic shock, since $\dot{h}(t) < \lambda_1(u_-(t))$ contradicts Assumption \ref{assum} (f). Thus, we must have $u_+ = u_-$.
Define $v\coloneqq u_+= u_-$. Since $v\notin \overline{\Pi}$, $\tilde{\eta}(v)\geq 0$. Therefore, using \eqref{eqc}, and the second inequality of \eqref{h_fast}, we find
\begin{equation}
\begin{aligned}
a_2\left[q(v;u_R)-\dot{h}(t) \ \eta(v|u_R)\right] -a_1\left[q(v;u_L)-\dot{h}(t) \ \eta(v|u_L)\right]  &= a_2\left(-\tilde{q}(v)+\dot{h}(t)\tilde\eta(v)\right)\\
	&\leq a_2\left(-\tilde{q}(v)-C^*\tilde\eta(v)\right).
\end{aligned}
\end{equation}
Since $a_2>0$, on one hand, if $\tilde{q}(v)\geq0$, the right hand side is non-positive. On the other hand, if $\tilde{q}(v)\leq0$, the choice of $C^*$ in \eqref{control_q_ratio} enforces the non-positivity. This proves \eqref{diss:shock} in Case 1.
\vskip0.3cm
\noindent\uline{\emph{Case 2}}: This case corresponds to $u_- \notin \overline{\Pi}$ and $u_+\in \overline{\Pi}$ and consequently, $u_- \neq u_+$ so that $(u_-(t),u_+(t),\dot{h}(t))$ is an entropic shock.
As in Case 1, due to \eqref{control_dot_h1} we have
\begin{align}
\dot{h}(t)\le \max(\lambda_1(u_+),\lambda_1(u_-) - (C^* + 2L)) < \max(\lambda_1(u_+),-L) \le \lambda_1(u_+)
%I \Big[\lambda_1(u_+),\lambda_1(u_-)-(C^*+2L)\Big].
\end{align}
%Thus, we have that $\dot{h}(t)< -L\leq \lambda_1(u_+)$.
However, this contradicts Assumption \ref{assum} (f). Thus, we conclude that this case %(\emph{Case 2})
cannot occur.

\vskip0.3cm
\noindent\uline{\emph{Case 3}}: This case corresponds to $u_- \in \overline{\Pi}$ and $u_+ \notin \overline{\Pi}$, so that once again $(u_-(t),u_+(t),\dot{h}(t))$ is an entropic shock. Consequently, \eqref{control_dot_h1} implies
\begin{equation}
\dot{h}(t) \le \max(\lambda_1(u_+)-(C^* + 2L),\lambda_1(u_-)) \le \lambda_1(u_-),
\end{equation}
and combined with Assumption \ref{assum} (g), this means $(u_-,u_+,\dot{h}(t))=(u_-,u_+,\sigma_\pm)$ is an entropic 1-shock with $u_- \in \overline{\Pi}$. Therefore, from Definition \ref{defn_DRH} and  \Cref{prop_shock}, we obtain
\begin{equation}
a_2\left[q(u_+;u_R)-\dot{h}(t) \ \eta(u_+|u_R)\right] -a_1\left[q(u_-;u_L)-\dot{h}(t) \ \eta(u_-|u_L)\right] =a_2D_{RH}(u_-,u_+,\sigma_\pm)\leq 0.
\end{equation}
We conclude \eqref{diss:shock} holds in this case. 
\vskip0.3cm

\noindent\uline{\emph{Case 4}}: This case corresponds to $u_-,u_+ \in \overline{\Pi}$. In this case, \eqref{control_dot_h1} implies
\begin{align}
\dot{h}(t) \le \max(\lambda_1(u_+),\lambda_1(u_-)).
\end{align}
Now, suppose first that $u_+ = u_- = v$. Then, either $v\in\bd\Pi$, in which case $\tilde\eta(v) = 0$, or \eqref{control_dot_h2} implies $\dot{h}(t) = \lambda_1(v)$. In either case, we find by Proposition \ref{prop_cont},
\begin{equation}
\begin{aligned}
a_2\left[q(v;u_R)-\dot{h}(t) \ \eta(v|u_R)\right] -a_1\left[q(v;u_L)-\dot{h}(t) \ \eta(v|u_L)\right] &= a_2\left(\tilde q(v) - \dot{h}(t)\tilde\eta(v)\right)\\
	&= a_2 D_{cont}(v)\leq 0.
\end{aligned}
\end{equation}
Second, suppose $u_- \neq u_+$, so that once more we deduce from Assumption \ref{assum} (g) that $(u_+,u_-,\dot{h}(t))$ is a 1-shock. We then use Proposition \ref{prop_shock} as in Case 3 to conclude \eqref{diss:shock} holds. 
This concludes the casework.
\vskip0.3cm
{\bf Step 5: Proof of  \eqref{control_one_shock1020}.}

To obtain \eqref{control_one_shock1020}, we note that in all four cases from Step 3, either $\dot{h}(t)\leq -L$ or $\dot{h}(t) \leq \sup_{v\in \Pi} \lambda_1(v)$. Therefore, combined with \eqref{fil_prop1}, for almost all $t>0$:
\begin{equation}
-C^*-3L\le -\|V\|_{L^\infty} \le \dot{h}(t) \le \sup_{v\in \Pi_{C,s_0}} \lambda_1(v).
\end{equation}
From Step 1, $\Pi_{C,s_0} \subset B(d,\eps_d/2)$. Therefore  \eqref{control_one_shock1020} is a consequence of \eqref{alpha} and \eqref{eps} with $\hat{\lambda} =2(C^*+3L)$.

%%%%%%%%%%%%%%%%%%%%%%%%%%%%%%%%%%%%%%%%%%%%%%%%%%%%%%%%%%%%%%%%%%%%%%%

\section{Proof of Proposition \ref{prop_cont}}\label{secWill1}

\subsection{Proof Sketch}

In this section, we prove Proposition \ref{prop_cont}. For this section, we fix $(u_L,u_R,\sigma_{LR})$, a single entropic $1$-shock of the system (\ref{cl}) with $u_R = \S^1_{u_L}(s_0)$. We moreover assume that $u_L \in \Nu$, $s_0$ is sufficiently small so that $u_R \in \Nu$, and $C$ is sufficiently large so that $\Pi_{C,s_0}$ is compactly contained in $\Nu$ (this is possible by Lemma \ref{lem_eta}). We recall that the continuous entropy dissipation function $D_{cont}$ is given by
\begin{equation}
D_{cont}(u) = \left[q(u ; u_R) - \lambda^1(u)\eta(u | u_R)\right] - (1+Cs_0)\left[q(u ; u_L) - \lambda^1(u)\eta(u | u_L)\right] = -\tilde q(u) + \lambda^1(u)\tilde\eta(u),
\end{equation}
where we recall $\tilde q$ and $\tilde \eta$ are given by
\begin{equation}
\tilde\eta(u) = (1+Cs_0)\eta(u|u_L) - \eta(u | u_R) \qquad \text{and}\qquad \tilde q(u) = (1+Cs_0)q(u;u_L) - q(u;u_R).
\end{equation}
Finally, we recall $\Pi_{C,s_0} = \set{u \ | \ \tilde\eta(u) < 0}$.
Now, to prove Proposition \ref{prop_cont}, it suffices to show $D_{cont}(u) \lesssim -s_0^3$ for $C$ sufficiently large, $s_0$ sufficiently small, and any $u\in \Pi_{C,s_0}$.

The proof is motivated by the scalar case in which the uniform negativity in Proposition \ref{prop_cont} is computed explicitly in the case of Burger's equation (see \cite[p.~9-10]{serre_vasseur} and also the proof of Lemma 3.3 in \cite{scalar_move_entire_solution}). As the $1$-characteristic field is genuinely nonlinear (Assumption \ref{assum} (b)), we expect the same dissipation rate for $u \in \S^1_{u_L}$. Therefore, the underlying principle of the proof is to localize the problem to $u\in\S^1_{u_L}$. We show that if the estimate holds near $\S^1_{u_L}$, where we can compute the dissipation rate nearly explicitly, then the estimate holds on all of $\Pi_{C,s_0}$. 
%In particular, \red{by the computations in citation} we expect that states $u \in \Pi$ with $u$ along the shock curve $\S^1_{u_L}$ maximize the continuous entropy dissipation, and control the entropy dissipation away from the shock curve.
More precisely, we proceed in three steps:
\begin{itemize}
\item \uline{Step 1}: We characterize maxima of $D_{cont}$ in $\Pi_{C,s_0}$ and show that for large enough $C$ there is a unique maximum, $u^*$, for $D_{cont}$ on $\Pi_{C,s_0}$.
\item \uline{Step 2}: We compute $D_{cont}(u)$ for states $u$ along the shock curve $\S^1_{u_L}$ and show $D_{cont}(u)$ satisfies the entropy dissipation estimate claimed in Proposition \ref{prop_cont}.
\item \uline{Step 3}: We show that the maximum $u^*$ is sufficiently close to the shock curve to conclude that $u^*$ also satisfies the desired estimate.
\end{itemize}
Before beginning Step 1, we note that because we have already chosen $C$ and $s_0$ so that $\Pi$ is compactly contained within $\Nu$, the system is more regular on $\Pi_{C,s_0}$ than on $\Nuo$. By the regularity assumptions on the system, $f \in C^4(\Pi)$ and $r_i,l^i,\lambda^i,\eta,q \in C^3(\Pi)$, which allows us to use classical techniques of differential calculus to analyze $D_{cont}$ on $\Pi$.

%%%%%%%%%%%%%%%%%%%%%%%%%%%%%%%%%%%%%%%%%%%%%%%%%%%%%%%%%%%%%%%%%%%%%%%%%%%%%%%%%%%%%%%%%%%%%%%%%%%%%%%

\subsection{Step 1: Characterization and Uniqueness of Maxima}

%Since our system is strictly hyperbolic, we obtain $f^\prime(u)$ has right eigenvectors $r_i(u)$ and left eigenvectors $l_i(u)$ corresponding to the eigenvalues $\lambda^i(u)$. Moreover, $r_i$ and $l_i$ can be chosen to vary smoothly in $u$ satisfying $|r_i(u)| = |l_j| = 1$, and  $r_i(u) \cdot l_j(u) = \delta_{ij}$, where $\delta_{ij}$ is the Kroenecker delta. Finally, since we imagine that $r_1(u_L)$ points towards $u_R$, i.e. $r_1$ points in the direction of the shock curve, we have $\lambda^1(u)$ is decreasing along $r_1$ and
%$$r_1 \cdot \nabla\lambda^1 < 0.$$
In this step, we study maxima of $D_{cont}$ in the set $\Pi_{C,s_0}$ and show that taking $C$ sufficiently large yields a unique maximum.
\begin{proposition}\label{prop_max}
For any $C$ sufficiently large and $s_0$ sufficiently small, there is a unique maximum $u^*$ of $D_{cont}$ on $\Pi_{C,s_0}$. Moreover, $u^*$ is the unique point in $\bd\Pi_{C,s_0}$ satisfying the condition
\begin{equation}
\nu(u^*) \parallel l^1(u^*) \qquad {and} \qquad r_1(u^*)\text{ is outward pointing,}
\end{equation}
where $\nu(u)$ is the outward unit normal to $\bd\Pi_{C,s_0}$.
\end{proposition}

%We prove Proposition \ref{prop_max} through the forthcoming sequence of lemmas.
We begin by showing that a maxima cannot occur in the interior of $\Pi$.
\begin{lemma}\label{lem1}
Say $u$ is a critical point of $D_{cont}$. Then, $u\in \bd \Pi$.
\end{lemma}

\begin{proof}
We compute the directional derivative of $D_{cont}$ along the integral curve emanating from $u$ in the $r_1$ direction.
Defining $u(t)$ as the solution to the ODE,
\begin{equation}
\begin{aligned}
\dot u(t) &= r_1(u(t))\\
u(0) &= u.
\end{aligned}
\end{equation}
Then, using $\nabla\tilde q = \nabla\tilde\eta f^\prime$ and $r_1(u(t))$ is a right eigenvector of $f^\prime(u(t))$ associated to the eigenvalue $\lambda^1(u(t))$, we compute
\begin{equation}\label{eqn_Dcont_r1}
\begin{aligned}
\frac{d}{dt} D_{cont}(u(t)) &= \left[-\nabla \tilde q(u(t)) +  \lambda^1(u(t))\nabla\tilde\eta(u(t)) + \tilde\eta(u(t))\nabla \lambda^1(u(t)) \right] \cdot r_1(u(t))\\
	&=\tilde \eta(u(t))\left[\nabla\lambda^1(u(t))\cdot r_1(u(t))\right].
\end{aligned}
\end{equation}
Finally, evaluating at $t = 0$ and recalling that $u$ is a critical point, we obtain
\begin{equation}
\tilde \eta(u)\left[\nabla\lambda^1(u)\cdot r_1(u)\right] = 0.
\end{equation}
Since the first characteristic field is genuinely nonlinear by Assumption \ref{assum} (b), we must have $\tilde \eta(u) = 0$.
\end{proof}

\begin{lemma}\label{lem2}
Say $u^*$ is a maxima of $D_{cont}(u)$ on $\Pi_{C,s_0}$. Then,
\begin{equation}
\nabla \tilde\eta(u^*) \parallel l^1(u^*) \qquad {and} \qquad r_1(u^*)\text{ points outwards.}
\end{equation}
\end{lemma}

\begin{proof}
By Lemma \ref{lem1}, $\tilde\eta(u^*) = 0$ and $u^*\in \bd\Pi$. Therefore, by Lagrange's theorem, and that $\nabla \tilde\eta$ is a non-vanishing (by Lemma \ref{lem_nu}) normal vector field on $\Pi$, we must have
\begin{equation}
\nabla D_{cont}(u^*) \parallel \nabla\tilde\eta(u^*).
\end{equation}
Thus, we compute 
\begin{equation}\label{eqn_Dcont}
\begin{aligned}
\nabla D_{cont}(u^*) &= -\nabla \tilde q(u^*) +  \lambda^1(u^*)\nabla\tilde\eta(u^*) + \tilde\eta(u^*)\nabla \lambda^1(u^*)\\
	&= \lambda^1(u^*)\nabla\tilde\eta(u^*) - \nabla\tilde\eta(u^*)f^\prime(u^*).
\end{aligned}
\end{equation}
Therefore, $\nabla\tilde\eta(u^*)$ must be a left eigenvector of $f^\prime(u^*)$ and $\nabla\tilde\eta(u^*)$ is parallel to $l^i(u^*)$ for some $1\le i \le n$.
Suppose for the sake of contradiction that $\nabla\tilde\eta(u^*)$ is parallel to $l^j(u^*)$ for some $j \neq 1$. 
\begin{description}
\item[Case 1]
Suppose $l^j(u^*)$ is inward-pointing, i.e.
$$\nabla\tilde\eta(u^*) \cdot l^j(u^*) < 0.$$
Then, we define $u(t)$ to be the integral curve in the $l^j$ direction emanating from $u^*$. Thus, because $l^j(u^*)$ is inward pointing, $u(t) \in \Pi$ for some small amount of time. Since $u(0) = u^*$ is a maximum of $D_{cont}$ on $\Pi$, we have
\begin{equation}
D_{cont}(u(t))\le  D_{cont}(u^*) \qquad \text{and} \qquad \frac{d}{dt}\bigg|_{t=0} D_{cont}(u(t)) \le 0.
\end{equation}
Now, using (\ref{eqn_Dcont}), we compute
\begin{equation}
\frac{d}{dt}\bigg|_{t=0} D_{cont}(u(t)) = \left[\nabla\tilde\eta(u(t))\left[\lambda^1(u(t))I - f^\prime(u(t))\right] + \tilde\eta(u(t))\nabla \lambda^1(u(t)) \right] \cdot l^j(u(t)).
\end{equation}
Evaluating at $t = 0$ and using $\tilde\eta(u^*) = 0$ and $\nabla\tilde\eta(u^*)$ is parallel to $l^j(u^*)$, we obtain
\begin{equation}
\frac{d}{dt}\bigg|_{t=0} D_{cont}(u(t)) = \left[\lambda^1(u^*) - \lambda^j(u^*)\right]\left[\nabla\tilde\eta(u^*) \cdot l^j(u^*)\right]
\end{equation}
However, since $\lambda^1(u^*) < \lambda^j(u^*)$ by Assumption \ref{assum} (a) and $l^j(u^*)$ is inward pointing, we obtain
\begin{equation}
\frac{d}{dt} D_{cont}(u(t)) > 0
\end{equation}
and we obtain a contradiction.

\item[Case 2]
Suppose $l^j(u^*)$ is outward-pointing, i.e.
$$\nabla\tilde\eta(u^*) \cdot l^j(u^*) > 0.$$
Then, reasoning as in Case 1, with $l^j$ replaced by $-l^j$, we take $u(t)$ the integral curve in the $-l^j$ direction. Since $-l^j(u^*)$ is inward pointing and $u^*$ is a maximum of $D_{cont}$ on $\Pi$,
\begin{equation}
\frac{d}{dt}\bigg|_{t=0} D_{cont}(u(t)) \le 0.
\end{equation}
As before, by explicit computation, we obtain,
\begin{equation}
\frac{d}{dt}\bigg|_{t=0} D_{cont}(u(t)) = -\left[\lambda^1(u^*) - \lambda^j(u^*)\right]\left[\nabla\tilde\eta(u^*) \cdot l^j(u^*)\right] > 0,
\end{equation}
again a contradiction.

\item[Case 3]
Suppose $l^j(u^*) \cdot \nabla\tilde\eta(u^*) = 0$. Then, $l^i(u^*) \cdot \nabla\tilde\eta(u^*) = 0$ for each $1\le i \le n$ and $\nabla\tilde\eta(u^*) = 0$, contradicting Lemma \ref{lem_nu}.
\end{description}
In all cases, we obtain a contradiction. Therefore, we conclude that $\nabla\tilde\eta(u^*)$ is parallel to $l^1(u^*)$. Finally, for the sake of contradiction, suppose that $r_1(u^*)$ is inward pointing. As in the proof of Lemma \ref{lem1}, taking $u(t)$ the integral curve in the direction of $r_1$ emanating from $u^*$,
\begin{equation}
\frac{d}{dt}D_{cont}(u(t)) = \tilde\eta(u(t))\left[\nabla\lambda^1(u(t))\cdot r_1(u(t))\right] 
\end{equation}
and $\frac{d}{dt} D_{cont}(u(0)) = 0$. Now, if $r_1(u^*)$ is inward-pointing, since $u^*$ is a maximum for $D_{cont}$, we must have
\begin{equation}
\frac{d^2}{dt^2}\bigg|_{t=0}D_{cont}(u(t)) \le 0.
\end{equation}
However, by explicit computation, and using $\tilde\eta(u^*) = 0$ once again,
\begin{equation}
\frac{d^2}{dt^2}\bigg|_{t=0} D_{cont}(u(t)) = \left[\nabla\tilde\eta(u^*)\cdot r_1(u^*)\right]\left[\nabla\lambda^1(u^*)\cdot r_1(u^*)\right].
\end{equation}
Since we have chosen $r_1$ to satisfy $\nabla\lambda^1(u) \cdot r_1(u) < 0$, (see Lemma \ref{lem_normalization}) it follows that if $r_1(u^*)$ is inward-pointing, $$\frac{d^2}{dt^2}\big|_{t=0} D_{cont}(u(t)) > 0,$$ a contradiction.
\end{proof}

\begin{lemma}\label{lem3}
Suppose $u^* \in \bd\Pi_{C,s_0}$ satisfies
\begin{equation}
\nabla \tilde\eta(u^*) \parallel l^1(u^*) \qquad {and} \qquad r_1(u^*)\text{ points outwards.}
\end{equation}
Then, for any $C$ sufficiently large and $s_0$ sufficiently small, $u^*$ is a maxima for $D_{cont}(u)$ on $\Pi_{C,s_0}$. Furthermore, for any $v\in \R^n$ with $v = \sum_{i = 1}^n v_i r_i(u^*)$
\begin{equation}\label{eqn_lem3_desired}
\nabla^2D_{cont}(u^*)v \cdot v \lesssim -s_0v_1^2 - Cs_0\left[\sum_{i=2}^n v_i^2\right]
\end{equation}
\end{lemma}

\begin{proof}
Since $\tilde\eta$ is an entropy for (\ref{cl}), we find that
\begin{equation}\label{eqn_lem3_eqn1}
\nabla D_{cont}(u) = -\nabla\tilde\eta(u)\left[f^\prime(u) - \lambda^1(u)I\right] + \tilde\eta(u)\nabla\lambda^1(u). 
\end{equation}
Since $\tilde\eta(u^*) = 0$ and $\nabla\tilde\eta(u^*)$ is a left eigenvector for $f^\prime(u^*)$ with eigenvalue $\lambda^1(u^*)$, it follows that $u^*$ is a critical point of $D_{cont}$.

Next, we note that for $u(t)$ the integral curve emanating from $u^*$ in the $r_1(u)$ direction, $\nabla D_{cont}(u^*) = 0$ implies
\begin{equation}\label{eqn_lem3_eqn2}
\frac{d^2}{dt^2}\bigg|_{t=0} D_{cont}(u(t)) = \frac{d}{dt}\bigg|_{t=0} \nabla D_{cont}(u(t))\cdot r_1(u(t)) = \nabla^2 D_{cont}(u^*)r_1(u^*) \cdot r_1(u^*).
\end{equation}
Therefore, we compute the left hand side of (\ref{eqn_lem3_eqn2}). Since $\frac{d}{dt} D_{cont}(u(t)) = \tilde\eta(u(t))\nabla\lambda^1(u(t)) \cdot r_1(u(t))$, taking a second derivative and using $\nabla\tilde\eta(u^*) \parallel l^1(u^*)$ and $\tilde\eta(u^*) = 0$ yields,
\begin{equation}\label{eqn_lem3_eqn3}
\frac{d^2}{dt^2}\bigg|_{t=0} D_{cont}(u(t)) = (\nabla\tilde\eta(u^*) \cdot r_1(u^*))(\nabla\lambda^1(u^*)\cdot r_1(u^*)) \sim -s_0
\end{equation}
because $r_1(u^*)$ is outward pointing, $\nabla\lambda^1(u)\cdot r_1(u) \sim -1$ by Assumption \ref{assum}(b), and $|\nabla\tilde\eta(u)|\sim s_0$ for $u\in \bd\Pi$ by Lemma \ref{lem_nu}. Together (\ref{eqn_lem3_eqn2}) and (\ref{eqn_lem3_eqn3}) imply $\nabla^2D_{cont}(u^*)r_1(u^*) \cdot r_1(u^*) \lesssim -s_0$.

On the other hand, differentiating (\ref{eqn_lem3_eqn1}) once more yields the formula
\begin{equation}\label{eqn_lem3_eqn4}
\begin{aligned}
\nabla^2 D_{cont}(u) &= -\nabla^2\tilde\eta(u)\left[f^\prime(u) - \lambda^1(u)I\right] - \nabla\tilde\eta(u)\left[f^{\prime\prime}(u) - I \tens \nabla\lambda^1(u)\right]\\
	&\quad + \tilde\eta(u)\nabla^2\lambda^1(u) + \nabla\lambda^1(u) \tens \nabla\tilde\eta(u).
\end{aligned}
\end{equation}
Therefore, evaluating (\ref{eqn_lem3_eqn4}) at $u^*$ using $\tilde\eta(u^*) = 0$ and $|\nabla\tilde\eta(u^*)| \lesssim s_0$ by Lemma \ref{lem_eta},
\begin{equation}\label{eqn_lem3_eqn5}
\nabla^2D_{cont}(u^*) = -Cs_0\nabla^2\eta(u^*)\left[f^\prime(u^*) - \lambda^1(u^*)I\right] + \bigO(s_0).
\end{equation}
Thus, right multiplying by $r_i(u^*)$ for $i \neq 1$ and left multiplying by $r_j(u^*)$, (\ref{eqn_lem3_eqn5}) implies
\begin{equation}\label{eqn_lem3_eqn6}
\nabla^2 D_{cont}(u^*)r_i(u^*) \cdot r_j(u^*) = -Cs_0(\lambda^i(u^*) - \lambda^1(u^*))\nabla^2\eta(u^*)r_i(u^*) \cdot r_j(u^*) + \bigO(s_0)
\end{equation}
By Assumption \ref{assum}(a), $\lambda^i(u^*) - \lambda^1(u^*) \gtrsim 1$. Also, since $\nabla^2\eta(u)r_i(u) \parallel l^i(u)$, by Lemma \ref{lem_normalization}, $\nabla^2\eta(u^*)r_i(u^*) \cdot r_j(u^*) \sim \delta_{ij}$, where $\delta_{ij}$ denotes the Kronecker delta. So, for $i\neq 1$, (\ref{eqn_lem3_eqn6}) yields
\begin{equation}\label{eqn_lem3_eqn7}
\nabla^2 D_{cont}(u^*)r_i(u^*) \cdot r_j(u^*) \lesssim -Cs_0\delta_{ij} + \bigO(s_0).
\end{equation}
Therefore, for $C$ sufficiently large and $s_0$ sufficiently small and any $1 \le i \le n$, $\nabla^2D_{cont}(u^*)r_i(u^*)\cdot r_i(u^*) < 0$. We conclude (\ref{eqn_lem3_desired}) holds for any $v$ when only one $v_i$ is nonzero. We now prove (\ref{eqn_lem3_desired}) in general, which implies $u^*$ is a maximum. Expanding the left hand side of (\ref{eqn_lem3_desired}) by linearity and using (\ref{eqn_lem3_eqn7}), there exist universal constants $K_1$, $K_2$, and $K_3$ such that
\begin{equation}\label{eqn_lem3_eqn8}
\nabla^2 D_{cont}(u^*)v \cdot v \le -K_1s_0v_1^2 - K_2Cs_0\left[\sum_{i=2}^n v_i^2\right] + K_3s_0\sum_{i\neq j} |v_i v_j|.
\end{equation}
Note, for $C$ sufficiently large, by Cauchy-Schwarz,
\begin{equation}\label{eqn_lem3_eqn9}
K_3s_0\sum_{i,j \neq 1, i\neq j} |v_i v_j| \le \frac{K_2}{4}Cs_0\left[\sum_{i=2}^n v_i^2\right].
\end{equation}
Also, again using Cauchy-Schwarz, for any $\epsilon > 0$,
\begin{equation}\label{eqn_lem3_eqn10}
K_3s_0\sum_{j\neq 1} |v_1 v_j| \le K_3s_0 \eps v_1^2 + \frac{K_3}{4\eps}s_0 \left[\sum_{j=2}^nv_j^2\right].
\end{equation}
Taking $\eps$ sufficiently small so that $K_3\eps < \frac{K_1}{2}$, (\ref{eqn_lem3_eqn8}), (\ref{eqn_lem3_eqn9}), and (\ref{eqn_lem3_eqn10}) together imply (\ref{eqn_lem3_desired}) for sufficiently large $C$ and sufficiently small $s_0$.
\end{proof}

\begin{lemma}\label{lem4}
For any $C$ sufficiently large and $s_0$ sufficiently small, there exists a unique point $u^*$ contained in $\bd\Pi_{C,s_0}$ satisfying the nonlinear system of constraints,
\begin{equation}\label{eqn_max_cond}
\nabla\tilde\eta(u^*)\parallel l^1(u^*) \qquad \text{and} \qquad r_1(u^*) \cdot \nabla\tilde\eta(u^*) > 0.
\end{equation}
\end{lemma}

\begin{proof}
Say $u_1$ and $u_2$ both satisfy (\ref{eqn_max_cond}) and define the outward unit normal $\nu(u) = \frac{\nabla\tilde\eta(u)}{|\nabla\tilde\eta(u)|}$. Then, because $r_1(u)\cdot l^1(u) > 0$ by our choice of normalization in Lemma \ref{lem_normalization}, $u_1$ and $u_2$ satisfy
\begin{equation}
\nu(u_1) = l^1(u_1) \qquad \text{and} \qquad \nu(u_2) = l^1(u_2).
\end{equation}
Therefore, by Lemma \ref{lem_pi},
\begin{equation}
|u_1 - u_2| \sim | l^1(u_1) - l^1(u_2)| = |\nu(u_1) - \nu(u_2)| \sim C|u_1 - u_2|.
\end{equation}
Thus, there is a universal constant $K$ such that for any $C$ sufficiently large,
\begin{equation}
C|u_1 - u_2| \le K|u_1 - u_2|.
\end{equation}
For $C$ larger than $K$, this can hold only if $u_1 = u_2$.
\end{proof}

\begin{flushleft}
\uline{\bf{Proof of Proposition \ref{prop_max}}}
\end{flushleft}
%\begin{proof}[Proof of Proposition \ref{prop_max}]
Combining Lemma \ref{lem1}, Lemma \ref{lem2}, and Lemma \ref{lem3}, we see that $u^*\in \Pi$ is a maxima of $D_{cont}$ if and only if $u^* \in \bd \Pi$ with $l^1(u^*) \parallel \nabla\tilde\eta(u^*)$ and $r_1(u^*)$ is outward-pointing. Therefore, Lemma \ref{lem4} guarantees that there is a unique maximum in $\Pi$ for sufficiently large $C$ and sufficiently small $s_0$.
%\end{proof}

%%%%%%%%%%%%%%%%%%%%%%%%%%%%%%%%%%%%%%%%%%%%%%%%%%%%%%%%%%%%%%%%%%%%%%%%%%%%%%%%%%%%%%%%%%%%%%%%%%%%%%%

\subsection{Step 2: Dissipation Rate Along the Shock Curve \texorpdfstring{$\S^1_{u_L}$}{S1uL}}

The main proposition of this step is the calculation of $D_{cont}$ along the shock curve connecting $u_L$ and $u_R$. The uniformly negative cubic entropy dissipation is sharp as it agrees with the known scalar dissipation rate (i.e. Burger's). The dissipation rate for Burger's follows immediately from \cite[p.~9-10]{serre_vasseur} and also the proof of Lemma 3.3 in \cite{scalar_move_entire_solution}.
\begin{proposition}\label{prop_shock_explicit}
Let $u_0$ denote the intersection point of the shock curve $\S^1_{u_L}$ and $\bd\Pi_{C,s_0}$. Then, for all $C$ sufficiently large and $s_0$ sufficiently small,
\begin{equation}
D_{cont}(u_0) \lesssim -s_0^3.
\end{equation}
\end{proposition}

To compute $D_{cont}(u_0)$, we introduce some auxiliary points and functions. First, we define $u_r$ to be the first intersection of $\Pi$ with the integral curve starting from $u_R$ along the $-r_1$ direction. Second, we define $u_l$ to be the first intersection of $\Pi$ with the integral curve starting from $u_L$ in the $r_1$ direction. Finally, we define $D_R(u)$ and $D_L(u)$ via
\begin{equation}
D_R(u) =  q(u ; u_R) - \lambda^1(u)\eta(u | u_R) \qquad \text{and} \qquad D_L(u) = -q(u ; u_L) + \lambda^1(u)\eta(u | u_L).
\end{equation}
Thus, each of $D_R$ and $D_L$ are roughly half of the entropy dissipation, in the sense that, by the definition of $D_{cont}$ in (\ref{eqn_Dcont}), 
\begin{equation}
D_{cont}(u) = D_R(u) + (1+Cs_0)D_L(u).
\end{equation}

We now prove four lemmas: The first says that $u_l$, $u_r$, and $u_0$ are close to $u_L$ so that we actually localize our picture by looking at these points. The second says that $u_l$ and $u_r$ are close to $u_0$ in a stronger quantitative sense. The third says that $D_L(u_l)$ and $D_R(u_r)$ are negative of order $s_0^3$. The fourth says that the derivative of the entropy dissipations at $u_l$ and $u_r$ are small in a quantitative sense.
\begin{lemma}\label{lem5}
For all $C$ sufficiently large and $s_0$ sufficiently small,
\begin{equation}\label{eqn_lem5_desired}
|u_l - u_L| + |u_0 - u_L| + |u_r - u_L| \sim s_0.
\end{equation}
\end{lemma}

\begin{proof}
We first prove $|u_L - u_l| \lesssim s_0$.
Indeed, say $u(t)$ is the integral curve emanating from $u_L$ in the $r_1$ direction and $u_l = u(t_l)$. Then,
\begin{equation}\label{eqn_lem5_eqn1}
u_l - u_L = \int_0^{t_l} r_1(u(t)) \ dt = t_l r_1(u_L) + \int_0^{t_l} r_1(u(t)) - r_1(u_L) \ dt = t_l r_1(u_L) + \bigO(t_l^2),
\end{equation}
where we have used that $|u(t) - u_L| \le t$ by the fundamental theorem of calculus. Since the defining property of $u_l$ and $t_l$ is $\tilde\eta(u_l) = 0$, using Taylor's theorem, we obtain
\begin{equation}
0 = \tilde\eta(u_L) + \nabla\tilde\eta(u_L) \cdot (u_l - u_L) + \frac{1}{2}\nabla^2\tilde\eta(u)(u - u_L) \cdot (u-u_L)
\end{equation}
for some $u$ with $|u - u_L| < |u_l - u_L|$. However, because $\tilde\eta$ is strictly convex, we obtain the inequality
\begin{equation}\label{eqn_lem5_eqn2}
-\tilde\eta(u_L) \ge \nabla\tilde\eta(u_L) \cdot (u_l - u_L).
\end{equation}
Using Lemma \ref{l2_rel_entropy_lemma}, the relative entropy is quadratic and we compute
\begin{equation}\label{eqn_lem5_eqn3}
-\tilde\eta(u_L) = \eta(u_L | u_R) \lesssim |u_L - u_R|^2 \sim s_0^2.
\end{equation}
Moreover, expanding the gradient at $u_L$ we obtain
\begin{equation}\label{eqn_lem5_eqn4}
\nabla\tilde\eta(u_L) = \nabla\eta(u_R) - \nabla\eta(u_L) = s_0\nabla^2\eta(u_L)r_1(u_L) + \bigO(s_0^2).
\end{equation}
Therefore, combining the expansions (\ref{eqn_lem5_eqn1}), (\ref{eqn_lem5_eqn3}), and (\ref{eqn_lem5_eqn4}) with the inequality (\ref{eqn_lem5_eqn2}) yields
\begin{equation}
s_0t_l\nabla^2\eta(u_L)r_1(u_L) \cdot r_1(u_L) + \bigO(t_ls_0^2 + s_0t_l^2) \lesssim s_0^2.
\end{equation}
Dividing by $s_0$ and taking $C$ sufficiently large (to guarantee $t_l$ is sufficiently small) and $s_0$ sufficiently small, we obtain $ |u_l - u_L| \lesssim t_l \lesssim s_0$. Second, we note that a completely analogous argument proves that $|u_r - u_R| \lesssim s_0$. Since $|u_R - u_L| \lesssim s_0$, we also obtain $|u_r - u_L| \lesssim s_0$.
Finally, since $u_0$ is along the $1$-shock curve connecting $u_L$ and $u_R$, $u_0 = \S^1_{u_L}(t_0)$ for some $0 < t_0 < s_0$, where the shock curve $\S^1_{u_L}$ is parametrized by arc length. It immediately follows that
$|u_0 - u_L| \sim t_0 < s_0$.
It remains to prove the reverse inequalities.
However, all of the reverse inequalities follow from applying Lemma \ref{lem_dist} at $u_L$, $-\tilde\eta(u_L) \lesssim s_0 d(u_L,\bd\Pi)$,
where $d(u_L,\bd\Pi)$ denotes the distance from $u_L$ to $\bd\Pi$. Using (\ref{eqn_lem5_eqn3}), we obtain the lower bound, $s_0 \lesssim d(u_L,\bd\Pi)$.
Since each of $u_0$, $u_l$, and $u_r$ are contained in $\bd\Pi$, we have the desired lower bound and (\ref{eqn_lem5_desired}) follows.
\end{proof}

\begin{lemma}\label{lem6}
For all $C$ sufficiently large and $s_0$ sufficiently small, 
\begin{equation}\label{eqn_lem6_desired}
|u_r - u_l| + |u_0 - u_l| + |u_r - u_0| \lesssim Cs_0^2.
\end{equation}
\end{lemma}

\begin{proof}
First, we prove (\ref{eqn_lem6_desired}) for $|u_0 - u_l|$. As in Lemma \ref{lem5}, we expand $u_0 = \S^1_{u_L}(t_0)$ and $u_l = u(t_l)$, where $u(t)$ is the integral curve emanating from $u_L$ in the $r_1$ direction. By the asymptotic expansion of  $\S^1_{u_L}(t)$ in Lemma \ref{lem_hugoniot}, we have
\begin{equation}
u_0 = u_L + t_0 r_1(u_L) + \bigO(t_0^2) \qquad \text{and}\qquad u_l = u_L + t_l r_1(u_L) + \bigO(t_l^2),
\end{equation}
where $t_0 \sim |u_0 - u_L|$ and $t_l \sim |u_l - u_L|$. The condition $u_l,u_0 \in \bd \Pi$ yields $\tilde\eta(u_0) = 0 = \tilde\eta(u_l)$.
Therefore, expanding about $u_L$, by Taylor's theorem and the computations of derivatives in Lemma \ref{lem_eta}, we obtain for $u_0$,
\begin{equation}
\left|\tilde\eta(u_0) - \left[\tilde\eta(u_L) + |u_0 - u_L|\nabla\tilde\eta(u_L) \cdot r_1(u_L)\right]\right| \lesssim Cs_0|u_0 - u_L|^2.
\end{equation}
Now, since $|u_0 - u_L| = t_0 + \bigO(t_0^2)$, we obtain
\begin{equation}\label{eqn_lem6_eqn1}
\left|\left[\tilde\eta(u_L) + t_0\nabla\tilde\eta(u_L) \cdot r_1(u_L)\right]\right| \lesssim Cs_0t_0^2.
\end{equation}
Using the same computation for $u_l$ and subtracting from (\ref{eqn_lem6_eqn1}), we obtain
\begin{equation}
\left|(t_0 -t_l)\left[r_1(u_L) \cdot \nabla\tilde\eta(u_L)\right]\right| \lesssim Cs_0(t_0^2 + t_l^2)
\end{equation}
Using the computation of $\nabla\tilde\eta(u_L)$ in (\ref{eqn_lem5_eqn4}) and the localization result of Lemma \ref{lem5},
\begin{equation}
|t_0 - t_l| \lesssim C(t_0^2 + t_l^2) \lesssim Cs_0^2.
\end{equation}
Thus, we conclude that 
\begin{equation}
|u_0 - u_l| \lesssim |(t_0 - t_l)r_1(u_L)| + \bigO(t_0^2 + t_l^2) \lesssim Cs_0^2.
\end{equation}
Second, expanding $u_0$ and $u_r$ in terms of $u_R$ gives a completely analogous proof that $|u_r - u_0| \lesssim Cs_0^2$.

%As before, say
%\begin{equation}
%u_0 = u_L + t_0 r_1(u_L) + \bigO(t_0^2) \qquad \text{and}\qquad u_r = u_R - t_r r_1(u_R) + \bigO(t_r^2).
%\end{equation}
%However, since $u_R$ is also on the shock curve from $u_L$,
%\begin{equation}
%u_R = u_L + s_0r_1(u_L) + \bigO(s_0^2) \qquad \text{and} \qquad r_1(u_R) = r_1(u_L) + \bigO(s_0)
%\end{equation}
%and we combine the expansions to obtain
%\begin{equation}
%\begin{aligned}
%u_0 &= u_R + (u_L - u_R) + t_0 r_1(u_L) + \bigO(t_0^2)\\
%	&= u_R - (s_0 - t_0)r_1(u_L) + \bigO(t_0^2 + s_0^2)\\
%	&= u_R - (s_0 - t_0)r_1(u_R) + \bigO(s_0^2 + t_0^2).
%\end{aligned}
%\end{equation}
%The argument now proceeds as in the case of $u_l$ after expanding $\nabla\tilde\eta$ about $u_R$ and noting that\\ $s_0 \lesssim t_0 \le s_0$.
\end{proof}

\begin{lemma}\label{lem7}
For $C$ sufficiently large, and $s_0$ sufficiently small, we have estimates,
\begin{equation}
D_R(u_r) \lesssim -s_0^3 \qquad \text{and} \qquad D_L(u_l) \lesssim -s_0^3.
\end{equation}
\end{lemma}

\begin{proof}
Let $u(t)$ be the integral curve emanating from $u_L$ in the direction of $r_1$. Then, we note $\nabla \lambda^1(u) \cdot r_1(u) \lesssim -1$ by Assumption \ref{assum} (b) and the choice of eigenvectors in Lemma \ref{lem_normalization}. Therefore, we estimate $D_L$ along $u(t)$ using Lemma \ref{l2_rel_entropy_lemma},
\begin{equation}
\begin{aligned}
\frac{d}{dt} D_L(u(t)) &= -\left[\nabla\eta(u(t)) - \nabla\eta(u_L)\right]\left[f^\prime(u(t)) - \lambda^1(u(t))I\right] \cdot r_1(u(t)) + \eta(u(t) | u_L)\nabla\lambda^1(u(t)) \cdot r_1(u(t))\\
	&= \eta(u(t) | u_L)\nabla\lambda^1(u(t)) \cdot r_1(u(t)) \lesssim - |u(t) - u_L|^2 \lesssim -t^2.\\
\end{aligned}
\end{equation}
Now, by the fundamental theorem of calculus and $D_L(u_L) = 0$, we estimate $D_L$ at $u_l = u(t_l)$,
\begin{equation}
D_L(u_l) = \int_0^{t_l} \frac{d}{d\tau} D_L(u(\tau)) \ d\tau \lesssim -(t_l)^3 \lesssim -s_0^3,
\end{equation}
since $t_l$ is comparable to $|u_l - u_L|$, which is lower bounded by Lemma \ref{lem5}. The same argument applied to $D_R$ and $u_r$ completes the proof.
\end{proof}

\begin{lemma}\label{lem8}
For all $C$ sufficiently large and $s_0$ sufficiently small, we have 
\begin{equation}\label{eqn_lem8_desired1}
|\nabla D_R(u_r)| \lesssim s_0^2 \qquad \text{and} \qquad |\nabla D_L(u_l)| \lesssim s_0^2.
\end{equation}
Moreover,
\begin{equation}\label{eqn_lem8_desired2}
|\nabla^2 D_R(u)| \lesssim 1 \qquad \text{and} \qquad |\nabla^2 D_L(u)| \lesssim 1.
\end{equation}
\end{lemma}

\begin{proof}
We begin by computing $\nabla D_L$ as
\begin{equation}
\nabla D_L(u) = -\left(\nabla \eta(u) - \nabla \eta(u_L)\right)\left[f^\prime(u) - \lambda^1(u)I\right] + \nabla\lambda^1(u)\eta(u | u_L).
\end{equation}
Now, we note that the second term is sufficiently small as $\eta(u | u_L) \lesssim |u-u_L|^2$ by Lemma \ref{l2_rel_entropy_lemma}. However, for the first term $|\nabla\eta(u) - \nabla\eta(u_L)| \sim |u - u_L|$. Thus, we need to expand this term more carefully and use cancellation from $f^\prime(u) - \lambda^1(u)$ applied to the $r_1$ terms. In particular, using that $\nabla^2\eta(u_L)r_1(u_L) \parallel l^1(u_L)$,
\begin{equation}
\begin{aligned}
|\nabla D_L(u_0)| &\lesssim \left(|u_0 - u_L|\nabla^2\eta(u_L)r_1(u_L) + \bigO(|u_0 - u_L|^2)\right)\left[f^\prime(u_0) - \lambda^1(u_0)I\right] - \nabla\lambda^1(u_0)\eta(u_0 | u_L)\\
	&\lesssim |u_0 - u_L|^2 + |u_0 - u_L| \nabla^2\eta(u_L)r_1(u_L)\left[f^\prime(u_0) - \lambda^1(u_0)I\right]\\
	&\lesssim |u_0 - u_L|^2 \lesssim s_0^2.
\end{aligned}
\end{equation}
Once again, a completely analogous computation proves (\ref{eqn_lem8_desired1}) for $D_R$. It remains only to verify the second derivative estimates (\ref{eqn_lem8_desired2}) at an arbitrary state $u\in \Pi$. For $D_L$, we compute
\begin{equation}
\begin{aligned}
\nabla^2D_L(u) &= -\nabla^2\eta(u)\left[f^\prime(u) - \lambda^1(u)I\right] -\left(\nabla \eta(u) - \nabla \eta(u_L)\right)\left[f^{\prime\prime}(u) - I \tens \nabla\lambda^1(u)\right]\\
	&\qquad+\nabla^2\lambda^1(u)\eta(u | u_L) + \nabla\lambda^1(u) \tens \left[\nabla\eta(u) - \nabla\eta(u_L)\right],
\end{aligned}
\end{equation}
where we are vague about the meaning of some tensor multiplications, as we do not need precise structure for our estimate. Indeed, since $|u - u_L| \lesssim C^{-1}$ for $u\in \Pi$ and $\eta(u|u_L) \lesssim |u-u_L|$ by Lemma \ref{l2_rel_entropy_lemma}, each term is bounded by either $C^{-1}$ or $1$ up to a constant depending only on the system. So, for $C$ sufficiently large (\ref{eqn_lem8_desired2}) follows for $D_L$ and an analogous argument yields (\ref{eqn_lem8_desired2}) for $D_R$.
\end{proof}

\begin{flushleft}
\uline{\bf{Proof of Proposition \ref{prop_shock_explicit}}}
\end{flushleft}
%\begin{proof}[Proof of Proposition \ref{prop_shock_explicit}]
Pick $C$ sufficiently large and $s_0$ sufficiently small so that
$$D_R(u_r) \lesssim -s_0^3 \qquad\text{and}\qquad D_L(u_l) \lesssim -s_0^3$$
and moreover
$$|u_r - u_0| + |u_l - u_0| \lesssim Cs_0^2 \qquad \text{and} \qquad |\nabla D_R(u_r)|\lesssim s_0^2 \qquad\text{and}\qquad |\nabla D_L(u_l)|\lesssim s_0^2.$$
Thus, at $u_0$, we have
\begin{equation}
\begin{aligned}
D_L(u_0) &= D_L(u_l) + \nabla D_L(u_l)(u_l - u_0) + \bigO(\|\nabla^2D_L\|_{L^\infty(\Pi)}|u_l - u_0|^2)\\
	&\lesssim -s_0^3 + Cs_0^4 + C^2s_0^4 \lesssim -s_0^3,
\end{aligned}
\end{equation}
for $C$ sufficiently large and $s_0$ sufficiently small and similarly for $D_R(u_0)$. We conclude, using $D_{cont} = D_L + (1+Cs_0)D_R$,
\begin{equation}
D_{cont}(u_0) = D_R(u_0) + (1+Cs_0)D_L(u_0) \lesssim -s_0^3 - Cs_0^4 \le -s_0^3.
\end{equation}
%\end{proof}

%%%%%%%%%%%%%%%%%%%%%%%%%%%%%%%%%%%%%%%%%%%%%%%%%%%%%%%%%%%%%%%%%%%%%%%%%%%%%%%%%%%%%%%%%%%%%%%%%%%%%%%%

\subsection{Step 3: Comparison of Maxima to Shock Curve}

\begin{proposition}\label{prop_comparison}
Let $C$ sufficiently big that $D_{cont}(u)$ has a unique maximum, $u^*$, in $\Pi_{C,s_0}$ and $u_0$ as in Proposition \ref{prop_shock_explicit}. Then, for all $C$ sufficiently large and $s_0$ sufficiently small (depending on $C$), $|u_0 - u^*| \lesssim \frac{s_0}{C}$ and for any $v\in \R^n$ with $|v| = 1$,
\begin{equation}\label{eqn_propcomparison_desired}
|\nabla^2D_{cont}(u^*)v \cdot v| \lesssim Cs_0.
\end{equation}
\end{proposition}

\begin{proof}
We have shown a lower bound on $|\nabla^2 D_{cont}|$ in Lemma \ref{lem3}. Here we prove the corresponding upper bound (\ref{eqn_propcomparison_desired}). Indeed, recalling the formula for $\nabla^2 D_{cont}$ in (\ref{eqn_lem3_eqn4}) and evaluating at $u^*$,
\begin{equation}
\begin{aligned}
\nabla^2 D_{cont}(u^*) &= -\nabla^2\tilde\eta(u^*)\left[f^\prime(u^*) - \lambda^1(u^*)I\right] - \nabla\tilde\eta(u^*)\left[f^{\prime\prime}(u^*) - I \tens \nabla\lambda^1(u^*)\right] + \nabla\lambda^1(u^*) \tens \nabla\tilde\eta(u^*).
\end{aligned}
\end{equation}
Therefore, taking norms and using $|\nabla\tilde\eta(u)| \lesssim s_0$ and $|\nabla^2\tilde\eta(u)| \lesssim Cs_0$ by Lemma \ref{lem_eta} gives (\ref{eqn_propcomparison_desired}) as desired.

It remains to prove $|u_0 - u^*| \lesssim C^{-1}s_0$. We recall that $\nu(u)$ denotes the outward unit normal on $\bd\Pi$. Since we a priori only know that $u^*\in \bd\Pi$ with $\nu(u^*) = l^1(u^*)$, we use $\nu$ at $u_0$ and $u^*$ as a proxy to compare $u_0$ and $u^*$, which is justified by Lemma \ref{lem_pi}.
In particular, we first compute $\nu = \frac{\nabla\tilde\eta}{|\nabla\tilde\eta|}$ at $u_0$.
As computed in Lemma \ref{lem_eta},
\begin{equation}
\nabla\tilde\eta(u_0) = Cs_0\left(\nabla\eta(u_0) - \nabla\eta(u_L)\right) + \left(\nabla\eta(u_R) - \nabla\eta(u_L)\right).
\end{equation}
We expand carefully about $u_L$ using $u_0 = u_L + |u_0 - u_L| r_1(u_L) + \bigO(s_0^2)$ and $\nabla^2\eta(u_L)r_1(u_L) = Kl^1(u_L)$ for some $K\sim 1$ to obtain
\begin{equation}
\begin{aligned}
\nabla\tilde\eta(u_0) &= Cs_0\left(c^*|u_0 - u_L| l^1(u_L) + \bigO(s_0^2)\right) + c^*s_0 l^1(u_L) + \bigO(s_0^2)\\
	&= K(s_0 + Cs_0|u_0 - u_L|)l^1(u_L) + \bigO(s_0^2 + Cs_0^3).
\end{aligned}
\end{equation}
Thus, for $b_1 := K(s_0 + Cs_0|u - u_L|)$, we have $b_1 \sim s_0$ and $\nabla\tilde\eta(u_0) = b_1l^1(u_L) + \bigO(s_0^2)$.
Therefore, we bound $\nu(u_0) - l^1(u_0)$ using the preceding computations, $|u_0 - u_L|\lesssim s_0$ by Lemma \ref{lem5}, and $|\nabla\tilde\eta(u_0)|\gtrsim s_0$ by Lemma \ref{lem_nu}.
\begin{equation}\label{eqn_prop_eqn1}
\begin{aligned}
\left|\nu(u_0) - l^1(u_0)\right| &= \left|\frac{\nabla\tilde\eta(u_0) - \left|\nabla\tilde\eta(u_0)\right| l^1(u_0)}{\left|\nabla\tilde\eta(u_0)\right|}\right|\\
	&\le \left|\frac{\left[\nabla\tilde\eta(u_0) - b_1l^1(u_L)\right] + b_1\left[l^1(u_L) - l^1(u_0)\right] +\left[b_1 - \left|\nabla\tilde\eta(u_0)\right|\right]l^1(u_0)}{\left|\nabla\tilde\eta(u_0)\right|}\right|\\
	&\lesssim\frac{s_0^2}{\left|\nabla\tilde\eta(u_0)\right|} \lesssim s_0.
\end{aligned}
\end{equation}
Next, we note that since $u^*$ is a maximum, by Proposition \ref{prop_max}, $\nu(u^*) \parallel l^1(u^*)$ and $r_1(u^*)$ is outward pointing. Since $r_1(u^*) \cdot \nu(u^*) > 0$, $\nu(u^*) \parallel l^1(u^*)$, and $l^1(u^*) \cdot r_1(u^*) > 0$, it follows $l^1(u^*) = \nu(u^*)$. Also, by Lemma \ref{lem_pi},
\begin{equation}\label{eqn_prop_eqn2}
|\nu(u^*) - \nu(u_0)| \sim C|u^* - u_0|,
\end{equation}
for $C$ sufficiently large.
Finally, we conclude by combining (\ref{eqn_prop_eqn1}) and (\ref{eqn_prop_eqn2}):
\begin{equation}
\begin{aligned}
|u^* - u_0| &\lesssim C^{-1}|\nu(u^*) - \nu(u_0)|\\
	&\lesssim C^{-1}\left(|\nu(u^*) - l^1(u^*)| + | l^1(u^*) - l^1(u_0)| + | l^1(u_0) - \nu(u_0)|\right)\\
	&\lesssim C^{-1}s_0,
\end{aligned}
\end{equation}
for sufficiently large $C$ and sufficiently small $s_0$.
\end{proof}

%%%%%%%%%%%%%%%%%%%%%%%%%%%%%%%%%%%%%%%%%%%%%%%%%%%%%%%%%%%%%%%%%%%%%%%%%%%%%%%%%%%%%%%%%%%%%%%%%%%%%%%%%

\subsection{Proof of Proposition \ref{prop_cont}}

\begin{lemma}\label{lem9}
For any $u \in \Pi$, $C$ sufficiently large, and $s_0$ sufficiently small,
\begin{equation}\label{eqn_lem9_desired}
|\nabla^3 D_{cont}(u)| \lesssim Cs_0
\end{equation}
\end{lemma}

\begin{proof}
Using $\nabla \tilde q = \nabla \tilde \eta \nabla f$ and the product rule, we differentiate $D_{cont}$ repeatedly to obtain
\begin{equation}
\begin{aligned}
\nabla^3 D_{cont} &= -Cs_0 \nabla \left[\nabla^2\eta\left(f^\prime - \lambda^1I\right)\right] -\nabla\tilde\eta\nabla\left[f^{\prime\prime} + \nabla\lambda^1 \tens I\right] -\nabla^2\tilde\eta\left[f^{\prime\prime} + \nabla\lambda^1 \tens I\right] +\nabla\lambda^1 \tens \nabla^2\tilde\eta\\
	&\quad+2\nabla^2\lambda^1 \tens \nabla \tilde\eta +\tilde\eta\nabla^3\lambda^1,
\end{aligned}
\end{equation}
where because we will not need to exploit any precise tensor structure, we have been rather vague on the meaning of the various products. Since $|\nabla^3\tilde\eta(u)| + |\nabla^2\tilde\eta(u)| \lesssim Cs_0$ and $|\nabla\tilde\eta(u)|\lesssim s_0$ by Lemma \ref{lem_eta}, each term containing a derivative of $\tilde\eta$ satisfies (\ref{eqn_lem9_desired}). Finally, the only term not containing a derivative of $\tilde\eta$ is bounded via Lemma \ref{lem_dist} and Lemma \ref{lem_eta} as
\begin{equation}
\left|\tilde\eta(u)\nabla^3\lambda^1(u)\right| \lesssim s_0 d(u,\bd \Pi) \lesssim C^{-1}s_0.
\end{equation}
\end{proof}

\begin{proof}[Proof of Proposition \ref{prop_cont}]
First, pick $C$ sufficiently large and $s_0$ sufficiently so that $D_{cont}(u)$ has a unique maximum $u^*$ in $\Pi_{C,s_0}$ and Proposition \ref{prop_max}, Proposition \ref{prop_shock_explicit}, and Proposition \ref{prop_comparison} hold. Then, we note that for $u\in \Pi$,
\begin{equation}\label{eqn_propcont_eqn1}
D_{cont}(u_0) \lesssim -s_0^3 \quad \text{and}\quad |u_0 - u^*| \lesssim \frac{s_0}{C} \quad \text{and} \quad |\nabla^2D_{cont}(u^*)|\lesssim Cs_0 \quad \text{and} \quad  |\nabla^3D_{cont}(u)| \lesssim Cs_0,
\end{equation}
where all implicit constants depend only on the system.
Therefore, since $u^*$ is a local maximum of $D_{cont}$ for $u \in \Pi_{C,s_0}$, the first order term term in the expansion about $u^*$ disappears and by Taylor's theorem and (\ref{eqn_propcont_eqn1}),
\begin{equation}\label{eqn_prop_1}
\begin{aligned}
D_{cont}(u^*) &= D_{cont}(u_0) - \frac{1}{2}\nabla^2 D_{cont}(u^*)(u^* - u_0) \cdot (u^* - u_0) + \bigO(|\nabla^3 D_{cont}||u_0 - u^*|^3)\\
	&\lesssim -s_0^3 + \frac{s_0^3}{C} + \frac{s_0^4}{C^2},
\end{aligned}
\end{equation}
where all of the implicit constants depend only on the system. Taking $C$ sufficiently large to dominate the implicit constants in (\ref{eqn_prop_1}), we obtain the existence of $K$, depending only on the system, such that for $C$ sufficiently large and $s_0$ sufficiently small, for any $u \in \Pi$,
\begin{equation}
D_{cont}(u) \le D_{cont}(u^*) \le -Ks_0^3.
\end{equation}
\end{proof}

%%%%%%%%%%%%%%%%%%%%%%%%%%%%%%%%%%%%%%%%%%%%%%%%%%%%%%%%%%%%%%%%%%%%%%%%%%%%%%%%%%%%%%%%%%%%%%%%%%
%%%%%%%%%%%%%%%%%%%%%%%%%%%%%%%%%%%%%%%%%%%%%%%%%%%%%%%%%%%%%%%%%%%%%%%%%%%%%%%%%%%%%%%%%%%%%%%%%%
%%%%%%%%%%%%%%%%%%%%%%%%%%%%%%%%%%%%%%%%%%%%%%%%%%%%%%%%%%%%%%%%%%%%%%%%%%%%%%%%%%%%%%%%%%%%%%%%%%
%%%%%%%%%%%%%%%%%%%%%%%%%%%%%%%%%%%%%%%%%%%%%%%%%%%%%%%%%%%%%%%%%%%%%%%%%%%%%%%%%%%%%%%%%%%%%%%%%%
%%%%%%%%%%%%%%%%%%%%%%%%%%%%%%%%%%%%%%%%%%%%%%%%%%%%%%%%%%%%%%%%%%%%%%%%%%%%%%%%%%%%%%%%%%%%%%%%%%

\section{Proof of Proposition \ref{prop_shock}}\label{secWill2}

\subsection{Proof Sketch}
%We take the following conventions for eigenvalues of $f^\prime$. Namely, since our system is strictly hyperbolic, we obtain $f^\prime(u)$ has right eigenvectors $r_i(u)$ and left eigenvectors $l_i(u)$ corresponding to the eigenvalues $\lambda^i(u)$. Moreover, $r_i$ and $l_i$ can be chosen to vary smoothly in $u$ satisfying $|r_i(u)| = |l_j| = 1$, and  $r_i(u) \cdot l_j(u) = \delta_{ij}$, where $\delta_{ij}$ is the Kroenecker delta. Finally, since we imagine that $r_1(u_L)$ points towards $u_R$, i.e. $r_1$ points in the direction of the shock curve, we have $\lambda^1(u)$ is decreasing along $r_1$ and
%$$r_1 \cdot \nabla\lambda^1 < 0.$$
In this section, we prove Proposition \ref{prop_shock}. We find that the full entropy dissipation function, $D_{RH}$, defined in (\ref{defn_DRH}), is unwieldy and hides important structural information by comparing the fixed shock $(u_L,u_R,\sigma_{LR})$ to an unbounded family of entropic shocks $(u,\S^1_u(s),\sigma_u(s))$. We instead use Assumption \ref{assum} (j) to introduce a single shock, $(u, u^+, \sigma_\pm)$, referred to below as the maximal shock, which we will use to control the others. More precisely, for each state $u\in \Pi_{C,s_0}$, we define $u^+$ as the unique state such that $(u,u^+,\sigma_\pm)$ is an entropic $1$-shock and 
\begin{equation}\label{eqn_Dmax_defn}
D_{max}(u) := \max_{s \ge 0} D_{RH}(u, \S^1_u(s), \sigma^1_u(s)) = D_{RH}(u, u^+, \sigma_\pm). 
\end{equation}
This definition should look suspicious for a number of reasons, not the least of which, we are taking a maximum over an unbounded set. Moreover, we are attempting to compare $(u_L,u_R,\sigma_\pm)$ to a global family of shocks which need not stay near $\Pi$ nor even in $\Nu$. Nevertheless, this will be justified below (see Lemma \ref{lem15}). Proposition \ref{prop_shock} essentially follows by decomposing $\Pi$ into several sets and showing that $D_{max}$ is negative on each set. To better structure the proof, we divide the proof into two parts based on the characteristic length scale of sets in each part:
\begin{itemize}
\item[\uline{Part 1:}] Localization from the global length scale, $C^{-1}$, on $\Pi$ to a set $R_{C,s_0}$ containing $u_L$ of length scale $s_0$.
\item[\uline{Part 2:}] Analysis (inside $R_{C,s_0}$) at the critical length scale $s_0$.
\end{itemize}
Within Part 1, we further decompose our proof into two steps. For approximate or pictorial definitions of $Q_{C,s_0}$ and $R_{C,s_0}$, see Figure \ref{fig:large_scale}. For rigorous definitions, see variously (\ref{eqn_Qdefn}) and (\ref{eqn_defn_R}).
\begin{itemize}
\item \uline{Step 1.1:} First, we show that $D_{max}(u)\le 0$ for states $u$ outside of the set $Q_{C,s_0}$ with maximal shock $(u, u^+,\sigma_\pm)$ and $u^+$ outside of $\Pi$. We use that in this regime $u$ is sufficiently far from $u^*$, the unique maximum of the continuous entropy dissipation, $D_{cont}$, to gain an improved estimate of $D_{cont}(u)$. The assumption that $u^+\notin \Pi$ should be interpreted as $u$ is close (to order $s_0$) to the boundary of $\Pi_{C,s_0}$ (see Lemma \ref{lem21} below for a precise statement).
\item \uline{Step 1.2:} Second, we show that $D_{max}(u) \le 0$ for states $u$ inside $Q_{C,s_0}$ but outside $R_{C,s_0}$ and remove the assumption that $u^+$ must lie outside of $\Pi$. This step should be interpreted as showing $D_{max}(u) \le 0$ for $u$ farther than order $s_0$ from the boundary of $\Pi_{C,s_0}$.
\end{itemize}
\begin{figure}[t]
    \centering
    \includegraphics[scale = 1.2,]{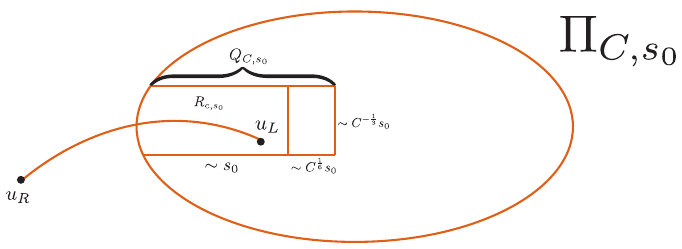}
    \caption{This is an idealized portrait of the two dimensional situation, where we imagine $r_1$ is the horizontal direction, $r_2$ is the vertical direction, and $r_1$ and $r_2$ are almost orthogonal. For Part 1, we divide $\Pi_{C,s_0}$ into three regimes: states inside $R_{C,s_0}$, which are very close to $u_L$ and are handled in Part 2; states inside $Q_{C,s_0}$ but outside $R_{C,s_0}$, which are close to $u_L$, but sufficiently far in the interior of $\Pi$; and states in the remainder of $\Pi$.}
    \label{fig:large_scale}
\end{figure}
Similarly, within Part 2, we decompose our proof into four steps. For approximate or pictorial definitions of the referenced sets, see Figure \ref{fig:small_scale}.
For rigorous definitions, see variously (\ref{defn_R0}), (\ref{defn_Rbd}), and (\ref{defn_R+R-}).
\begin{itemize}
\item \uline{Step 2.1:} First, we show that $D_{max}(u) \le 0$ for states $u$ in the set $R_{C,s_0}^{bd}$. In this regime, the proof follows from a comparison to $u^*$, the global maximum of $D_{cont}$ on $\Pi$.
\item \uline{Step 2.2:} Second, we show $D_{max}(u) \le 0$ for states $u$ in the set $R_{C,s_0}^0$. We prove that in this regime, $u_L$ is a local maximum for $D_{max}$ with $D_{max}(u_L) = 0$.
\item \uline{Step 2.3:} Third, we show that $D_{max}(u) \le 0$ for states $u$ in the sets $R_{C,s_0}^+$ and $R_{C,s_0}^-$, by ruling out critical points of $D_{max}$. In essence, we take $C$ sufficiently large to make the set $R_{C,s_0}$ narrow enough the problem becomes sufficiently one dimensional and analogous to the scalar case. More precisely, we show that any critical point of $D_{max}$ in the interior of $R_{C,s_0}$ must satisfy $u_R - u^+ \approx u - u_L$. However, we show, via an ODE argument, that for states $u\in R_{C,s_0}^+$, $u_R - u^+$ and $u - u_L$ point in essentially opposite directions. The proof in $R_{C,s_0}^-$ replaces the ordering argument in $R_{C,s_0}^+$ with a magnitude based one. In particular, using an ODE argument, we show $|u-u_L| \gg |u_R - u^+|$ for $u\in R_{C,s_0}^-$, which again precludes the existence of critical points.
%\item \uline{Step 2.4:} Fourth, we show that $D_{max}(u) \le 0$ for states $u$ in the set $R_{C,s_0}^-$, by ruling out critical points of $D_{max}$. Again, this is proved by ruling out the possibility that $u_R - u^+ \approx u - u_L$ for $u\in R_{C,s_0}^-$. However, we replace the sign argument in Step 3, with a magnitude argument: We use an ODE argument that shows as the state $u$ moves right from $u_L$, $|u - u_L|$ grows much faster than $|u_R - u^+|$.
\end{itemize}
\begin{figure}[b]
    \centering
    \includegraphics[scale= .5]{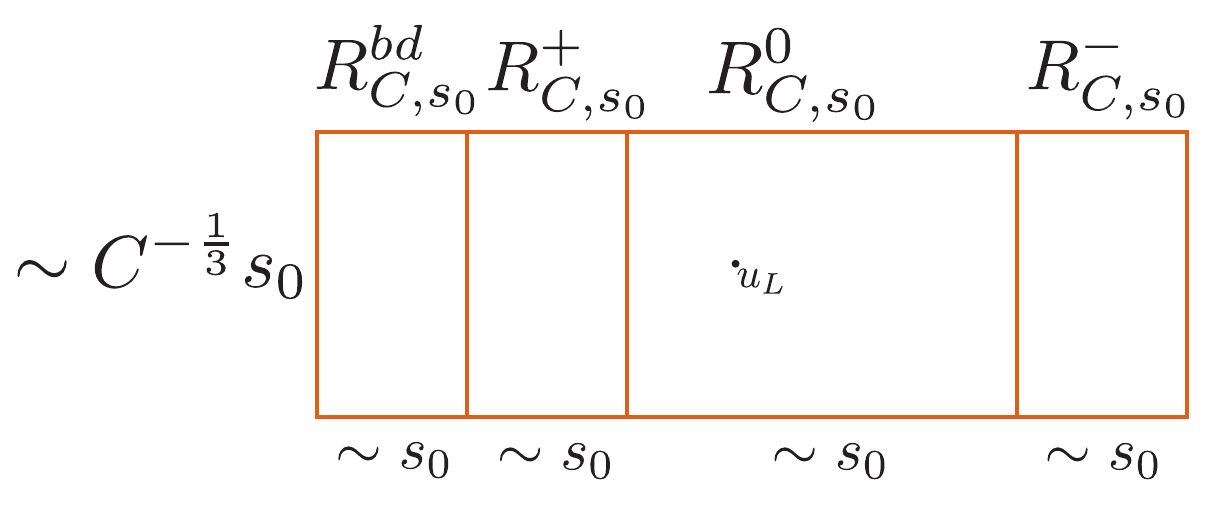}\hspace{1cm}
    \caption{Again, this is an idealization of the two dimensional picture, in which the horizontal direction is $r_1$ and the vertical direction is $r_2$. In Part 2, we subdivide the region $R_{C,s_0}$ into four regions, $R_{C,s_0}^{bd}$, $R_{C,s_0}^0$, $R_{C,s_0}^+$, and $R_{C,s_0}^-$, each of width $\sim s_0$. Each region is then treated individually.}
    \label{fig:small_scale}
\end{figure}

%%%%%%%%%%%%%%%%%%%%%%%%%%%%%%%%%%%%%%%%%%%%%%%%%%%%%%%%%%%%%%%%%%%%%%%%%%%%%%%%%%%%%%%%%%%%

%\red{subsubsection{Maximal Entropy Dissipation}}

Before beginning Part 1, we justify the definitions of $D_{max}$ and the maximal shock $(u,u^+(u),\sigma_\pm(u))$ stated above in (\ref{eqn_Dmax_defn}) and also that for appropriate $C$ and $s_0$, the set of possible maximal shocks, $\mathcal{M} = \set{u^+ \ | \ u\in \Pi_{C,s_0}}$, is compactly contained in $\Nu$.  To this end, in the following lemma, we show that the function
\begin{equation}
s \mapsto \F(s,u) := D_{RH}(u, \S^1_u(s),\sigma_u^1(s))
\end{equation}
has a unique maximum $s^*(u)$ with $s^*(u) \lesssim C^{-1/2}s_0^{1/2}$. The existence justifies the maximum over the unbounded set in (\ref{eqn_Dmax_defn}) while the uniqueness justifies the definition of $u^+$. The estimate justifies that $\mathcal{M}$ is compactly contained in $\Nu$ where $f,\ \eta,\ q,\ \lambda^i,\ r_i,\ \text{and }l^i$ are more regular. The additional regularity allows us to apply the classical techniques of calculus as in the proof of Proposition \ref{prop_cont}.

\begin{lemma}\label{lem15}
The function $s\mapsto \F(s,u)$ has a unique maximum, say $s^*$, for $s \in [0,\infty)$. Furthermore, $s^*$ is uniquely defined by the equation
\begin{equation}
\eta(u | \S^1_u(s^*)) = - \tilde\eta(u).
\end{equation}
Moreover, for $u^+ = \S^1_u(s^*)$, $|u - u^+| \lesssim C^{-1/2}s_0^{1/2}$ so that for $C$ sufficiently large and $s_0$ sufficiently small,
$\set{u^+ \ | \ u \in \Pi_{C,s_0}}$ is compactly contained in $\Nu$.
\end{lemma}

\begin{proof}
We expand $\F(s,u) = D_{RH}(u, \S^1_u(s), \sigma^1_u(s))$ as
\begin{equation}
\begin{aligned}
\F(s,u) = \left[q(u ; u_R) - \sigma^1_u(s) \eta(u | u_R) + \int_0^s \dot{\sigma}^1_u(t)\eta(u | \S^1_u(t)) \ dt \right]
    -(1+Cs_0)\left[q(u ; u_L) - \sigma^1_u(s) \eta(u | u_L)\right],
\end{aligned}
\end{equation}
where we have used Lemma \ref{lem_lax} with $v = u_R = \S^1_{u_L}(s_0)$.
Taking a derivative in $s$, we obtain
\begin{equation}
\begin{aligned}
\frac{d\F}{ds} &= -\dot{\sigma}^1_u(s)\eta(u | u_R) + \dot{\sigma}^1_u(s) \eta(u | \S^1_u(s)) + (1+Cs_0)\dot{\sigma}^1_u(s)\eta(u | u_L).
\end{aligned}
\end{equation}
Since $\dot{\sigma}^1(s) < 0$ for all $s > 0$ by Assumption \ref{assum} (j), we obtain $\F(s)$ is decreasing in $s$ if and only if
\begin{equation}
(1+Cs_0)\eta(u | u_L) - \eta(u | u_R) + \eta(u | \S^1_u(s)) = \tilde\eta(u) + \eta(u | \S^1_u(s)) \ge 0. 
\end{equation}
Define $g(s)$ as
\begin{equation}
g(s) = \tilde\eta(u) + \eta(u | \S^1_u(s)).
\end{equation}
Certainly, $g(0)  = \tilde\eta(u) \le 0$. Moreover, using Assumption \ref{assum} (j) again, for any $s > 0$,
\begin{equation}
g^\prime(s) = \frac{d}{ds}\eta(u | \S^1_u(s)) > 0.
\end{equation}
Therefore, we conclude there is a unique point $s^*$ such that $g(s^*) = 0$. By the above discussion, we note that $s^*$ is the unique global maximum for $s\mapsto \F(s,u)$ and is completely characterized by $g(s^*) = 0$. Finally, we let $u^+ = \S^1_u(s^*)$ and use $g(s^*) = 0$, Lemma \ref{l2_rel_entropy_lemma}, Lemma \ref{lem_dist}, and Lemma \ref{lem_eta} to estimate $|u-u^+|$ via
\begin{equation}
|u - u^+|^2 \lesssim \eta(u|u^+) =  -\tilde\eta(u) \lesssim s_0 d(u,\bd\Pi) \lesssim C^{-1}s_0
\end{equation}
Since $\Pi$ is compactly contained in $\Nu$ for appropriate $C$ and $s_0$, so too is the set of possible states $u^+$.
\end{proof}

%%%%%%%%%%%%%%%%%%%%%%%%%%%%%%%%%%%%%%%%%%%%%%%%%%%%%%%%%%%%%%%%%%%%%%%%%%%%%%%%%%%%%%%%%%%%%%%%%%%%%%%%

\subsection{Step 1.1: Shocks outside \texorpdfstring{$Q_{C,s_0}$}{QCs0}}
In this step, we show the existence of a set $Q = Q_{C,s_0}$ of ``width'' in the $r_i$ for $i\neq 1$ directions of order $C^{-1/3}s_0$ and ``height'' in the $r_1$ direction of order $C^{1/6}s_0$ such that outside of $Q$ but inside of $\Pi$, the continuous entropy dissipation $D_{max}(u)$ satisfies an improved uniform estimate. 
In particular, this allows us to handle the maximal entropy dissipation, $D_{max}(u)$ outside $Q_{C,s_0}$.

More precisely, take $C$ sufficiently large and $s_0$ sufficiently small such that $D_{cont}$ has a unique maximum $u^*$ on $\Pi_{C,s_0}$. Define the family of closed cylinders based at $u^*$, $Q_{C,s_0}$, via
\begin{equation}\label{eqn_Qdefn}
Q_{C,s_0} := \set{p = u^* - \sum_{i = 1}^n b_i r_i(u^*) \ \biggr| \ p\in \Pi_{C,s_0},\ |b_1| \le C^{1/6}s_0, \text{ and }\left(\sum_{i = 2}^n |b_i|^2\right)^{1/2} \le C^{-1/3}s_0}. 
\end{equation}
The main proposition of this step is then:
\begin{proposition}\label{prop_Qs0}
For any $C$ sufficiently large, $s_0$ sufficiently small (depending on $C$), and for any $u$ such that
$$u \in \Pi_{C,s_0} \backslash Q_{C,s_0},$$
also satisfies $u^+ = \S^1_u(s^*(u)) \notin \Pi_{C,s_0}$, the maximal entropy dissipation satisfies,
\begin{equation}
D_{max}(u) \lesssim -s_0^3.
\end{equation}
\end{proposition}
The main ingredient in the proof of Proposition \ref{prop_Qs0} is the following improvement of Proposition \ref{prop_cont} outside of $Q_{C,s_0}$:
\begin{proposition}\label{prop_improved}
For any $C$ sufficiently large, $s_0$ sufficiently small (depending on $C$), and for any $u$ such that
$$u \in \Pi_{C,s_0} \backslash Q_{C,s_0},$$
the continuous entropy dissipation satisfies,
\begin{equation}
D_{cont}(u) \lesssim -C^{1/3}s_0^3.
\end{equation}
\end{proposition}

Before beginning the proof of Proposition \ref{prop_improved}, we introduce two auxiliary sets related to $Q$. We say the top of $Q$ is the set
\begin{equation}
T_{C,s_0} := \set{p = u^* - \sum_{i = 1}^n b_i r_i(u^*) \ \biggr| \ p \in Q_{C,s_0},\ |b_1| = C^{1/6}s_0}.
\end{equation}
Similarly, the bottom of $Q$ is the set
\begin{equation}
B_{C,s_0} := \bd \Pi_{C,s_0} \cap Q_{C,s_0}.
\end{equation}

\begin{flushleft}
\uline{\bf{Proof of Proposition \ref{prop_improved}}}
\end{flushleft}
We note that by Lemma \ref{lem1}, $D_{cont}$ does not have any critical points in the interior of $\Pi$. Therefore, it suffices to prove the improved estimate on the boundary of the region $\Pi\backslash Q$, that is, on
\begin{equation}
\bd \Pi \backslash B_{C,s_0} \cup \bd Q_{C,s_0} \backslash B_{C,s_0}.
\end{equation}
We first estimate $D_{cont}$ on $\bd Q_{C,s_0} \backslash B_{C,s_0}$ using the precise estimates on $D_{cont}$ near $u^*$ proved in Proposition \ref{prop_cont}:

\begin{lemma}\label{lem16}
For $Q = Q_{C,s_0}$, any $C$ sufficiently large, and any $s_0$ sufficiently small, the continuous entropy dissipation, $D_{cont}$, satisfies the improved estimate
\begin{equation}
D_{cont}(u) \lesssim -C^{1/3}s_0^3,
\end{equation}
for any state $u$ in the set $\bd Q \backslash B$
\end{lemma}

\begin{proof}
By Proposition \ref{prop_cont}, Lemma \ref{lem3}, and Lemma \ref{lem9} we have for sufficiently large $C$ and sufficiently small $s_0$, and for any $v = \sum v_i r_i(u^*)$,
\begin{equation}\label{eqn_lem16_eqn1}
\begin{aligned}
D_{cont}&(u^*) \lesssim -s_0^3 \quad\text{and}\quad \nabla D_{cont}(u^*) = 0 \quad\text{and}\quad \|\nabla^3 D_{cont}\|_{L^\infty(\Pi_{C,s_0})}\lesssim Cs_0\\
&\quad\text{and}\quad \nabla^2 D_{cont}(u^*)v \cdot v \lesssim -s_0|v_1|^2 -Cs_0\sum_{i=2}^n |v_i|^2
\end{aligned}
\end{equation}
Then, using the definition of the cylinder $Q$, we split $\bd Q\backslash B$ into the top, $T$, and the sides, $\bd Q\backslash (T \cup B)$. First, for $\bd Q \backslash (T\cup B)$, we have an expansion of $u$ as
\begin{equation}\label{eqn_lem16_eqn2}
u = u^* + b_1r_1(u^*) + \sum_{i = 2}^n b_i r_1(u^*) \quad \text{where}\ \quad \left(\sum_{i=2}^n |b_i|^2\right)^{1/2} = C^{-1/3}s_0
\end{equation}
Next, we expand $D_{cont}(u)$ about $u^*$ and use Taylor's theorem, (\ref{eqn_lem16_eqn1}), and (\ref{eqn_lem16_eqn2}) to conclude for any $u\in \bd Q\backslash(T \cup B)$,
\begin{equation}
D_{cont}(u) \lesssim -s_0^3 - (Cs_0)(C^{-1/3}s_0)^2 + (Cs_0)(C^{1/6}s_0)^3\lesssim -C^{1/3}s_0^3.
\end{equation}
Now, for $u\in T$, we have an expansion of $u$ as
\begin{equation}\label{eqn_lem16_eqn3}
u = u^* + b_1r_1(u^*) + \sum_{i = 2}^n b_i r_1(u^*) \quad \text{where}\ \quad |b_1| = C^{1/6}s_0.
\end{equation}
Again, we expand $D_{cont}(u)$ about $u^*$ and use Taylor's theorem, (\ref{eqn_lem16_eqn1}), and (\ref{eqn_lem16_eqn3}) to conclude for any $u\in T$,
\begin{equation}
D_{cont}(u) \lesssim -s_0^3 - (s_0)(s_0C^{1/6})^2 + (Cs_0)(C^{1/6}s_0)^3 \lesssim -C^{1/3}s_0^3.
\end{equation}
\end{proof}

We now begin to prove the improved estimate on $\bd \Pi \backslash B_{C,s_0}$ by showing that there is a unique local maximum $\tilde u$ for $D_{cont}$ on $\bd \Pi \backslash B$. We will need an auxiliary lemma from Riemannian geometry on the surjectivity of the Gauss map. In our case, the following lemma has an elementary proof:
\begin{lemma}\label{lem17}
Let $M$ be a bounded, open, connected subset of $\R^n$ with $C^1$ boundary $\bd M$.
Then, for any fixed unit vector $v\in \R^n$, there is a point $p\in \bd M$ such that
$\nu(p) = v$
for $\nu$ the outward unit normal.
\end{lemma}

\begin{proof}
Fix $v\in \R^n$ a unit vector and define the map $f: M \rightarrow \R$ via $f(u) = u \cdot v$. Then, certainly $f$ is continuous and obtains a maximum on $M$, say $p^*$. However, since $\nabla f(p^*) = v \neq 0$, it follows that $p^*\in \bd M$. Now, by the Lagrange multiplier method, it follows that $\nabla f(p^*) \parallel \nu(p^*)$. Hence, by normalization considerations, $v = \pm \nu(p^*)$. Since $p^*$ is a maximum, differentiating $f$ along a curve $\gamma$ approaching $v$ from the interior of $M$, we see that $v$ must be outward pointing at $p^*$. Therefore, $v = \nu(p^*)$.

%Second, suppose the curvature condition holds and $\nu(p_1) = \nu(p_2) = v$ for $v\in \R^n$ a unit vector and $p_1,p_2 \in \bd M$. Thus, the intermediate value theorem (say along a geodesic) guarantees the existence of a point $p_3 \in \bd M$ such that $(D_{p_3}\nu)u = 0$ for some $u\in T_{p_3}\bd\Pi$, which contradicts non-zero principal curvatures.
\end{proof}

\begin{lemma}\label{lem18}
For any $C$ sufficiently large, and any $s_0$ sufficiently small, there is a unique point $\tilde u$ on $\bd \Pi$ such that
\begin{equation}
l^1(\tilde u) \parallel \nabla\tilde\eta(\tilde u) \qquad \text{and} \qquad r_1(\tilde u) \text{ is inward pointing.}
\end{equation}
Moreover, $\tilde u$ is the unique local maximum for $D_{cont}$ on $\bd \Pi \backslash B_{C,s_0}$.
\end{lemma}

\begin{proof}
We note that these conditions are equivalent to saying that there is a unique point $\tilde u \in \bd \Pi$ satisfying the nonlinear system of equations
\begin{equation}
\nu(\tilde u) = -l^1(\tilde u).
\end{equation}
Applying Lemma \ref{lem17} to $\Pi$ and using that Lemma \ref{lem_pi} implies $p \mapsto \nu(p)$ is injective for sufficiently large $C$, for any $p\in \bd\Pi$, there is a unique point $f(p)\in \bd\Pi$ such that
\begin{equation}
\nu(f(p)) = -l^1(p).
\end{equation}
Thus, we seek a fixed point of the map $p\mapsto f(p)$. However, for $C$ sufficiently large, again using Lemma \ref{lem_pi}, for all $p_1,p_2 \in \bd\Pi$:
\begin{equation}
|f(p_1) - f(p_2)| \lesssim C^{-1}|\nu(f(p_1)) - \nu(f(p_2))| = C^{-1}|l^1(p_1) - l^1(p_2)| \lesssim C^{-1}|p_1 - p_2|.
\end{equation}
So, for $C$ sufficiently large to dominate the implicit constants, $p\mapsto f(p)$ is a contraction on $\bd\Pi$ and the contraction mapping principle yields a unique fixed point to the map $p \mapsto f(p)$.

By the proof of Proposition \ref{prop_max}, a point is a local maximum for $D_{cont}$ on $\bd\Pi$ if and only if $l^1(\tilde u) \parallel \nabla\tilde\eta(\tilde u)$. However, there are only two such points for sufficiently large $C$, namely $u^*$ and $\tilde u$. Since $u^* \in B_{C,s_0}$ for any $C,s_0$, it follows that $\tilde u$ is the unique local maximum for $D_{cont}$ on $\bd \Pi \backslash B$.
\end{proof}

To finish the proof of Proposition \ref{prop_improved} on $\bd\Pi\backslash B_{C,s_0}$, it suffices to show the improved estimate of $D_{cont}$ holds at $\tilde u$.
\begin{lemma}\label{lem19}
The point $\tilde u$ satisfies the estimate
\begin{equation} 
D_{cont}(\tilde u) \lesssim -\frac{s_0}{C^2},
\end{equation}
for all $C$ sufficiently big and $s_0$ sufficiently small.
\end{lemma}

\begin{proof}
Let $u(t)$ be the integral curve in the $r_1$ direction emanating from $\tilde u$ and let $\tilde t > 0$ the maximal time such that $u(t)\in \Pi$ for all $0 \le t \le \tilde t$. Then, for any $0 \le t\le \tilde t$, using the explicit computation of $\nabla D_{cont}(u(t)) \cdot r_1(u(t))$ (see Lemma \ref{lem1}), Proposition \ref{prop_cont}, and Assumption \ref{assum} (b)
\begin{equation}\label{eqn_lem19_eqn1}
\begin{aligned}
D_{cont}(\tilde u) &= D_{cont}(u(t)) - \int_0^{t} \nabla D_{cont}(u(s)) \cdot r_1(u(s)) \ ds\\
	&\lesssim -s_0^3 - \int_0^{t} \tilde \eta(u(s)) \left[\nabla\lambda^1(u(s)) \cdot r_1(u(s))\right] \ ds\\
	&\lesssim -s_0^3 + \int_0^{t}\tilde\eta(u(s)) \ ds,
\end{aligned}
\end{equation}
where all implicit constants are independent of $t$.
Next, since $|\nabla\tilde\eta(\tilde u)|\sim s_0$ by Lemma \ref{lem_nu} and $\nu(\tilde u) = -l^1(\tilde u)$ by Lemma \ref{lem18}, by the choice of normalization of eigenvectors in Lemma \ref{lem_normalization}, 
\begin{equation}
\nabla\tilde\eta(\tilde u) \cdot r_1(\tilde u) \sim -s_0.
\end{equation}
Moreover, expanding $\nabla\tilde\eta(u)\cdot r_1(u)$ about $\tilde u$ by Taylor's theorem, there is a point $v$ with $|v - \tilde u| \le |u - \tilde u|$ such that
\begin{equation}
\nabla\tilde\eta(u)\cdot r_1(u) = \nabla\tilde\eta(\tilde u) \cdot r_1(\tilde u) + \left[\nabla^2\tilde\eta(v)r_1(v) + \nabla r_1(v)\nabla\tilde\eta(v)\right] \cdot (v - \tilde u).
\end{equation}
Therefore, we estimate
\begin{equation}
\nabla\tilde \eta(u) \cdot r_1(u) \lesssim -s_0 + Cs_0|u - \tilde u|.
\end{equation}
So, picking $u$ so that $|u-\tilde u| \le K_1C^{-1}$, there is a universal constant $K_2 > 0$ such that,
\begin{equation}
\nabla\tilde \eta(u)\cdot r_1(u)\le (K_1K_2 - K_2)s_0 \le -\frac{K_2}{2}s_0,
\end{equation}
for sufficiently small $K_1$, where $K_1$ is chosen small independent of $C$ and $s_0$. Since $|u(t) - \tilde u| \sim t$, there is a universal constant $K_3$ such that for $t \le K_3C^{-1}$, $|u(t) - \tilde u| \le K_1C^{-1}$ and for such $t$,
\begin{equation}\label{eqn_lem19_eqn2}
\tilde\eta(u(t)) = \int_0^t \nabla\tilde\eta(u(s)) \cdot r_1(u(s)) \ ds \lesssim -s_0t.
\end{equation}
Finally, combining (\ref{eqn_lem19_eqn1}) and (\ref{eqn_lem19_eqn2}), with $t = K_3C^{-1}$, we see that 
\begin{equation}
D_{cont}(\tilde u) \lesssim -s_0^3 + \int_0^{t}\tilde\eta(u(s)) \ ds \lesssim -s_0^3 - C^{-2}s_0 \lesssim -C^{-2}s_0,
\end{equation}
for sufficiently large $C$ and sufficiently small $s_0$.
\end{proof}

\begin{comment}
\begin{flushleft}
\uline{\bf{Proof of Proposition \ref{prop_improved}}}
\end{flushleft}
%\begin{proof}[Proof of Proposition \ref{prop_improved}]
Let $\overline{u} \in \bd\Pi \cup \bd Q_{C,s_0} \backslash B_{C,s_0}$. Then, the preceding sequence of lemmas imply the estimate
\begin{equation}
D_{cont}(\overline{u}) \lesssim -C^{1/3}s_0^3.
\end{equation}
Now, since $\nabla D_{cont}(u) \neq 0$ for $u\in \mathrm{int}(\Pi)$, we conclude
\begin{equation}
\max_{u\in \Pi_{C,s_0} \backslash Q_{C,s_0}} D_{cont}(u) = \max_{u\in \bd\Pi_{C,s_0} \cup \bd Q_{C,s_0} \backslash B_{C,s_0}} D_{cont}(u) \lesssim -C^{1/3}s_0^3.
\end{equation}
%\end{proof}
\end{comment}

%%%%%%%%%%%%%%%%%%%%%%%%%%%%%%%%%%%%%%%%%%%%%%%%%%%%%%%%%%%%%%%%%%%%%%%%%%%%%%%%%%%%%%%%%%%

\begin{flushleft}
\uline{\bf{Proof of Proposition \ref{prop_Qs0}}}
\end{flushleft}
%In this subsection, we use the improved estimate on the continuous entropy dissipation contained in Proposition \ref{prop_improved} to show a uniformly negative estimate on the maximal entropy dissipation outside of $Q_{C,s_0}$.
%\begin{proof}[Proof of Proposition \ref{prop_Qs0}]
Let $Q_{C,s_0}$ be defined as in (\ref{eqn_Qdefn}). Recall the notation, $\F(s,u) = D_{RH}(u,\S^1_u(s),\sigma^1_u(s))$, used in the proof of Lemma \ref{lem15}. By the fundamental theorem of calculus (in $s$), $\F$ satisfies the relation
\begin{equation}\label{eqn_propQ_eqn0}
\F(s,u) = D_{cont}(u) + \int_0^s \dot\sigma^1_u(t)\left[\tilde\eta(u) + \eta(u | \S^1_u(t)) \right] \ dt.
\end{equation}
By Lemma \ref{lem15}, it suffices to estimate, $\F(s^*,u)$, where $s^*(u)$ is the unique maximum of $s\mapsto \F(s,u)$ and $u^+ = \S^1_u(s^*)$. Using the improved estimate of $D_{cont}$ in Proposition \ref{prop_improved}, $\dot{\sigma} \sim -1$ by Lemma \ref{lem_hugoniot}, Lemma \ref{lem_dist}, and Lemma \ref{l2_rel_entropy_lemma}, we obtain
\begin{equation}\label{eqn_propQ_eqn1}
\begin{aligned}
\F(s^*,u) &\lesssim -C^{1/3}s_0^3 -  \int_0^{s^*} \left[\tilde\eta(u) + \eta(u | \S^1_u(t)) \right] \ dt\\
	&\lesssim - C^{1/3}s_0^3 -  \int_0^{s^*} |u - \S^1_u(t)|^2 - s_0 d(u,\bd \Pi) \ dt\\
	&\lesssim - C^{1/3}s_0^3 - (s^*)^3 + s_0 s^* d(u,\bd\Pi).
\end{aligned}
\end{equation}
Next, since $u^+ \notin \Pi$ by assumption, $d(u,\bd\Pi) \lesssim |u - u^+| \lesssim s^*$. Also, as $s^*$ is completely characterized by the equation $-\tilde\eta(u) = \eta(u | u^+)$, we have
\begin{equation}\label{eqn_propQ_eqn2}
(s^*)^2 \lesssim |u - u^+|^2 \lesssim \eta(u | u^+) = -\tilde\eta(u) \lesssim s_0 d(u,\bd\Pi) \lesssim s_0s^*,
\end{equation}
which yields $d(u,\bd \Pi) \lesssim s^* \lesssim s_0$. Combined with (\ref{eqn_propQ_eqn1}), we obtain
\begin{equation}
D_{max}(u) = \F(s^*,u) \lesssim -s_0^3,
\end{equation}
provided $C$ is sufficiently large and $s_0$ is sufficiently small.
%\end{proof}

%%%%%%%%%%%%%%%%%%%%%%%%%%%%%%%%%%%%%%%%%%%%%%%%%%%%%%%%%%%%%%%%%%%%%%%%%%%%%%%%%%%%%%%%%%%%%%%

\subsection{Step 1.2: Shocks inside \texorpdfstring{$Q_{C,s_0}$}{QCs0} but outside \texorpdfstring{$R_{C,s_0}$}{RCs0}}

In this step, we show $D_{max}(u) \le 0$ except on possibly a region localized about $u_L$ to order $s_0$. In particular, we improve Proposition $\ref{prop_Qs0}$ in two ways: First, we remove the assumption that $u^+\notin \Pi_{C,s_0}$. Second, we show that there is a universal constant $K_h$ such that for $u \in Q_{C,s_0}$ with $|u^* - u| > K_hs_0$, $D_{max}(u) \le 0$. Although these improvements look rather disconnected, we will show $u^+ \in \Pi$ for $d(u,\bd\Pi) \gtrsim s_0$, which implies this step can be rephrased as showing $D_{max}(u) \le 0$ for $d(u,\bd\Pi) \gtrsim s_0$. In other words, in this step, we handle states far in the interior of $\Pi_{C,s_0}$. More precisely, we define the family of cylinders, $R_{C,s_0}(K_h)$, as
\begin{equation}\label{eqn_defn_R}
R_{C,s_0}(K_h) := \left\{u = u^* - \sum_{i=1}^n b_i r_i(u^*) \ | \ u\in \Pi_{C,s_0}, \ |b_1| \le K_hs_0, \ \left(\sum_{i=2}^n |b_i|^2\right)^{1/2} \le C^{-1/3}s_0\right\}.
\end{equation}
With this language, the main proposition of this step becomes: 
\begin{proposition}\label{prop_R}
There is a $K_h > 0$ depending only on the system such that for any $C$ sufficiently large and $s_0$ sufficiently small and any $u \notin R_{C,s_0}(K_h)$,
\begin{equation}
D_{max}(u) \lesssim -s_0^3.
\end{equation}
\end{proposition}

We first prove the following lemma, which bounds the entropy dissipation along the maximal shock $(u,u^+,\sigma_{\pm})$ for $u,u^+\in\Pi$ by $D_{cont}(u^+)$. 
\begin{lemma}\label{lem20}
Suppose that the maximal shock $(u,u^+,\sigma_\pm)$ satisfies $u,u^+ \in \Pi$. Then,
\begin{equation}
D_{RH}(u,u^+,\sigma_\pm) \le D_{cont}(u^+).
\end{equation}
Moreover, for $C$ sufficiently large and $s_0$ sufficiently small,
\begin{equation}
D_{RH}(u,u^+,\sigma_\pm) \le D_{cont}(u^+) \lesssim -s_0^3.
\end{equation}
\end{lemma}

\begin{proof}
First, by the entropy inequality (\ref{entropy}),
\begin{equation}
q(u^+ ; u_L) - \sigma_\pm\eta(u^+ | u_L) \le q(u ; u_L) - \sigma_\pm\eta(u | u_L).
\end{equation}
Second, we bound the entropy dissipation as
\begin{equation}
D_{RH}(u,u^+,\sigma_\pm) \le \left[q(u^+ ; u_R) - \sigma_\pm \eta(u^+ | u_R)\right] - (1+ Cs_0)\left[q(u^+ ; u_L) - \sigma_\pm\eta(u^+ | u_L)\right].
\end{equation}
Next, adding and subtracting $\lambda^1(u_+)\eta(u_+ | u_{R,L})$,
\begin{equation}
D_{RH}(u,u^+,\sigma_\pm) \le D_{cont}(u^+) + (\sigma_\pm - \lambda^1(u^+))\left[-\eta(u^+ | u_R) + (1 + Cs_0) \eta(u^+ | u_L)\right] \le D_{cont}(u^+),
\end{equation}
where in the last inequality we have used $u^+ \in \Pi$ and that $(u,u^+,\sigma_\pm)$ is Lax admissible.
\end{proof}

We continue with the following lemma, which proves that for $-\tilde\eta(u) \gtrsim s_0^2$, $u^+ \in \Pi$ and $D_{max}(u) \le 0$. Combined with Lemma $\ref{lem_dist}$, this provides a formal justification that for $d(u,\bd\Pi) \gtrsim s_0$, $u^+ \in \Pi$ and, therefore, for $u$ sufficiently far from the boundary of $\Pi$, $D_{max}(u)$ is uniformly negative.
\begin{lemma}\label{lem21}
There is a universal constant $K^*$ such that for all $C$ sufficiently big and $s_0$ sufficiently small, for $u \in \Pi_{C,s_0}$ satisfying
$$-\tilde\eta(u) \ge K^* s_0^2,$$
we have $u^+\in \Pi$ and, consequently, for such $u$,
$$D_{max}(u) \lesssim -s_0^3.$$
\end{lemma}

\begin{proof}
Suppose $-\tilde\eta(u) \ge K^*s_0^2$ and $u^+ = \S^1_u(s^*)$ is the maximal shock and assume $u^+ \notin \Pi$. Then, for $K_1$ the universal constant from Lemma \ref{lem_dist},
\begin{equation}
K^*s_0^2 \le -\tilde\eta(u) \le K_1s_0 d(u,\bd \Pi) \le K_1 s_0 |u - u^+|.
\end{equation}
Thus, $K^*s_0 \le K_1|u - u^+|$. Now, for $K_2$ the universal constant from Lemma \ref{l2_rel_entropy_lemma}, the characterization of $s^*$ in Lemma \ref{lem16} implies
\begin{equation}
K_2|u - u^+|^2 \le \eta(u | u^+) = -\tilde\eta(u) \le K_1 s_0|u - u^+|.
\end{equation}
Thus, $|u- u^+| \le \frac{K_1}{K_2}s_0$. We conclude that
\begin{equation}
K^*s_0 \le \frac{K_1^2}{K_2}s_0.
\end{equation}
Therefore, choosing $K^*$ strictly larger than $K_1^2K_2^{-1}$, we note that $-\tilde\eta(u) \ge K^*s_0^2$ implies $u^+\in\Pi_{C,s_0}$.
%On the other hand,
%\begin{equation}
%K^*s_0^2 \le -\tilde\eta(u) = \eta(u | u^+) \le K_2|u - u^+|^2.
%\end{equation}
%Taking $K^*$ sufficiently large so that 
%\begin{equation}
%\frac{K_1}{K_2} < \sqrt{\frac{K^*}{K_2}},
%\end{equation}
%it follows that
%\begin{equation}
%|u - u^+| \le \frac{K_1}{K_2} s_0 < s_0\sqrt{\frac{K^*}{K_2}} \le |u - u^+|,
%\end{equation}
%which is certainly a contradiction. Therefore, for $K^*$ sufficiently large (independent of $C$, $s_0$), we conclude that $u^+\in \Pi$. 
Finally, applying Lemma \ref{lem20}, we have
\begin{equation}
D_{max}(u) \le D_{cont}(u^+) \lesssim -s_0^3.
\end{equation}
\end{proof}

\begin{flushleft}
\uline{\bf{Proof of Proposition \ref{prop_R}}}
\end{flushleft}
%\begin{proof}[Proof of Proposition \ref{prop_R}]
We now take $u \in Q_{C,s_0}$ so that $u$ is of the form
\begin{equation}
u = u^* - \sum_{i=1}^n b_i r_i(u^*) \qquad\text{where}\qquad \left(\sum_{i=2}^n |b_i|^2\right)^{1/2} \le C^{-1/3}s_0 \quad\text{and}\quad |b_1|\le C^{1/6}s_0.
\end{equation}
We will show that for $u\in Q_{C,s_0}$, there exists a universal constant $K > 0$ such that for $b_1 > Ks_0$ and $C$ sufficiently large, $s_0$ sufficiently small, $-\tilde\eta(u) \ge K^*s_0^2$, where $K^*$ is the universal constant from Lemma \ref{lem21}, which combined with Lemma \ref{lem20} and Proposition \ref{prop_Qs0} completes the proof of the existence of $R_{C,s_0}(K_h)$ with $K_h = K$.
%By the definition of $Q_{C,s_0}$ and Proposition \ref{prop_Qs0}, we conclude that it suffices to show our estimate for $s_0 \lesssim b_1 \lesssim C^{1/6}s_0$.

Therefore, we suppose $u\in Q_{C,s_0}$ with $b_1 > Ks_0$ and, using the fundamental theorem of calculus, rewrite $-\tilde\eta(u)$ as
\begin{equation}\label{eqn_propR_eqn1}
-\tilde\eta(u) = -\tilde\eta(u^*) + \int_{u^*}^u \nabla\tilde\eta(v) \ dv.
\end{equation}
Next, noting $\tilde\eta(u^*) = 0$ and $|u-u_L| \lesssim C^{1/6}s_0$, we compute
\begin{equation}\label{eqn_propR_eqn2}
\nabla\tilde\eta(u) = Cs_0\left[\nabla \eta(u) - \nabla\eta(u_L)\right] + \left[\nabla\eta(u_R) - \nabla\eta(u_L)\right] = s_0\nabla^2\eta(u_L)r_1(u_L) + \bigO(C^{7/6}s_0^2),
\end{equation}
Parametrizing $v$ as $v(t) = tu^* + (1-t)u$, we have $\dot{v}(t) = u^* - u = \sum_{i=1}^n b_i r_i(u^*)$. Combined with (\ref{eqn_propR_eqn1}) and (\ref{eqn_propR_eqn2}), we obtain
\begin{equation}\label{eqn_propR_eqn3}
\begin{aligned}
-\tilde\eta(u) &= \int_0^1 \left[s_0\nabla^2\eta(u_L)r_1(u_L) + \bigO(C^{7/6}s_0^2)\right]\cdot\left[\sum_{i=1}^n b_ir_i(u^*)\right]\\
	&= \int_0^1b_1s_0 \nabla^2\eta(u_L)r_1(u_L) \cdot r_1(u^*) + \bigO(C^{4/3}s_0^3 + C^{-1/3}s_0^2)\\
	&= b_1s_0 \nabla^2\eta(u_L)r_1(u_L)\cdot r_1(u_L) + \bigO(C^{-1/3}s_0^2 + |b_1|s_0^2).
\end{aligned}
\end{equation}
Since $\nabla^2\eta(u_L)r_1(u_L) \cdot r_1(u_L) \ge K_1 > 0$, a universal constant, and $b_1 \ge Ks_0$, (\ref{eqn_propR_eqn3}) can be reformulated as the estimate
\begin{equation}
-\tilde\eta(u) \ge KK_1s_0^2 + \bigO(C^{-1/3}s_0^2 + |b_1|s_0^2).
\end{equation}
Now, as $|b_1| \le C^{1/6}s_0$, we may take $C$ sufficiently large and $s_0$ sufficiently small to make the error smaller than $s_0^2$ and obtain
\begin{equation}
-\tilde\eta(u) \ge (KK_1-1)s_0^2.
\end{equation}
Therefore, choosing $K \ge \frac{K^* + 1}{K_1}$ yields $-\tilde\eta(u) \ge K^*s_0^2$ for all $u \in Q_{C,s_0}$ with $b_1 > Ks_0$. By the discussion at the beginning of the proof, taking $K_h = K$ completes the construction of $R_{C,s_0}(K_h)$.
%\end{proof}

%%%%%%%%%%%%%%%%%%%%%%%%%%%%%%%%%%%%%%%%%%%%%%%%%%%%%%%%%%%%%%%%%%%%%%%%%%%%%%%%%%%%%%%%%%%%%%%%%%%%%%%%%

\subsection{Step 2.1: Shocks close to \texorpdfstring{$\bd \Pi$}{bd(Pi)}}
In this step, we show that states $u$ near to $u^*$ (of order $s_0$) satisfy $D_{max}(u) \le 0$. Many estimates used in the following steps crucially rely on a lower bound on the strength of the maximal shock $(u,u^+,\sigma_\pm)$. However, for $u\in \bd \Pi$, $u^+ = u$ and the necessary estimates degenerate near $\bd\Pi$. Therefore, we handle this regime more in the flavor of Proposition \ref{prop_Qs0}. More precisely, we prove the following proposition:
\begin{proposition}\label{prop_Rbd}
There is a universal constant $K_{bd} > 0$ such that for any $C$ sufficiently large and $s_0$ sufficiently small (depending on $C$), defining 
\begin{equation}\label{defn_Rbd}
R_{C,s_0}^{bd} := \set{u\in R_{C,s_0} \ \biggr| \ |u-u^*| \le K_{bd}s_0}
\end{equation}
for any $u\in R_{C,s_0}^{bd}$, $D_{max}(u) \lesssim -s_0^3$.
\end{proposition}

\begin{proof}
This proof is a minor variant of the proof of Proposition \ref{prop_Qs0}. We recall from (\ref{eqn_propQ_eqn0}),
\begin{equation}
D_{max}(u) = D_{cont}(u) + \int_0^{s^*(u)} \dot\sigma_u^1(t) \left[\tilde\eta(u) - \eta(u | \S^1_u(t))\right] \ dt.
\end{equation}
Following (\ref{eqn_propQ_eqn1}) and (\ref{eqn_propQ_eqn2}), in the proof of Proposition \ref{prop_Qs0}, we estimate
\begin{equation}
D_{max}(u) \lesssim -s_0^3 + s_0s^* d(u,\bd\Pi) \lesssim -s_0^3 + s_0^2d(u,\bd\Pi).
\end{equation}
Therefore, there are universal constants $K_{bd},K > 0$ such that if $d(u,\bd\Pi) \le \frac{1}{2}K_{bd}s_0$, then $D_{max}(u) \le -Ks_0^3$. Finally, we note that if $d(u,\bd\Pi) \le \frac{1}{2}K_{bd}s_0$ for $u\in R_{C,s_0}$, then $|u-u^*| \le K_{bd}s_0$ for all sufficiently large $C$ and sufficiently small $s_0$.
\end{proof}

%%%%%%%%%%%%%%%%%%%%%%%%%%%%%%%%%%%%%%%%%%%%%%%%%%%%%%%%%%%%%%%%%%%%%%%%%%%%%%%%%%%%%%%%%%%%%%%%%%%%%%%%%%%%%%

\subsection{Step 2.2: Shocks close to \texorpdfstring{$u_L$}{uL}}

In this step, we prove that states $u$ near to $u_L$ (of order $s_0$) satisfy $D_{max}(u) \le 0$. Since $D_{max}(u_L) = 0$, we lose the uniform negativity of the estimates until this point. However, we show the existence of a region $R_{C,s_0}^0$ of size of order $s_0$ in the $r_1$ direction on which $u_L$ is a strict local maximum and which satisfies $D_{max}(u) \lesssim -s_0^3$ on the boundary of $R_{C,s_0}^0$. Our proof of Proposition \ref{prop_shock} then implies $(u_L,u_R,\sigma_{LR})$ is the unique entropic $1$-shock with left state in $\Pi_{C,s_0}$ for which $D_{RH}$ vanishes. More precisely, the main result of this step is:
\begin{proposition}\label{prop_R0}
There is a universal constant $K^0$ such that for any $C$ sufficiently large and $s_0$ sufficiently small (depending on $C$), defining 
\begin{equation}\label{defn_R0}
R_{C,s_0}^0 := \set{u\in R_{C,s_0}\ \biggr| \ |u-u_L|\le K_0s_0} 
\end{equation}
for any $u\in R_{C,s_0}^0$, $D_{max}(u) \le 0$.
\end{proposition}

We begin by studying $D_{max}$ and computing derivatives. We recall that $D_{max}(u) = D_{RH}(u,u^+,\sigma_\pm)$ where $u^+$ and $\sigma_{\pm}$ are uniquely defined via the conditions
\begin{equation}\label{eqn_u+_defn}
\begin{aligned}
f(u^+) - f(u) &= \sigma_{\pm}(u^+ - u)\\
-\tilde\eta(u) &= \eta(u | u^+)\\
\sigma_{\pm} &\le \lambda^1(u).
\end{aligned}
\end{equation} 
Now, we have the following algebraic identities:
\begin{lemma}\label{lem22}
Let $u^+ = \S^1_u(s^*)$ be such that $D_{max}(u) = D(u,u^+,\sigma_{\pm})$. Then, the following identities hold:
\begin{align}
&\left[f^\prime(u^+) - \sigma_{\pm}I\right]\nabla u^+ = \left[f^\prime(u) -\sigma_\pm I\right] + (u^+ - u)\tens\sigma_{\pm}^\prime\label{eqn_lem22_RK},\\
&\left[\nabla\eta(u_R) - \nabla\eta(u^+)\right] + (1 + Cs_0)\left[\nabla\eta(u) - \nabla\eta(u_L)\right] = (u - u^+)^t\left[\nabla^2\eta(u^+)\nabla u^+\right],\label{eqn_lem22_maximality}
\end{align}
and
\begin{equation}\label{eqn_lem22_grad}
\nabla D_{max}(u) = \left[\nabla\eta(u^+) - \nabla\eta(u_R) - (1+Cs_0)\left[\nabla\eta(u) - \nabla\eta(u_L)\right]\right]\left(f^\prime(u) - \sigma_{\pm}I\right).
\end{equation}
\end{lemma}

\begin{proof}
First, we differentiate the two implicit equations for $u^+$ in (\ref{eqn_u+_defn}). Differentiating in $u$ and viewing $u^+$ and $\sigma_{\pm}$ as functions of $u$, we obtain
\begin{equation}
f^\prime(u^+)\nabla u^+ - f^\prime(u) = (u^+ - u)\tens\sigma_{\pm}^\prime + \sigma_{\pm}(\nabla u^+ - I)
\end{equation}
and
\begin{equation}
\begin{aligned}
\nabla\eta(u) - \nabla\eta(u^+)\nabla u^+ - (u - u^+)^t\nabla^2\eta(u^+)\nabla u^+ - \nabla\eta(u^+)(Id - \nabla u^+) = - &Cs_0\left[\nabla \eta(u) - \nabla\eta(u_L)\right]\\
	&+ \left[\nabla \eta(u_L) - \nabla\eta(u_R)\right].
\end{aligned}
\end{equation}
Rearranging terms yields the desired identities (\ref{eqn_lem22_RK}) and (\ref{eqn_lem22_maximality}), respectively.

Second, we differentiate the relation $D_{max}(u) = D_{RH}(u,u^+,\sigma_{\pm})$ using the chain rule:
\begin{equation}\label{eqn_lem22_eqn0}
\begin{aligned}
\nabla D_{max}(u) = &\left[\nabla\eta(u^+) - \nabla\eta(u_R)\right]\left[f^\prime(u^+) - \sigma_{\pm}I\right]\nabla u^+\\
	& - (1+Cs_0)\left[\nabla\eta(u) - \nabla\eta(u_L)\right]\left[f^\prime(u) - \sigma_{\pm}I\right]\\
	&-\sigma_{\pm}^\prime\left[\eta(u^+ | u_R) - (1+Cs_0)\eta(u | u_L)\right].
\end{aligned}
\end{equation} 
Next, we apply (\ref{eqn_lem22_RK}) to simplify the first line of (\ref{eqn_lem22_eqn0}) as
\begin{equation}\label{eqn_lem22_eqn3}
\begin{aligned}
\nabla D_{max}(u) = &\left[\nabla\eta(u^+) - \nabla\eta(u_R) - (1+Cs_0)\left(\nabla\eta(u) - \nabla\eta(u_L)\right)\right]\left[f^\prime(u) - \sigma_{\pm}I\right]\\
	&+\left[\nabla\eta(u^+) -\nabla\eta(u_R)\right](u^+ - u)\tens\sigma_{\pm}^\prime -\sigma_{\pm}^\prime\left[\eta(u^+ | u_R) - (1+Cs_0)\eta(u | u_L)\right]
\end{aligned}
\end{equation}
and use the definition of the tensor product to simplify the second line of (\ref{eqn_lem22_eqn3}) as
\begin{equation}
\begin{aligned}\label{eqn_lem22_eqn1}
\nabla D_{max}(u) = &\left[\nabla\eta(u^+) - \nabla\eta(u_R) - (1+Cs_0)\left(\nabla\eta(u) - \nabla\eta(u_L)\right)\right]\left[f^\prime(u) - \sigma_{\pm}I\right]\\
	&+\sigma_{\pm}^\prime\left[\left(\nabla\eta(u^+) -\nabla\eta(u_R)\right) \cdot (u^+ - u) - \eta(u^+ | u_R) + (1+Cs_0)\eta(u | u_L)\right].
\end{aligned}
\end{equation}
Then, rearranging terms in the definition of relative entropy, we have the identity:
\begin{equation}\label{eqn_lem22_eqn2}
\begin{aligned}
\eta(u | u^+) - \eta(u | u_R) &= \left[\eta(u) - \eta(u^+) - \nabla\eta(u^+)\cdot (u - u^+)\right] - \left[\eta(u) - \eta(u_R) - \nabla\eta(u_R)\cdot(u - u_R)\right]\\
	&= -\left[\eta(u^+) - \eta(u_R) - \nabla\eta(u_R)\cdot(u^+ - u_R)\right] + \left[\nabla\eta(u^+) - \nabla\eta(u_R)\right] \cdot (u^+ - u)\\
	&= -\eta(u^+ | u_R) + \left[\nabla\eta(u^+) - \nabla\eta(u_R)\right] \cdot (u^+ - u).
\end{aligned}
\end{equation}
Finally, combining (\ref{eqn_lem22_eqn1}) and (\ref{eqn_lem22_eqn2}) yields
\begin{equation}
\begin{aligned}
&\sigma_{\pm}^\prime\left[\left(\nabla\eta(u^+) -\nabla\eta(u_R)\right) \cdot (u^+ - u) - \eta(u^+ | u_R) + (1+Cs_0)\eta(u | u_L)\right]\\
	&\qquad= \sigma^\prime_\pm\left[\eta(u | u^+) - \eta(u | u_R) + (1 + Cs_0)\eta(u | u_L)\right]\\
	&\qquad=\sigma^\prime_\pm\left[\eta(u | u^+) + \tilde\eta(u)\right] = 0,
\end{aligned}
\end{equation}
which completes the proof of (\ref{eqn_lem22_grad}).
\end{proof}

%\begin{lemma}\label{lem23}
%Let $u^+ = \S^1_u(s^*)$ be such that $D_{max}(u) = D_{RH}(u,u^+,\sigma_{\pm})$. Then, the following representation of $\nabla D_{max}$ holds:
%\begin{equation}\label{eqn_lem23_desired1}
%\nabla D_{max}(u) = \left[\nabla\eta(u^+) - \nabla\eta(u_R) - (1+Cs_0)\left[\nabla\eta(u) - \nabla\eta(u_L)\right]\right]\left(f^\prime(u) - \sigma_{\pm}I\right).
%\end{equation}
%Moreover, we also have
%\begin{equation}\label{eqn_lem23_desired2}
%\nabla^2 D_{max}(u) = \left[\nabla^2\eta(u^+)\nabla u^+ - (1+ Cs_0)\nabla^2\eta(u)\right](f^\prime(u) - \sigma_\pm I) + (u^+ - u)^t(\nabla^2\eta(u^+)\nabla u^+)(f^{\prime\prime} - I\tens \sigma_{\pm}^\prime).
%\end{equation}
%\end{lemma}

Next, we prove several rough estimates on $u^+$ and $\sigma_\pm$, using only that states $u\in R_{C,s_0}\backslash R_{C,s_0}^{bd}$ are separated from $\bd \Pi$.
\begin{lemma}\label{lem23}
For all $C$ sufficiently large and $s_0$ sufficiently small, for any $u\in R_{C,s_0}\backslash R_{C,s_0}^{bd}$, we have
\begin{equation}\label{eqn_lem23_desired1}
|u - u^+| \sim s_0 \quad \text{and} \quad \lambda^1(u) - \sigma_{\pm} \sim s_0 \quad \text{and} \quad \sigma_{\pm} - \lambda^1(u^+) \sim s_0.
\end{equation}
Furthermore, we have the derivative estimates
\begin{equation}\label{eqn_lem23_desired2}
|\nabla u^+| + |\sigma_\pm^\prime| \lesssim 1 \quad \text{and} \quad |\nabla^2 u^+| + |\sigma_\pm^{\prime\prime}| + \lesssim s_0^{-1}.
\end{equation}
\end{lemma}

\begin{proof}
We begin with (\ref{eqn_lem23_desired1}). We estimate $|u - u^+|$ using Lemma \ref{l2_rel_entropy_lemma}, Lemma \ref{lem16}, and Lemma \ref{lem_nu} via
\begin{equation}
|u - u^+|^2 \sim \eta(u | u^+) = - \tilde\eta(u) \lesssim s_0d(u,\bd\Pi) \lesssim s_0^2,
\end{equation}
which proves $|u - u^+| \lesssim s_0$. For the reverse direction, we show $-\tilde\eta(u) \gtrsim s_0d(u,\bd\Pi)$. Indeed, take $\overline{u}\in\bd\Pi$ so that $|u - \overline{u}| = d(u,\bd\Pi)$. Then, by Taylor's theorem and Lemma \ref{lem_eta},
\begin{equation}
|\tilde\eta(u) - \nabla\tilde\eta(\overline{u})\cdot (u - \overline{u})| \lesssim Cs_0^3.
\end{equation} 
Next, $\nabla\tilde\eta(\overline{u}) = b_1 l^1(u^*) + \bigO(s_0^2)$ and $u-\overline{u} = c_1r_1(u^*) + \bigO(s_0^2)$, where $b_1 \sim s_0$ by Lemma \ref{lem_nu} and $c_1 \sim s_0$ because $u \notin R_{C,s_0}^{bd}$.
Thus, $-\tilde\eta(u) \gtrsim s_0^2$, which completes the proof of the first estimate of (\ref{eqn_lem23_desired1}). Note since $|u - u^+| \sim s^*$, we have shown $s^* \sim s_0$ on $R_{C,s_0} \backslash R_{C,s_0}^{bd}$.
Now, by Lemma \ref{lem_hugoniot},
\begin{equation}\label{eqn_lem23_eqn0}
\sigma_\pm = \frac{1}{2}\left(\lambda^1(u) + \lambda^1(u^+)\right) + \bigO(|u-u^+|^2).
\end{equation}
Using $(u,u^+,\sigma_\pm)$ is Lax admissible, (\ref{eqn_lem23_eqn0}) implies (\ref{eqn_lem23_desired1}).

To obtain the derivative estimates of $u^+$ and $\sigma_\pm$ in (\ref{eqn_lem23_desired2}), we estimate $\nabla s^*(u)$ and $\nabla^2 s^*(u)$. In particular, since $\eta(u | \S^1_u(s^*(u)) = -\tilde\eta(u)$,
\begin{equation}\label{eqn_lem23_eqn1}
\frac{d}{ds}\biggr|_{s = s^*} \eta(u | \S^1_u(s)) \nabla s^*(u) = -\nabla\tilde\eta(u) - \left[\nabla\eta(u) - \nabla\eta(u^+) - \nabla^2\eta(u^+)\nabla_u\S^1_u(s^*)(u - u^+)\right] := \mathcal{E}
\end{equation}
Thus, using $u\mapsto \S^1_u(s)$ is $C^1$ by Lemma \ref{lem_hugoniot}, $s^* \sim |u - u^+| \sim s_0$ on $R\backslash R^{bd}$, and Lemma \ref{lem_eta}, we have $|\mathcal{E}| \lesssim s_0$. On the other hand, using the precise asymptotics of $s\mapsto \S^1_u(s)$ in Lemma \ref{lem_hugoniot} and $\eta$ is strictly convex, we compute
\begin{equation}\label{eqn_lem23_eqn2}
\begin{aligned}
\frac{d}{ds}\biggr|_{s = s^*} \eta(u | \S^1_u(s)) &= -\nabla^2 \eta(\S^1_u(s^*)) \frac{d}{ds}\S^1_u(s^*) \cdot (u - u^+)\\
	&= s^*\nabla^2\eta(u)r_1(u) \cdot r_1(u) + \bigO(|u - u^+|^2) \gtrsim s_0.
\end{aligned}
\end{equation}
Combining (\ref{eqn_lem23_eqn1}) and (\ref{eqn_lem23_eqn2}) yields the bound $|\nabla s^*| \lesssim 1$. Differentiating the relations
\begin{equation}\label{eqn_lem23_eqn6}
u^+ = \S^1_u(s^*(u)) \qquad \text{and} \qquad \sigma_\pm = \sigma_u^1(s^*(u))
\end{equation}
using the chain rule and again appealing to Lemma \ref{lem_hugoniot}, we obtain the first part of (\ref{eqn_lem23_desired2}), namely, the bounds $|\nabla u^+| \lesssim 1$ and $|\sigma_\pm^\prime| \lesssim 1$. Next, differentiating (\ref{eqn_lem23_eqn1}) yields
\begin{equation}\label{eqn_lem23_eqn3}
\frac{d^2}{ds^2}\biggr|_{s = s^*} \eta(u | \S^1_u(s)) \nabla s^*(u) \tens \nabla s^*(u) + \frac{d}{ds}\biggr|_{s = s^*} \eta(u | \S^1_u(s)) \nabla^2 s^*(u) = \nabla \mathcal{E}.
\end{equation}
Using $u\mapsto \S^1_u(s)$ is $C^2$ by Lemma \ref{lem_hugoniot}, $s^*\sim s_0$ on $R\backslash R^{bd}$, Lemma \ref{lem_eta}, and $|\nabla u^+|\lesssim 1$, we compute
\begin{equation}\label{eqn_lem23_eqn4}
|\nabla\mathcal{E}| = \left|-\nabla^2\tilde\eta(u) - \nabla^2\eta(u) + \nabla^2\eta(u^+)\nabla u^+ - \nabla\left[\nabla^2\eta(u^+)[\nabla_u\S^1_u](s^*(u))(u - u^+)\right]\right|\lesssim 1.
\end{equation}
Again, using Lemma \ref{lem_hugoniot} so that $s \mapsto \S^1_u(s)$ is $C^2$ and $s^*\sim s_0$ on $R\backslash R^{bd}$, we compute
\begin{equation}\label{eqn_lem23_eqn5}
\begin{aligned}
\frac{d^2}{ds^2}\biggr|_{s = s^*} \eta(u | \S^1_u(s)) &= \nabla^3 \eta(u^+) \frac{d}{ds}\S^1_u(s^*)\frac{d}{ds}\S^1_u(s^*) \cdot (u^+ - u) + \nabla^2\eta(u^+)\frac{d^2}{ds^2}\S^1_u(s^*) \cdot (u^+ - u)\\
	&\qquad+ \nabla^2\eta(u^+)\frac{d}{ds}\S^1_u(s^*)\cdot \frac{d}{ds}\S^1_u(s^*)\\
	&= \nabla\eta^2(u)r_1(u) \cdot r_1(u) + \bigO(|u - u^+|) \lesssim 1.
\end{aligned}
\end{equation}
Combining (\ref{eqn_lem23_eqn3}), (\ref{eqn_lem23_eqn4}), and (\ref{eqn_lem23_eqn5}) with the estimates $|\nabla s^*| \lesssim 1$ and $\frac{d}{ds}\S^1_u(s^*) \gtrsim s_0$, we obtain
\begin{equation}
|\nabla^2 s^*(u)| \lesssim \frac{|\nabla\mathcal{E}| + \left|\frac{d^2}{ds^2} \eta(u | \S^1_u(s^*))\right| |\nabla s^*(u)|^2}{\left|\frac{d}{ds}\S^1_u(s^*)\right|} \lesssim s_0^{-1}.
\end{equation}
Finally, the second derivative estimates in (\ref{eqn_lem23_desired2}) follow by differentiating (\ref{eqn_lem23_eqn0}) twice using the chain rule and appealing to the estimates on $s^*$ and Lemma \ref{lem_hugoniot}.
\end{proof}

%Second, we use the identity (\ref{eqn_lem22_RK}) to write
%\begin{equation}
%\left[f^\prime(u^+) - \sigma_{\pm}Id\right]\nabla u^+ = \left[f^\prime(u^+) - \sigma_{\pm}Id\right] + (u^+ - u)\tens \sigma_{\pm}^\prime + \left[f^\prime(u) - f^\prime(u^+)\right].
%\end{equation}
%Therefore, left multiplying by $l^i(u^+)$, 
%\begin{equation}\label{eqn_lem_eqn1}
%l^i(u^+)\nabla u^+ = l^i(u^+) + (\lambda^i(u^+) - \sigma_{\pm})^{-1}\bigO(|u^+ - u|).
%\end{equation}

Before proving Proposition \ref{prop_R0}, we prove a last lemma which uses more precise information about the maximal shock $(u,u^+,\sigma_\pm)$ by combining the algebraic identities in Lemma \ref{lem22} with the estimates in Lemma \ref{lem23}. The role of the following lemma is to guarantee that the region $R^0_{C,s_0}$, on which $u_L$ is a local maximum, is sufficiently large.
\begin{lemma}\label{lem24}
Suppose $u\in R_{C,s_0}\backslash R_{C,s_0}^{bd}$. Then, for all sufficiently large $C$ and sufficiently small $s_0$, we have the third derivative bound $|\nabla^3D_{max}(u)|\lesssim 1$.
\end{lemma}

\begin{proof}
Differentiating (\ref{eqn_lem22_grad}) once yields the second derivative formula
\begin{equation}\label{eqn_lem24_eqn1}
\begin{aligned}
\nabla^2 D_{max}(u) &= \left[\nabla^2\eta(u^+)\nabla u^+ - (1+Cs_0)\nabla^2\eta(u)\right]\left(f^\prime(u) - \sigma_\pm I\right)\\
	&+ \left[\nabla\eta(u^+)-\nabla\eta(u_R) - (1+Cs_0)\left(\nabla\eta(u) - \nabla\eta(u_L)\right)\right]\left(f^{\prime\prime}(u) - I \tens \sigma_\pm^\prime\right).
\end{aligned}
\end{equation}
Differentiating (\ref{eqn_lem24_eqn1}) once more yields the formula
\begin{equation}\label{eqn_lem24_eqn2}
\begin{aligned}
\nabla^3 D_{max}(u) &= \left[\nabla^3\eta(u^+)\nabla u^+ \nabla u^+ + \nabla^2\eta(u^+)\nabla^2 u^+ - (1+Cs_0)\nabla^3\eta(u)\right]\left(f^\prime(u) - \sigma_\pm I\right)\\
	&\qquad+2\left[\nabla^2\eta(u^+)\nabla u^+ - (1+Cs_0)\nabla^2\eta(u)\right]\left(f^{\prime\prime}(u) - I \tens \nabla\sigma_\pm\right)\\
	&\qquad +\left[\nabla \eta(u^+) - \nabla\eta(u_R) - (1+Cs_0)\left(\nabla\eta(u) - \nabla\eta(u_L)\right)\right]\left(f^{\prime\prime\prime}(u) - I \tens \nabla^2\sigma_\pm\right)\\
	&:= \left[\nabla \eta(u^+) - \nabla\eta(u_R) - (1+Cs_0)\left(\nabla\eta(u) - \nabla\eta(u_L)\right)\right] \left(I \tens \nabla^2\sigma_\pm\right)\\
	&+\left[\nabla^2\eta(u^+)\nabla^2 u^+ \left(f^\prime(u) - \sigma_\pm I\right)\right] + \mathcal{M}\\
	&:= \mathcal{E}_1 + \mathcal{E}_2 + \mathcal{M}
\end{aligned}
\end{equation}
We have been rather vague with the exact meaning of the various products of tensors above. While we will not need any precise structure of the terms in $\mathcal{M}$, we will need to estimate $\mathcal{E}_1$ and $\mathcal{E}_2$ very carefully and will expand on their structure shortly. Indeed, each term collected in $\mathcal{M}$ depends on only $u^+$, $\nabla u^+$, $\sigma_\pm$, and $\sigma_\pm^\prime$ and can be bounded using only the estimates in Lemma \ref{lem23}. So, we have $|\mathcal{M}| \lesssim 1$. 

Next, let us bound $\mathcal{E}_1$. Because $|u^+ - u| \lesssim s_0$, $|u_R - u_L| \lesssim s_0$, and $|u - u_L| \lesssim s_0$,
\begin{equation}\label{eqn_lem24_eqn3}
\left|\nabla \eta(u^+) - \nabla\eta(u_R) - (1+Cs_0)\left(\nabla\eta(u) - \nabla\eta(u_L)\right)\right| \lesssim s_0 + Cs_0^2 \lesssim s_0.
\end{equation}
Combining (\ref{eqn_lem24_eqn3}) with the estimate $|\sigma_\pm^{\prime\prime}| \lesssim s_0^{-1}$ from Lemma \ref{lem23} yields $|\mathcal{E}_1| \lesssim 1$.

It remains only to bound $\mathcal{E}_2$, which is written more precisely in index notation as
\begin{equation}\label{eqn_lem24_eqn4}
\mathcal{E}_2 = \partial_i\partial_j \eta(u^+)\partial_l\partial_k u^+_j (\partial_k f_m(u) - \sigma_\pm \delta_{k,m}).
\end{equation}
Let us now define the trilinear forms $\nabla^3 D_{max}(u)(w^1,w^2,w^3)$ and $\mathcal{E}_2(w^1,w^2,w^3)$ via
\begin{equation}
\begin{aligned}
\nabla^3D_{max}(u)(w^1,w^2,w^3) &:= \partial_i\partial_j\partial_k D_{max}(u) w_i^1w_j^2 w_k^3\\
%\quad \text{and} \quad 
\mathcal{E}_2(u)(w^1,w^2,w^3) &:= w^1_iw^2_l\partial_i\partial_j \eta(u^+)\partial_l\partial_k u^+_j (\partial_k f_m(u) - \sigma_\pm \delta_{k,m})w^3_m
\end{aligned}
\end{equation}
We note that while it is clear $\nabla^3 D_{max}$ is symmetric in all indices, this is not true for $\mathcal{E}_2$. Thus, using $|\nabla^2 u^+| \lesssim s_0^{-1}$ and $\lambda^1(u) - \sigma_\pm(u) \lesssim s_0$ by Lemma \ref{lem23}, $|\mathcal{E}_2(\cdot,\cdot, r_1(u))| \lesssim 1$. Therefore, exploiting the symmetry of $\nabla^3 D_{max}$ and using the decomposition (\ref{eqn_lem24_eqn2}),
\begin{equation}\label{eqn_lem24_eqn5}
|\nabla^3 D_{max}(u)(r_1(u),\cdot,\cdot)| = |\nabla^3 D_{max}(\cdot, r_1(u),\cdot) = |\nabla^3 D_{max}(u)(\cdot,\cdot,r_1(u))| \lesssim 1.
\end{equation}
Next, we bound $\mathcal{E}_2(r_c(u),\cdot,\cdot)$ for $2 \le c \le n$. 
First, differentiating (\ref{eqn_lem22_RK}) yields
\begin{equation}\label{eqn_lem24_eqn6}
\begin{aligned}
\left[f^\prime(u^+) - \sigma_\pm I\right]\nabla^2 u^+ &= \left[f^{\prime\prime}(u) - I \tens \sigma_\pm^\prime\right] +\left[\nabla u^+ \tens \sigma_\pm^+ + (u^+ - u) \tens \sigma_\pm^{\prime\prime}\right]\\
	&\quad-\left[f^{\prime\prime}(u^+)\nabla u^+ - I \tens \sigma_\pm^\prime\right]\nabla u^+.
\end{aligned}
\end{equation}
Note, by Lemma \ref{lem23} each term on the right hand side of (\ref{eqn_lem24_eqn6}) is bounded uniformly in $s_0$. Therefore, in index notation, we have
\begin{equation}\label{eqn_lem24_eqn7}
\left[\partial_j f^i(u^+) - \sigma_\pm \delta_{ij}\right]\partial_l\partial_k u^+_j = \bigO(1).
\end{equation}
Left multiplying (\ref{eqn_lem24_eqn7}) by $l^c(u^+)$ and using $\lambda^c(u^+) -\sigma_\pm \gtrsim 1$ for $1 < c \le n$, $|l^c(u^+) - l^c(u)|\lesssim s_0$ and $|\nabla^2u^+|\lesssim s_0^{-1}$, we obtain
\begin{equation}\label{eqn_lem24_eqn8}
l^c_j(u) \partial_l\partial_k u^+_j = l^c_j(u^+) \partial_l\partial_k u^+_j + \bigO(1) =  \bigO(1),
\end{equation}
for each $2 \le c \le n$. Using $\eta(u)r_c(u) \parallel l^c(u)$, (\ref{eqn_lem24_eqn8}) yields the bound on the trilinear form of 
\begin{equation}
|\mathcal{E}_2(u)(r_c(u),\cdot,\cdot)| = |r^c_i(u)\partial_i\partial_j \eta(u^+)\partial_l\partial_k u^+_j (\partial_k f_m(u) - \sigma_\pm \delta_{k,m})| \lesssim 1,
\end{equation}
for $1 < c \le n$, which by the decomposition (\ref{eqn_lem24_eqn2}) yields the corresponding bound $|\nabla^3D_{max}(u)(r_i(u),\cdot,\cdot)|\lesssim 1$ for $1 < i \le n$. By the linearity of $\nabla^3 D_{max}(u)(w^1,w^2,w^3)$ in $w^1$, it follows that for $w^1 = \sum w^1_i r_i(u)$,
\begin{equation}
|\nabla^3 D_{max}(u)(w^1,w^2,w^3)| \le \sum_{i = 1}^n |w^1_i\nabla^3 D_{max}(u)(r_i(u),w^2,w^3)| \lesssim |w^2||w^3|\sum_{i=1}^n |w^1_i| \lesssim |w^1||w^2||w^3|,
\end{equation}
which completes the proof.

\end{proof}

%%%%%%%%%%%%%%%%%%%%%%%%%%%%%%%%%%%%%%%%%%%%%%%%%%%%%%%%%%%%%%%%%%%%%%%%%%%%
\begin{flushleft}
\uline{\bf{Proof of Proposition \ref{prop_R0}}}
\end{flushleft}
%\begin{proof}[Proof of Proposition \ref{prop_R0}]
We show that $u_L$ is a local maximum for $D_{max}$. We first note that since $-\tilde\eta(u_L) = \eta(u_L | u_R)$
and $(u_R, u_L, \sigma_{LR})$ is an entropic $1$-shock, $u^+(u_L) = u_R$. In other words, $D_{max}(u_L) = D_{RH}(u_L,u_R,\sigma_{LR}) = 0$.
Similarly, evaluating the identity (\ref{eqn_lem22_RK}) at $u_L$,
\begin{equation}
\nabla D_{max}(u_L) = \left[\nabla\eta(u_R) - \nabla\eta(u_R) - (1+Cs_0)\left[\nabla\eta(u_L) - \nabla\eta(u_L)\right]\right]\left(f^\prime(u_L) - \sigma_{LR}I\right) = 0,
\end{equation}
yields that $u_L$ is a critical point of $D_{max}$.
Next, we evaluate $\nabla^2D_{max}$, computed in equation (\ref{eqn_lem24_eqn2}), at $u_L$ to obtain
\begin{equation}
\nabla^2 D_{max}(u_L) = \left[\nabla^2\eta(u_R)\nabla u^+ - (1+ Cs_0)\nabla^2\eta(u_L)\right](f^\prime(u_L) - \sigma_{LR} I)
\end{equation}
It remains to show that the matrix $\nabla^2 D_{max}(u_L)$ is (quantitatively) negative definite.
First, we show that $D_{max}$ is concave in the $r_1(u_L)$ direction. Indeed, we compute
\begin{equation}
\begin{aligned}\label{eqn_propR0_eqn1}
\nabla^2D_{max}(u_L)r_1(u_L) &= \left(\lambda^1(u_L) - \sigma_{LR}\right)\left[\nabla^2\eta(u_R)\nabla u^+ - (1+Cs_0)\nabla^2\eta(u_L)\right]r_1(u_L).
\end{aligned}
\end{equation}
Now, $\lambda^1(u_L) - \sigma_{LR} \sim s_0$ by Lemma \ref{lem23}. Also, evaluating (\ref{eqn_lem22_maximality}) at $u_L$ yields
\begin{equation}\label{eqn_propR0_eqn2}
s_0^{-1}(u_R - u_L)^t\nabla^2\eta(u_R)\nabla u^+(u_L) = 0.
\end{equation}
Therefore, substituting the Taylor expansion $u_R = u_L + s_0r_1(u_L) + \bigO(s_0^2)$ into (\ref{eqn_propR0_eqn2}) and right multiplying by $r_1(u_L)$,
\begin{equation}\label{eqn_propR0_eqn3}
\nabla^2 \eta(u_R)\nabla u^+(u_L)r_1(u_L) \cdot r_1(u_L) \lesssim s_0 |\nabla u^+(u_L)| \lesssim s_0.
\end{equation}
Combining (\ref{eqn_propR0_eqn1}) and (\ref{eqn_propR0_eqn3}) yields
\begin{equation}\label{eqn_propR0_eqn4}
\begin{aligned}
r_1(u_L)\cdot \nabla^2D_{max}(u_L)r_1(u_L) &= -s_0(1 + Cs_0) r_1(u_L)\cdot \left[\nabla^2\eta(u_L)\right]r_1(u_L) + \bigO(s_0^2) \lesssim -s_0,
\end{aligned}
\end{equation}
provided $s_0$ is sufficiently small.
Second, we show that $D_{max}$ is concave in the $r_i(u_L)$ directions for $1 < i \le n$. Indeed, we note that by left multiplying the identity (\ref{eqn_lem22_RK}) by $l^i(u_L)$,
\begin{equation}
(\lambda^i(u_L) - \sigma_{LR})l^i(u_L)\nabla u^+(u_L) =  (\lambda^i(u_L) - \sigma_{LR})l^i(u_L) + \left[(u_R - u_L) \cdot l^i(u_L)\right]\sigma_\pm^\prime(u_L) + \bigO(s_0)
\end{equation}
Since $|\nabla u^+|\lesssim 1$ and $|\sigma_\pm^\prime| \lesssim 1$ by Lemma \ref{lem23} and $\lambda^i(u) -\sigma_\pm \sim 1$ for $i \neq 1$ by Assumption \ref{assum} (a), we obtain
\begin{equation}
|l^i(u_L)\nabla u^+(u_L) - l^i(u_L)| \lesssim s_0 \qquad \text{for}\quad 1 < i \le n.
\end{equation}
Using $r_i(u_L)\nabla^2\eta(u_L) \parallel l^i(u_L)$, for any $1 \le i \le n$ and $1< j \le n$, we obtain the bound
\begin{equation}\label{eqn_propR0_eqn5}
\begin{aligned}
r_i(u_L)\cdot \nabla^2D_{max}(u_L)r_j(u_L) &= r_i(u_L)\cdot \left(\lambda^j(u_L) - \sigma_{LR}\right)\left[\nabla^2\eta(u_R)\nabla u^+ - (1+Cs_0)\nabla^2\eta(u_L)\right]r_j(u_L)\\
	&\lesssim -Cs_0r_i(u_L)\cdot \nabla^2\eta(u_L)r_j(u_L) + r_i(u_L) \cdot \left[\nabla^2\eta(u_R)\nabla u^+ - \nabla^2\eta(u_L)\right]r_j(u_L)\\
	&\lesssim -Cs_0\delta_{ij} + \bigO(s_0),
\end{aligned}
\end{equation}
for $C$ sufficiently large and $s_0$ sufficiently small.
To summarize, we have shown in (\ref{eqn_propR0_eqn1}), (\ref{eqn_propR0_eqn4}), and (\ref{eqn_propR0_eqn5}) that for any $2\le i \le n$ and $1 \le j,k \le n$ with $j\neq k$,
\begin{align}
\nabla^2 D_{max}(u_L)r_1(u_L) \cdot r_1(u_L) \lesssim -s_0\\
\nabla^2 D_{max}(u_L)r_i(u_L) \cdot r_i(u_L) \lesssim -Cs_0\\
\nabla^2 D_{max}(u_L)r_j(u_L) \cdot r_k(u_L) \lesssim s_0.
\end{align}
The same application of Cauchy-Schwarz inequality as in the proof of Lemma \ref{lem3} implies there exist universal constants $K_1,K_2 > 0$ depending only on the system such that for $C$ sufficiently large and $s_0$ sufficiently small and any $v = \sum_{i=1}^n v_i r_i(u_L)$,
\begin{equation}\label{eqn_propR0_eqn6}
\nabla^2 D_{max}(u_L)v \cdot v \le -K_1s_0v_1^2 - K_2Cs_0 \left[\sum_{i=2}^n v_i^2\right] \le -K_1s_0|v|^2,
\end{equation}
which implies $u_L$ is a local maximum of $D_{max}$. To construct $R_{C,s_0}^0$, we take $u\in R_{C,s_0}\backslash R_{C,s_0}^{bd}$ and suppose moreover that $K_3 > 0$ is the constant from Lemma \ref{lem24} so that $|\nabla^3D_{max}(u)| \le K_3$.
We estimate $D_{max}(u)$ using Taylor's theorem. In particular, there exists a $v$ between $u$ and $u_L$ such that
\begin{equation}
D_{max}(u) = \frac{1}{2}\nabla^2D_{max}(u_L)(u-u_L)\cdot (u-u_L) + \frac{1}{6}\nabla^3D_{max}(v)(v-u_L,v-u_L,v-u_L).
\end{equation}
Now, using (\ref{eqn_propR0_eqn6}),
\begin{equation}\label{eqn_propR0_eqn7}
D_{max}(u) \le -\frac{K_1}{2}s_0|u-u_L|^2 + \frac{K_3}{6}|u - u_L|^3,
\end{equation}
which directly implies there is a constant $K_0>0$ depending only on the system such that for all $C$ sufficiently large and $s_0$ sufficiently small, $D_{max}(u) \le 0$ provided $|u-u_L| \le K_0s_0$.
%\end{proof}

%%%%%%%%%%%%%%%%%%%%%%%%%%%%%%%%%%%%%%%%%%%%%%%%%%%%%%%%%%%%%%%%%%%%%%%%%%%%%%%%%%%%%%%%%%%%%%%%%%%%%%%%

\subsection{Step 2.3: Shocks in \texorpdfstring{$R_{C,s_0}^+$}{} and \texorpdfstring{$R_{C,s_0}^-$}{}}

In this step, we show that $D_{max}(u) \le 0$ for $u$ strictly between $R_{C,s_0}^{bd}$ and $R_{C,s_0}^0$ and for $u \in R_{C,s_0}$ but far enough in the interior of $\Pi$ such that $u\notin R_{C,s_0}^0$. This step is necessary because in Propositions \ref{prop_R}, \ref{prop_Rbd}, and \ref{prop_R0}, we must choose a specific universal constants $K_h$, $K_{bd}$, and $K_0$ to define the height in the $r_1(u^*)$ direction of the sets $R_{C,s_0}$, $R_{C,s_0}^0$, and $R_{C,s_0}^{bd}$, respectively. In particular, because we must choose $K_h$ sufficiently large and $K_{bd}$ and $K_0$ sufficiently small, it remains to check that if $R_{C,s_0} \neq R_{C,s_0}^{bd}\cup R_{C,s_0}^0$, the desired inequality, $D_{max} \le 0$, holds on the remainder of $R_{C,s_0}$. Since we know the inequality $D_{max}(u)\le 0$ holds on $\bd R_{C,s_0}$, $\bd R_{C,s_0}^{bd}$, and $\bd R_{C,s_0}^0$, it suffices to show there are no critical points for $D_{max}$ in $R_{C,s_0}$ except $u_L$, $u^*$, or possibly states $u$ very close to $u_L$ or $u^*$. More precisely, we have:
\begin{proposition}\label{prop_R+}
Let $R_{C,s_0}$ and $R_{C,s_0}^{bd}$ be as in the preceding sections. Then, for $C$ sufficiently large and $s_0$ sufficiently small, and for any $u \in R_{C,s_0} \backslash \left[R_{C,s_0}^{bd} \cup R_{C,s_0}^0\right]$, $\nabla D_{max}(u) \neq 0$.
\end{proposition}

In the following, we expand states $u$ about $u_L$ as
\begin{equation}\label{eqn_expansion}
u = u_L + \sum_{i=1}^n b_i(u) r_i(u_L).
\end{equation}
When we refer to $b_i$ below, we always intend $b_i$ to be interpreted as the $i$-th coefficient in this expansion.
Now, since $R$ has width $C^{-1/3}s_0$ in the $r_i(u^*) \approx r_i(u_L)$ directions for $1 < i \le n$ and $R^0$ is a ball of radius $\sim s_0$ about $u_L$ intersected with $R$, $R^0$ splits $R\backslash R^0$ into two connected components. Furthermore, the function $u\mapsto sgn(b_1(u))$ determines the component in which a state $u$ lies. Therefore, we define
\begin{equation}\label{defn_R+R-}
R_{C,s_0}^+ := \set{u \in R\backslash (R^0\cup R^{bd}) \ \big| \ b_1(u) > 0}\quad \text{and} \quad R_{C,s_0}^- := \set{u \in R\backslash (R^0\cup R^{bd}) \ \big| \ b_1(u) < 0}.
\end{equation}
Note, by the definition of $b_i$ in (\ref{eqn_expansion}), the sets $R_{C,s_0}^+$, $R_{C,s_0}^-$, $R_{C,s_0}^{bd}$, and $R_{C,s_0}^0$ are ordered as in Figure \ref{fig:small_scale}. We begin with the following characterization of critical points which is crucial to our analysis:
\begin{lemma}\label{lem25}
A state $u$ in the interior of $\Pi$ is a critical point of $D_{max}$ if and only if
\begin{equation}\label{eqn_lem25_desired}
\left[\nabla\eta(u^+) - \nabla\eta(u_R)\right] - (1+Cs_0)\left[\nabla\eta(u) - \nabla\eta(u_L)\right] = 0.
\end{equation}
\end{lemma}

\begin{proof}
By Lemma \ref{lem22}, for any state $u \in \Pi$,
\begin{equation}
\nabla D_{max}(u) = \left[\nabla\eta(u^+) - \nabla\eta(u_R)\right] - (1+Cs_0)\left[\nabla\eta(u) - \nabla\eta(u_L)\right]\left(f^\prime(u) - \sigma_{\pm}I\right).
\end{equation}
Clearly, if (\ref{eqn_lem25_desired}) holds, then $\nabla D_{max}(u) = 0$.
Conversely, if $\nabla D_{max}(u) = 0$, then since $u$ is in the interior of $\Pi$, $u^+ \neq u$ and $\lambda^1(u^+) < \sigma_\pm < \lambda^1(u)$ as $(u,u^+,\sigma_\pm)$ is Lax admissible. Therefore, $f^\prime(u) - \sigma_\pm I$ is invertible and we must have (\ref{eqn_lem25_desired}).
\end{proof}
The remainder of the step will be devoted to showing that for $u\in R_{C,s_0}$, (\ref{eqn_lem25_desired}) holds only if $u$ is quite close to $u_L$ or $\bd \Pi$. The next lemma expresses that by looking at $R_{C,s_0}$ for very large $C$, we have localized to a sufficiently narrow strip so that $u^+$, $u_R$, $u$, and $u_L$ are effectively colinear. Therefore, tracking how far a state $u$ is from satisfying (\ref{eqn_lem25_desired}) can be quantitatively reduced to studying a scalar ODE.
\begin{lemma}\label{lem26}
Let $u \in R_{C,s_0}\backslash R_{C,s_0}^{bd}$ with $b_1 = b_1(u)$ and define $u(t) = tu + (1-t)u_L$. Define the functions $F,G,e:[0,1]\rightarrow \R$ as
\begin{equation}\label{eqn_FandG_defn}
\begin{aligned}
F(t) &:= (1+Cs_0)\left[\nabla\eta(u(t)) - \nabla\eta(u_L)\right] \cdot sgn(b_1)r_1(u_L),\\
G(t) &:= \left[\nabla\eta(u^+(t)) - \nabla\eta(u_R)\right] \cdot sgn(b_1)r_1(u_L),\\
e(t) &:= \frac{|u^+(t) - u(t)|}{b_1},
\end{aligned}
\end{equation}
where $u^+(t)$ abbreviates $u^+(u(t))$. Then, $F$ and $G$ satisfy the ODE
\begin{equation}\label{eqn_FandGode}
\begin{aligned}
-e(t)G^\prime(t) + G(t)&= F(t) + \delta(t)\\
G(0) &= 0,
\end{aligned}
\end{equation}
where $\delta(t)$ is an error term satisfying the uniform bound
\begin{equation}\label{eqn_errorbound}
|\delta(t)| \lesssim \frac{C^{-1/3}s_0^2}{b_1} + s_0^2.
\end{equation}
\end{lemma} 

\begin{proof}
Evaluating (\ref{eqn_lem22_RK}) at $u(t)$ and right multiplying by $sgn(b_1)r_1(u_L)$, we obtain
\begin{equation}
G(t) - F(t) = sgn(b_1)(u^+(t) - u(t))^t\nabla^2\eta(u^+(t))\nabla u^+(t)r_1(u_L).
\end{equation}
Next, we compute $G^\prime(t)$ as
\begin{equation}
G^\prime(t) = sgn(b_1)\left[r_1(u_L)\right]^t\nabla^2\eta(u^+(t))\nabla u^+(t)(u - u_L),
\end{equation}
where we have used $\dot{u}(t) = u - u_L$. So $G$ solves the ODE
\begin{equation}
\begin{aligned}
G(t) - F(t) &= e(t)G^\prime(t) + \delta_1(t) + \delta_2(t)\\
G(0) &= 0,
\end{aligned}
\end{equation}
where $\delta_1$ and $\delta_2$ are error terms given as
\begin{equation}
\begin{aligned}
&\delta_1(t):= sgn(b_1)\left[u^+(t) - u(t) - |u^+(t) - u(t)|r_1(u_L)\right]^t\nabla^2\eta(u^+(t))\nabla u^+(t)r_1(u_L)\\
&\delta_2(t):= sgn(b_1)e(t)\left[r_1(u_L)\right]^t\nabla^2\eta(u^+(t))\nabla u^+(t)\left[b_1r_1(u_L) - (u - u_L)\right].
\end{aligned}
\end{equation}
It remains only to show that the error $\delta(t):= \delta_1(t) + \delta_2(t)$ satisfies (\ref{eqn_errorbound}). We estimate $\delta_1$ and $\delta_2$ separately. First, since $(u,u^+,\sigma_\pm)$ is an entropic $1$-shock, it satisfies the expansion
\begin{equation}
u^+(t) = u(t) + |u^+(t) - u(t)|r_1(u(t)) + \bigO(|u^+ - u|^2).
\end{equation}
Combined with the estimates $|\nabla u^+| \lesssim 1$ and $|u^+ - u|\sim s_0$ by Lemma \ref{lem23} and $u(t)\in R\backslash R^{bd}$ for each $t$, we obtain the uniform bound
\begin{equation}\label{eqn_lem26_eqn1}
|\delta_1(t)| \lesssim s_0|u^+(t) - u(t)| + |u^+(t) - u(t)|^2 \lesssim s_0^2.
\end{equation}
On the other hand, since $u \in R\backslash R^{bd}$, with $r_1(u_L)$ coefficient $b_1 = b_1(u)$, 
\begin{equation}
|b_1r_1(u_L) - (u - u_L)| \lesssim C^{-1/3}s_0.
\end{equation}
Again using $|\nabla u^+|\lesssim 1$ and $|u^+ - u| \lesssim s_0$ on $R\backslash R^{bd}$, we obtain the uniform bound
\begin{equation}\label{eqn_lem26_eqn2}
|\delta_2(t)| \lesssim C^{-1/3}s_0|e(t)| \lesssim \frac{C^{-1/3}s_0^2}{b_1}.
\end{equation}
Together (\ref{eqn_lem26_eqn1}) and (\ref{eqn_lem26_eqn2}) imply (\ref{eqn_errorbound}).
\end{proof}

\begin{lemma}\label{lem27}
For any $C$ sufficiently large and $s_0$ sufficiently small, for each $u\in R_{C,s_0}^+$, $u$ is not a critical point of $D_{max}$.
\end{lemma}

\begin{proof}
Fix $u \in R^+$ and define $u(t)$, $F(t)$, $G(t)$, $e(t)$, and $\delta(t)$ as in (\ref{eqn_FandG_defn}) so that $G$ solves the ODE (\ref{eqn_FandGode}) and $\delta$ satisfies the bound (\ref{eqn_errorbound}). Solving (\ref{eqn_FandGode}) by the integrating factor method, we have the explicit formula
\begin{equation}\label{eqn_lem27_eqn0}
G(t) = \frac{-1}{E(t)} \int_0^t \frac{E(s)}{e(s)}\left[F(s) + \delta(s)\right]\ ds, \quad \text{where }\quad E(t) = \exp\left(\int_0^t -e(s)^{-1} \ ds\right).
\end{equation}
We note that since $e(t) = \frac{|u^+(t) - u(t)|}{b_1} \sim 1$ for $u \in R_{C,s_0}^+$ and $t\in [0,1]$, $E(t) \sim 1$ over the same domain.
Next, by (\ref{eqn_FandG_defn}) and Lemma \ref{lem25}, $u$ is a critical point of $D_{max}$ if and only if $F(1) = G(1)$. However, $F(t) \ge 0$ for all $0 \le t \le 1$ and using that $e(t),E(t) \ge 0$, the preceding formula for $G$ yields $F(1) \ge 0 \ge G(1)$, provided the error term involving $\delta$ is sufficiently small relative to $F$. In particular, we will show $\delta$ is sufficiently small and $F(1)$ is sufficiently large, so that $u$ is not a critical point of $D_{max}$.

More formally, we estimate the error term involving $\delta$ using Lemma \ref{lem26}, $E(t) + e(t) \sim 1$, and $b_1 \sim s_0$ for $u\in R^+$:
\begin{equation}\label{eqn_lem27_eqn1}
\begin{aligned}
\left|\frac{-1}{E(t)} \int_0^t \frac{E(s)\delta(s)}{e(s)}\ ds\right| \lesssim \sup_{0 < s < 1}|\delta(s)| \lesssim C^{-1/3}s_0. %&\lesssim \frac{C^{-1/3}s_0^2 + b_1s_0^2}{E(t)b_1}\left(\int_0^t \frac{E(s)}{e(s)}\ ds\right)\\
	%&\lesssim \frac{C^{-1/3}s_0^2 + b_1s_0^2}{b_1E(t)}\left[E(0) - E(t)\right]\\
	%&\lesssim C^{-1/3}s_0 + s_0^2 \lesssim C^{-1/3}s_0.
\end{aligned}
\end{equation}
We also estimate the contribution of $F$ in the formula for $G$ using $E(t) \sim 1$ and the definition of $F$,
\begin{equation}\label{eqn_lem27_eqn2}
\begin{aligned}
\left|\frac{-1}{E(1)} \int_0^1 \frac{E(s)F(s)}{e(s)}\ ds\right| \gtrsim \int_0^1 F(s) \ ds &= \int_0^1 \left[\nabla\eta(u(s)) - \nabla\eta(u_L)\right] \cdot r_1(u_L)\ ds\\
	&= \int_0^1 sr_1(u_L) \cdot \nabla^2\eta(u(s))(u-u_L) + \bigO(ss_0^2) \ ds\\
	&= \int_0^1 sb_1 r_1(u_L) \cdot \nabla^2\eta(u_L)r_1(u_L) + \bigO(C^{-1/3}s_0s) \ ds\\
	&\sim s_0.
\end{aligned}
\end{equation}
By (\ref{eqn_lem27_eqn1}) and (\ref{eqn_lem27_eqn2}), there exist universal constants $K_1 > 0$ and $K_2 > 0$ such that 
\begin{equation}\label{eqn_lem27_eqn3}
\left|\frac{1}{E(t)} \int_0^t \frac{E(s)\delta(s)}{e(s)}\ ds\right| \le K_1C^{-1/3}s_0 \quad \text{and} \quad \frac{1}{E(t)} \int_0^t \frac{E(s)f(s)}{e(s)}\ ds \ge K_2s_0.
\end{equation}
Therefore, taking $C$ sufficiently large, depending only on $K_1$ and $K_2$, (\ref{eqn_lem27_eqn0}) and (\ref{eqn_lem27_eqn3}) imply $G(1) < 0 < F(1)$. In particular, $F(1) \neq G(1)$, and $u = u(1)$ cannot be a critical point of $D_{max}$ for $C$ sufficiently large and $s_0$ sufficiently small.
\end{proof}

\begin{lemma}\label{lem28}
For any $C$ sufficiently large and $s_0$ sufficiently small, for each $u\in R_{C,s_0}^-$, $u$ is not a critical point of $D_{max}$.
\end{lemma}

\begin{proof}
Fix $u\in R^-$ and define $u(t)$, $F(t)$, $G(t)$, $e(t)$, and $\delta(t)$ as in (\ref{eqn_FandG_defn}) so that $G$ satisfies the ODE (\ref{eqn_FandGode}) and $\delta$ satisfies the bound (\ref{eqn_errorbound}). Since $b_1 = b_1(u) < 0$, $e(t) < 0$, solving (\ref{eqn_FandGode}) yields the the explicit formula for $G$,
\begin{equation}\label{eqn_lem28_eqn1}
G(t) = \frac{1}{E(t)} \int_0^t \frac{E(s)}{|e(s)|}\left[F(s) + \delta(s)\right]\ ds, \quad \text{where }\quad E(t) = \exp\left(\int_0^t |e(s)|^{-1} \ ds\right).
\end{equation}
We note that since $e(t) = \frac{|u^+(t) - u(t)|}{b_1} \sim -1$ for $u \in R_{C,s_0}^-$ and $t\in [0,1]$, $E(t) \sim 1$ over the same domain.
As in the preceding section, we note by (\ref{eqn_FandG_defn}) Lemma \ref{lem26}, $u$ is a critical point of $D_{max}$ if and only if $F(1) = G(1)$. Moreover, we note that by (\ref{eqn_errorbound}),
\begin{equation}\label{eqn_lem28_eqn2}
\left|\frac{1}{E(t)} \int_0^t \frac{E(s)\delta(s)}{|e(s)|}\ ds\right| \lesssim C^{-1/3}s_0.
\end{equation}
Now, we estimate $F$ via
\begin{equation}\label{eqn_lem28_eqn3}
F^\prime(t) = -(1+Cs_0)r_1(u_L)\cdot\nabla^2\eta(u(t))(u - u_L) = r_1(u_L) \cdot \nabla^2\eta(u_L)r_1(u_L) + \bigO(s_0).
\end{equation}
Therefore, taking $s_0$ sufficiently small, $F$ is strictly increasing.
Next, we estimate $G(1)$ using (\ref{eqn_lem28_eqn1}), (\ref{eqn_lem28_eqn2}), and $F$ increasing, to obtain
\begin{equation}\label{eqn_lem28_eqn4}
G(1) \le \left[\frac{1}{E(1)} \int_0^1 \frac{E(s)}{|e(s)|}\ ds\right] F(1) + \bigO(C^{-1/3}s_0).
\end{equation}
Then, by the fundamental theorem of calculus, as $E^\prime(s) = \frac{E(s)}{|e(s)|}$,
\begin{equation}
\left[\frac{1}{E(1)} \int_0^1 \frac{E(s)}{|e(s)|}\ ds\right] = 1 - \frac{E(0)}{E(1)}.
\end{equation}
Let $\alpha = 1 - \frac{E(0)}{E(1)}$. Since $E(s) \ge 0$, $E(0) = 1$, and $E^\prime(s) \sim 1$, $0 < \alpha < K_1 < 1$, where $K_1$ is a universal constant. Therefore, (\ref{eqn_lem28_eqn4}) implies the inequality
\begin{equation}
G(1) \le K_1 F(1) + \bigO(C^{-1/3}s_0).
\end{equation}
Finally, taking $K_1 < K_2 < 1$ and $C$ sufficiently large and $s_0$ sufficiently small, we obtain $G(1) \le K_2 F(1) < F(1)$. Therefore, $G(1) \neq F(1)$ and $u = u(1)$ cannot be a critical point for $D_{max}$.
\end{proof}

\subsection{Proof of Proposition \ref{prop_shock}}

\begin{proof}[Proof of Proposition \ref{prop_shock}]
First, we fix $C$ sufficiently large and $s_0$ sufficiently small such that each of Proposition \ref{prop_Qs0}, Proposition \ref{prop_R}, Proposition \ref{prop_Rbd}, Proposition \ref{prop_R0}, and Proposition \ref{prop_R+} hold. Propositions \ref{prop_Qs0} and \ref{prop_R} guarantee we obtain a decomposition of $\Pi_{C,s_0}$ into
$$\Pi_{C,s_0} = R_{C,s_0} \cup \left[\Pi_{C,s_0}\backslash R_{C,s_0}\right],$$
where $D_{max}(u) \lesssim -s_0^3$ on $\Pi_{C,s_0}\backslash R_{C,s_0}$. Propositions \ref{prop_Rbd}, \ref{prop_R0}, and \ref{prop_R+} guarantee we obtain a decomposition of $R_{C,s_0}$ into
$$R_{C,s_0} = R_{C,s_0}^{bd} \cup R_{C,s_0}^+ \cup R_{C,s_0}^0 \cup R_{C,s_0}^-,$$
where $D_{max}(u) \lesssim -s_0^3$ on $\bd R_{C,s_0} \cup R_{C,s_0}^{bd} \cup \bd R_{C,s_0}^0$, $D_{max}(u) \le 0$ on $R_{C,s_0}^0$, and $\nabla D_{max} \neq 0$ on $R_{C,s_0}^+ \cup R_{C,s_0}^-$. As there are no critical points in $R_{C,s_0}^+ \cup R_{C,s_0}^-$, it follows that 
\begin{equation}
    \max_{u\in R^+ \cup R^-} D_{max}(u) = \max_{u\in \bd R^+ \cup \bd R^-} D_{max}(u) \lesssim -s_0^3.
\end{equation}
Since $u_L$ is a strict maximum of $D_{max}$ on $R_{C,s_0}^0$, we conclude $D_{max}(u) \le 0$ on $\Pi_{C,s_0}$ with $D_{max}(u) = 0$ if and only if $u = u_L$. Finally, recalling the definiton of $D_{max}$ in (\ref{eqn_Dmax_defn}), for any $(u_-, u_+, \sigma_\pm)$ an entropic $1$-shock with $u_-\in \Pi_{C,s_0}$,
\begin{equation}
D_{RH}(u_-,u_+,\sigma_\pm) \le D_{max}(u_-) \le 0.
\end{equation}
Moreover, equality holds only if $u_- = u_L$ and $u_+ = u_R$.
\end{proof}

%%%%%%%%%%%%%%%%%%%%%%%%%%%%%%%%%%%%%%%%%%%%%%%%%%%%%%%%%%%%%%%%%%%%%%%%%%%%%

\bibliographystyle{apalike}
\bibliography{references-2}

\end{document}